\documentclass[a4paper,11pt]{amsart}

\usepackage[english]{babel}
\usepackage[utf8]{inputenc}
\usepackage[T1]{fontenc}
\usepackage{amsmath}
\usepackage{amssymb}
\usepackage{multirow}
\usepackage{alltt}
\usepackage{graphicx}
\usepackage{bbm}
\usepackage{amsthm}
\usepackage{mathrsfs}
\usepackage{textcomp}
\usepackage{calc}
\usepackage{fullpage}
\usepackage[abbrev]{amsrefs}

\usepackage[usenames,dvipsnames]{xcolor}
\definecolor{darkblue}{RGB}{0,0,170}
\definecolor{brickred}{RGB}{200,0,0}
\usepackage[pdfauthor={Nicola Abatangelo}, hyperindex, linktocpage]{hyperref}
\hypersetup{colorlinks=true, linkcolor=darkblue, citecolor=brickred}
\usepackage{tikz}

\newcommand{\R}{\mathbb{R}}
\newcommand{\N}{\mathbb{N}}
\newcommand{\Sn}{{\mathbb{S}}^{n-1}}

\newcommand{\eps}{\varepsilon}
\newcommand{\dist}{\mathrm{dist}}

\renewcommand{\div}{\mathrm{div}}
\newcommand{\diam}{\text{diam}}
\newcommand{\pv}{\mathrm{p.v.}\!\!}
\newcommand{\Ds}{{\left(-\Delta\right)}^s}

\newcommand{\grad}{\nabla}

\newcommand{\intpart}[1]{\lfloor #1 \rfloor}
\newcommand{\lrangle}[1]{\langle #1 \rangle}

\newcommand{\Rn}{{\R}^n}

\theoremstyle{plain}
\newtheorem{theorem}{Theorem}[section]
\newtheorem{lemma}[theorem]{Lemma}

\newtheorem{proposition}[theorem]{Proposition}
\newtheorem{corollary}[theorem]{Corollary}

\theoremstyle{remark}
\newtheorem{remark}[theorem]{Remark}
\newtheorem{notation}[theorem]{Notation}

\theoremstyle{definition}
\newtheorem{definition}[theorem]{Definition}

\allowdisplaybreaks

\date{\today}
\address{Institut für Mathematik, Universität Zürich, Wintherturerstrasse 190, 8057 Zürich, Switzerland.}
\email{abatange@math.uni-frankfurt.de \ -- \ xavier.ros-oton@math.uzh.ch}
\author{Nicola Abatangelo}
\author{Xavier Ros-Oton}

\title{Obstacle problems for integro-differential operators: higher regularity of free boundaries}


\thanks{X.R. was supported by the European Research Council under the Grant Agreement No. 801867 ``Regularity and singularities in elliptic PDE (EllipticPDE)'', by the Swiss National Science Foundation, and by MINECO grant MTM2017-84214-C2-1-P}

\begin{document}

\begin{abstract}
We study the higher regularity of free boundaries in obstacle problems for integro-differential operators. 
Our main result establishes that, once free boundaries are~$C^{1,\alpha}$, then they are~$C^\infty$. 
This completes the study of regular points, initiated in~\cite{inventiones}.

In order to achieve this, we need to establish optimal boundary regularity estimates for solutions to linear nonlocal equations in~$C^{k,\alpha}$ domains. 
These new estimates are the core of our paper, and extend previously known results by Grubb (for~$k=\infty$) and by the second author and Serra (for~$k=1$).
\end{abstract}

\maketitle


\section{Introduction}

Obstacle problems for integro-differential operators appear naturally in Probability and Finance. 
Namely, they arise when considering optimal stopping problems for L\'evy processes with jumps, which have been used in pricing models for American options since the 1970s; see~\cites{CT,Merton}.
More recently, such kind of obstacle problems have found applications in interacting particle systems and other related models in statistical mechanics; 
see~\cites{CDM,Serfaty,Ros-survey} and references therein.

Because of their connections to Probability, Finance, and Physics, in the last fifteen years there have been considerable efforts to understand obstacle problems for such kind of nonlocal operators.
Usually, one considers the obstacle problem
\begin{equation}\label{obstacle}
\begin{split}\min\{Lv,\,v-\varphi\}&=0\quad\textrm{in}\ \R^n, \\
 \lim_{|x|\to\infty} v(x) &=0,\end{split}
 \end{equation}
for a nonlocal operator~$L$, where~$\varphi$ is a given smooth obstacle with compact support.

The most basic and canonical example of integro-differential operator~$L$ is the fractional Laplacian~$(-\Delta)^s$, $s\in(0,1)$.
The mathematical study of the obstacle problem for the fractional Laplacian was initiated by Silvestre~\cite{S-obst} and Caffarelli, Salsa, and Silvestre~\cite{CSS}, and it is nowadays pretty well understood; see the survey paper~\cite{DS-survey}. 

The main regularity result for the free boundary~$\partial\{v>\varphi\}$ in the obstacle problem for the fractional Laplacian establishes that the free boundary is~$C^\infty$ outside a certain set of \textit{degenerate} ---or \textit{singular}--- points.
To show this, one takes the following steps: 
\begin{itemize} 
\item[(a)] The free boundary splits into regular points and degenerate points.
\item[(b)] Near regular points, the free boundary is~$C^{1,\alpha}$.
\item[(c)] Once the free boundary is~$C^{1,\alpha}$ near regular points, then it is actually~$C^\infty$.
\end{itemize} 
Parts (a) and (b) were established in~\cite{CSS} (see also Athanasopoulos, Caffarelli, and Salsa~\cite{ACS}), while part (c) was established first for~$s=\frac12$ by Koch, Petrosyan, and Shi~\cite{KPS} and by De Silva and Savin~\cite{DS} (independently and with different proofs), and later for all~$s\in(0,1)$ by Koch, Rüland, and Shi~\cite{KRS} and by Jhaveri and Neumayer~\cite{JN} (independently and with different proofs).

After the results in~\cite{CSS}, many more results have been obtained concerning the set of degenerate/singular points~\cites{GP,BFR,FS,GR,FJ}, the case of the fractional Laplacian with a drift~\cites{PP,GPPS,FR}, and also the parabolic version of the problem~\cites{CF,BFR2}.

For more general integro-differential operators, however, much less is known.
One of the few works in this direction is the one by Caffarelli, the second author, and Serra~\cite{inventiones}, which extended the results of~\cite{CSS} to a whole family of integro-differential operators of the form
\begin{equation}\label{L}
\begin{split}
Lu(x) &=\, \pv\int_{\Rn}\big(u(x)-u(x+y)\big)K(y)\;dy\\
&=\frac12\int_{\Rn}\big(2u(x)-u(x+y)-u(x-y)\big)K(y)\;dy
\end{split}
\end{equation}
with the kernel~$K$ satisfying
\begin{align}\label{L2}
\begin{split}
& K\ \textrm{is even, homogeneous, and} \\ 
& \frac{\lambda}{|y|^{n+2s}}\leq K(y)\leq\frac{\Lambda}{|y|^{n+2s}},\quad
\text{for any }y\in\Rn,\ 
\textrm{with } 0<\lambda\leq \Lambda,\ s\in(0,1).
\end{split}
\end{align}
The main result in~\cite{inventiones} establishes that, if~$\varphi\in C^{2,1}(\R^n)$ and
\begin{align}\label{phi}
\{\varphi>0\} \text{ is bounded},
\end{align}
then the free boundary splits into \emph{regular points}~$x_0$, at which 
\begin{equation}\label{regular-point}
\sup_{B_r(x_0)}(v-\varphi)\approx r^{1+s} \quad \textrm{for $r>0$ small},
\end{equation}
and a set of degenerate points, at which~$\sup_{B_r(x_0)}(v-\varphi)\lesssim r^{1+s+\alpha}$, with~$\alpha>0$.
Moreover, the set of regular points is an open subset of the free boundary, and it is~$C^{1,\alpha}$.
\medskip

The aim of this paper is to continue the study of regular points initiated in~\cite{inventiones}, and to show that, once the free boundary is~$C^{1,\alpha}$ near regular points, then it is actually~$C^\infty$ (as long as~$\varphi$ is~$C^\infty$).
This is stated next.

\begin{theorem}\label{thm:Cinfty}
Let $L$ be an operator as in \eqref{L}-\eqref{L2}, 
with $K\in C^\infty(\Sn)$, and $v$ be any solution to \eqref{obstacle}
with $\varphi\in C^\infty(\R^n)$ satisfying \eqref{phi}.
Let $x_0\in\partial\{v>\varphi\}$ be any regular free boundary point.

Then, the free boundary is~$C^\infty$ in a neighbourhood of~$x_0$.
\end{theorem}

This is the analogue of step (c) explained above, and extends the results of~\cites{KRS,JN} to a much more general setting.

Furthermore, for less regular obstacles~$\varphi\in C^\beta$ we establish sharp regularity estimates for the free boundary, too.
Here, and throughout the paper, when~$\beta\notin\mathbb N$ we denote by~$C^\beta$ the space~$C^{k,\alpha}$, with~$k\in\mathbb Z$, $\alpha\in(0,1)$, and~$\beta=k+\alpha$.

\begin{theorem}\label{thm:low-reg}
Let~$L$ be an operator as in~\eqref{L}-\eqref{L2}, and~$v$ be any solution to~\eqref{obstacle}
with~$\varphi$ satisfying~\eqref{phi}.
Let~$\theta>2$ be such that~$\theta\notin \mathbb N$ and~$\theta\pm s\notin \mathbb N$.
Assume that~$\varphi\in C^{\theta+s}(\R^n)$, that~$K\in C^{2\theta-1}(\Sn)$, and let~$x_0\in \{v>\varphi\}$ be any regular free boundary point.

Then, the free boundary is~$C^{\theta}$ in a neighbourhood of~$x_0$.
\end{theorem}

This sharp estimate for non-$C^\infty$ obstacles seems to be new even for the fractional Laplacian~$(-\Delta)^s$: it was only known for~$s=\frac12$, see~\cite{KPS}.

\subsection{Strategy of the proof}

To establish Theorems~\ref{thm:Cinfty} and~\ref{thm:low-reg}, we need a very fine understanding of solutions to nonlocal equations in~$C^{k,\alpha}$ domains.
It was first observed by De Silva and Savin~\cite{desilva-savin} (in the context of the classical obstacle problem) that the higher regularity of free boundaries can be proved by ``simply'' having sharp estimates for harmonic functions in~$C^{k,\alpha}$ domains.
More precisely, they showed a higher order boundary Harnack estimate of the type:
\begin{equation}\label{harmonic}
\left. \begin{array}{c} 
u_1,u_2\ \textrm{harmonic in}\ \Omega\cap B_1 \\ 
u_1=u_2=0 \textrm{ in } \partial\Omega\cap B_1 \\
u_2>0\ \textrm{in}\ \Omega\cap B_1 \\ 
\partial\Omega\in C^\beta,\quad \beta\notin\mathbb Z 
\end{array}\right\}
\quad \Longrightarrow\quad \frac{u_1}{u_2}\in C^\beta(\overline\Omega\cap B_{1/2}).
\end{equation}
Notice that this is better than what Schauder estimates give.
Indeed, by boundary Schauder estimates, we have that~$u_1,u_2\in C^\beta(\overline\Omega\cap B_{1/2})$ and 
this yields that the quotient~$u_1/u_2$ is~$C^{\beta-1}$ up to the boundary\footnote{By the Hopf Lemma, $u_2$ is comparable to the distance to the boundary, therefore the division by $u_2$ roughly corresponds to taking one derivative on $u_1$.}.
The result summarized in \eqref{harmonic} shows that the regularity of the quotient~$u_1/u_2$ can be improved to~$C^\beta$. We refer to~\cite{desilva-savin} for more details about this proof in the case of the classical obstacle problem.

Once one has \eqref{harmonic}, then the idea is to take~$u_1,u_2$ to be two derivatives of a solution~$v$ to the obstacle problem, with~$\partial\Omega$ being the free boundary, and then deduce that
\begin{align*}
\partial\Omega\in C^{1,\alpha}\Longrightarrow \frac{u_1}{u_2}\in C^{1,\alpha}\Longrightarrow \partial\Omega\in C^{2,\alpha}\Longrightarrow \frac{u_1}{u_2}\in C^{2,\alpha}\Longrightarrow \partial\Omega\in C^{3,\alpha}\Longrightarrow ...\Longrightarrow \partial\Omega\in C^\infty
\end{align*}
---this is because the normal vector to~$\partial\Omega$ can be expressed in terms of derivatives of~$v$, see~\eqref{nu1}.
Such strategy was later extended in~\cite{JN} in order to show the higher regularity of free boundaries in the obstacle problem for the fractional Laplacian, and it is the same strategy that we use here in order to prove Theorems~\ref{thm:Cinfty} and~\ref{thm:low-reg}.

The main difficulty thus is to establish fine estimates for solutions in~$C^{k,\alpha}$ domains.
This is a highly nontrivial task in the context of nonlocal operators, and even the sharp boundary Schauder-type estimates in~$C^{k,\alpha}$ domains was a completely open problem for operators of the type~\eqref{L}-\eqref{L2}.
The only known results in this direction are due to the second author and Serra~\cites{annali,duke,jde,JMPA} for~$k=1$, or to Grubb~\cites{Grubb,Grubb2} for~$k=\infty$, and are actually very delicate to establish. 

In case of the fractional Laplacian~$(-\Delta)^s$ there is an extra tool that one can use: the extension problem of Caffarelli and Silvestre~\cite{caffarelli-silvestre}.
Thanks to this, \cite{JN} established the necessary Schauder-type and higher order boundary Harnack estimates for the fractional Laplacian in~$C^{k,\alpha}$ domains.
Unfortunately, such extension technique is not available for more general nonlocal operators~\eqref{L}-\eqref{L2}, and thus our proofs must be completely independent from those in~\cite{JN}.

\subsection{Fine estimates for nonlocal operators in \texorpdfstring{$C^{k,\alpha}$}{Ck,a} domains}

We show the following generalization of~\eqref{harmonic} to nonlocal elliptic operators of the type~\eqref{L}-\eqref{L2}.
We remark that this is the first higher order boundary Harnack estimate for general nonlocal operators, and it even refines the estimates from~\cite{JN} for the fractional Laplacian.

\begin{theorem}\label{thm:u1overu2}
Let~$\beta>1$ be such that~$\beta\not\in\N$, $\beta\pm s\not\in\N$.
Let~$L$ be an operator as in~\eqref{L}-\eqref{L2}, with~$K\in C^{2\beta+1}(\Sn)$.
Let~$\Omega\subseteq\Rn$ be any bounded~$C^{\beta}$ domain
and~$u_1,u_2\in L^\infty(\Rn)$ be solutions of 
\begin{align*}
\left\lbrace\begin{aligned}
Lu_i &= f_i & & \hbox{in }\Omega\cap B_1 \\
u_i &= 0 & & \hbox{in }B_1\setminus\Omega,
\end{aligned}\right.
\end{align*}
with~$f_1,f_2\in C^{\beta-s}(\overline\Omega)$, $u_2\geq c_1d^s$ in~$B_1$ for some $c_1>0$, and~$\|f_2\|_{C^{\beta-s}(\overline{\Omega})}+\|u_2\|_{L^\infty(\Rn)} \leq C_2$.

Then,
\begin{align*}
\Big\|\frac{u_1}{u_2}\Big\|_{C^\beta(\overline\Omega\cap B_1)} \leq C\big(\|f_1\|_{C^{\beta-s}(\overline{\Omega})}+
\|u_1\|_{L^\infty(\Rn)}\big)
\end{align*}
for some~$C>0$ depending only on~$n,\ s,\ \beta,\ c_1,\ C_2,\ \Omega,\ \lambda,\ \Lambda,$ and~$\|K\|_{C^{2\beta+1}(\Sn)}$.
\end{theorem}

Here, and throughout the paper, $d$ denotes a regularized version of the distance to the boundary function, see Definition~\ref{defi-d}.

An important step towards the proof of Theorem~\ref{thm:u1overu2} is the following boundary Schauder-type estimate for solutions to nonlocal elliptic equations in~$C^{k,\alpha}$ domains.

\begin{theorem}\label{thm:uoverds}
Let~$\beta>s$ be such that~$\beta\not\in\N$, $\beta\pm s\not\in\N$.
Let~$L$ be an operator as in~\eqref{L}-\eqref{L2}, with $K\in C^{2\beta+3}(\Sn)$.
Let~$\Omega\subseteq\Rn$ be any bounded~$C^{\beta+1}$ domain,
and~$u\in L^\infty(\Rn)$ be any solution of 
\begin{align*}
\left\lbrace\begin{aligned}
Lu &= f & & \hbox{in }\Omega\cap B_1 \\
u &= 0 & & \hbox{in }B_1\setminus\Omega
\end{aligned}\right.
\end{align*}
with $f\in C^{\beta-s}(\overline\Omega)$.
Then
\begin{align*}
\Big\|\frac{u}{d^s}\Big\|_{C^\beta(\overline\Omega\cap B_1)} \leq C\big(\|f\|_{C^{\beta-s}(\overline{\Omega})}+\|u\|_{L^\infty(\Rn)}\big)
\end{align*}
for some $C>0$ depending only on~$n,\ s,\ \beta,\ \Omega,\ \lambda,\ \Lambda,$ and $\|K\|_{C^{2\beta+3}(\Sn)}$.
\end{theorem}

This extends for the first time to all~$k\in \mathbb N$ the results for~$k=\infty$~\cites{Grubb,Grubb2}, and those for~$k=1$~\cites{annali,duke,jde,JMPA}.
Thus, our result completely settles the open question of establishing boundary Schauder estimates for nonlocal operators of the form~\eqref{L}-\eqref{L2} in~$C^{k,\alpha}$ domains.

\subsection{On the proofs of Theorems \ref{thm:u1overu2} and \ref{thm:uoverds}}

In order to establish our new fine estimates for nonlocal equations in~$C^{k,\alpha}$ domains, we develop a new, higher order version of the blow-up and compactness technique from~\cite{duke}.
This remained as an open problem after the results of~\cite{duke} mainly because of two reasons.

First, because the functions would grow too much at infinity whenever we want a higher order estimate, and thus one must be very careful when taking limits and giving a meaning to the limiting equation.

Second, because of a technical problem involving the 
function $d^s$: one needs to show a result of the type
\begin{equation}\label{computation-Lds}
\partial\Omega\in C^\beta\quad\Longrightarrow\quad L(d^s)\in C^{\beta-1-s}(\overline\Omega).
\end{equation}
This was one of the results that had to be proved in \cite{duke}; however the proof given therein only gave that $L(d^s)\in C^s(\overline\Omega)$ (and actually under a non-sharp assumption of the domain).
To show that $L(d^s)$ is more regular than $C^s$ (in $C^{k,\alpha}$ domains) remained as an open problem after the results of~\cite{duke}.

We solve the first technical difficulty here by using some ideas by Dipierro, Savin, and Valdinoci~\cite{growth}; notice however that our proofs are completely independent from those in~\cite{growth}, and we moreover show some new results concerning nonlocal operators for functions with polynomial growth.
We think that some of these results (proved in Section~\ref{sec:polynom}) could be of independent interest.

Concerning the second key difficulty, we provide here a complete understanding of the regularity of~$L(d^s)$ in terms of the regularity of~$\partial\Omega$, proving~\eqref{computation-Lds} for the first time.
This answers the open question left in \cite{duke} and it allows us to proceed with the higher order blow-up and compactness technique to show Theorem~\ref{thm:uoverds}.
The proof of~\eqref{computation-Lds} is extremely technical.
Moreover, it is not simply a tedious computation but it requires several new ideas concerning nonlocal operators with homogeneous kernels~\eqref{L}-\eqref{L2}.
On top of that, there are various essential cancellations without which~\eqref{computation-Lds} would not hold.

Additionally, in order to prove Theorem~\ref{thm:u1overu2}, we need to establish a result in the spirit of~\eqref{computation-Lds} but for~$L(\eta d^s)$, $\eta\in C^\infty$, with an extra cancellation taking place in case that $\eta$ vanishes at a boundary point.
All this is done in Section~\ref{sec:Lds}, and we believe this to be an important contribution of this paper.

Finally, it is important to notice that the development of the new techniques in this paper (\textit{i.e.}, the higher order version of the blow-up technique from~\cite{duke}, and the proof of~\eqref{computation-Lds}) open the road to the study of the higher regularity of free boundaries in other obstacle problems that until now seemed out of reach, such as nonlocal operators with drift~\cites{PP,GPPS,FR}, or even the parabolic obstacle problem for the fractional Laplacian~\cites{CF,BFR2}.

\subsection{Organization of the paper}

The paper is organized as follows.
Section \ref{sec:Lds} is devoted to proving \eqref{computation-Lds} and related estimates. 
Section \ref{sec:polynom} deals with an extension of the definition of $L$
to include its evaluation on functions growing polynomially at infinity: beside the definition itself,
we are going to provide with interior and boundary regularity estimates, Liouville-type theorems, and some other technical details.
Section \ref{sec:uover} contains the proofs of Theorems \ref{thm:uoverds} and \ref{thm:u1overu2}.
Section \ref{sec:final} proves Theorems \ref{thm:Cinfty} and \ref{thm:low-reg}, and it concludes the paper.
We also attach in an appendix some small details and remarks that we need in the proofs, to lighten these up.

\subsection{Notations}
As already mentioned above, when $\beta\notin\mathbb N$ we use the single index notation $C^\beta$ for the Hölder spaces:
this corresponds to $C^{\intpart\beta,\beta-\intpart\beta}$ where $\intpart\cdot$ denotes the integer part of a positive real number.

Throughout the paper, we will denote 
$\lrangle{w}=w/|w|,\ w\in\Rn$.
Also, we will make extensive use of multi-indices $\alpha\in\N^n$, $\alpha=(\alpha_1,\ldots,\alpha_n)$, $|\alpha|=\alpha_1+\ldots+\alpha_n$: these will be mainly used 
to shorten higher order derivatives in the following way
\begin{align*}
\partial^\alpha = \Big(\frac{\partial}{\partial x_1}\Big)^{\alpha_1} \circ \ldots \circ \Big(\frac{\partial}{\partial x_n}\Big)^{\alpha_n}.
\end{align*}
As to other notations for derivatives, $\nabla$ will denote the gradient as customary. 
Instead, $D^k$, $k\in\N$, will be the full $k$-linear operator entailed by all possible derivatives of order $k$:
in this spirit, we also have
\begin{align*}
D^k=\big(\partial^\alpha\big)_{|\alpha|=k}.
\end{align*}

By $\mathbf{P}_k$ we mean the space of polynomials of order $k$: mind that we allow ourselves to avoid specifying the number of variables, 
as there will be never confusion to this regard. The coefficients of the polynomials will be identified as 
\begin{align*}
Q\in\mathbf{P}_k \qquad \Longrightarrow \qquad Q(x)=\sum_{\alpha\in\N^n,|\alpha|\leq k}q^{(\alpha)}x^\alpha,\quad x\in\Rn.
\end{align*}

Finally, as it often happens, $C$ will indicate an unspecified constant not depending on any of the relevant quantities,
and whose value will be allowed to change from line to line. We will make use of sub-indices whenever we will want to underline
the hidden dependencies of the constant.

\section{Nonlocal operators and the distance function}\label{sec:Lds}

The goal of this section is to prove \eqref{computation-Lds} and other related estimates for the distance function~$d^s$.

\subsection{A regularized distance}
Actually, we need $d$ to be more regular in the interior of $\Omega$ than just the distance function.
For this reason, we need the following.

\begin{definition}\label{defi-d}
Let $\Omega\subset\R^n$ be an open set with $C^\beta$ boundary.
We denote by $d\in C^\infty(\Omega)\cap C^\beta(\overline\Omega)$
a function satisfying
\begin{align*}
\frac1C\dist(\,\cdot\,,\Omega^c)\leq d\leq C\dist(\,\cdot\,,\Omega^c),\qquad
|D^jd|\leq C_jd^{\beta-j},\qquad
\text{ for all }j>\beta \text{ and some }C,C_j>0.
\end{align*}
The construction of such $d$ is provided in Lemma \ref{lem:regularized distance}.
\end{definition}

We aim at proving the following.

\begin{theorem}\label{thm:Lds}
Let $K$ be a kernel as in \eqref{L2}.
Let $\Omega\subseteq\Rn$ be a domain such that $0\in\partial\Omega$ and $\partial\Omega\cap B_1\in C^\beta$,
for some $\beta>1+s$, $\beta-s\not\in\N$, and assume $K\in C^{2\beta+1}(\Sn)$. 
Let $\psi\in C^{\beta-1}(\overline{B}_1)\cap C^\infty(\Omega\cap B_1)$ be given,
and let $d$ be given by Definition \ref{defi-d}. Assume
\[
\|\psi\|_{C^j(\{{\rm dist}(\,\cdot\, ,\Omega^c)>r\}\cap B_1)} \leq C_\star  r^{\beta-1-j}\quad\textrm{for all}\qquad j>\beta-1.
\]
Then the function defined by
\begin{align*}
L_\psi(d^s)(x)\ :=\ \pv\int_{B_1} \nabla(d^s)(y)\;K(y-x)\;\psi(y)\cdot (y-x)\;dy
\end{align*}
is of class $C^{\beta-1-s}$ in $B_{1/2}$ with
\begin{align*}
\big|D^j L_\psi(d^s)(x)\big|\leq C_j \, \big(|\psi(0)|+|x|\big) \, d(x)^{\beta-1-s-j}
\qquad\text{in }B_{1/2},\ \text{for any }j\in\N,\ \beta-1-s<j<\beta,
\end{align*}
for some $C_j$ depending only on $j,\ n,\ s,\ \beta,\ C_\star,\ \Omega,\ \lambda,\ \Lambda,$ and $\|K\|_{C^{2\beta+1}(\Sn)}$.
\end{theorem}

Before turning to its proof, we first give the following consequence, which implies  \eqref{computation-Lds}.

\begin{corollary}\label{cor:Lds}
Let $K$ be a kernel as in \eqref{L2}.
Let $\Omega\subseteq\Rn$ be a domain such that $0\in\partial\Omega$ and $\partial\Omega\cap B_1\in C^\beta$, for some $\beta>1+s$, $\beta-s\not\in\N$, and assume $K\in C^{2\beta+1}(\Sn)$. 
Let $\eta\in C^\infty(\Rn)$ be given, and let $d$ be given by Definition \ref{defi-d}. 

Then, $L(\eta d^s)\in C^{\beta-1-s}(\overline \Omega\cap B_{1/2})$, with 
\[\|L(\eta d^s)\|_{C^{\beta-1-s}(\overline \Omega\cap B_{1/2})} \leq C.\]
Moreover, for every $j\in\N$, $\beta-1-s<j<\beta$, we have
\begin{align*}
\big|D^j L(\eta d^s)(x)\big|\leq C_j \, \big(|\eta(0)|+|x|\big) \, d(x)^{\beta-1-s-j}
\qquad\text{in }\,\overline\Omega\cap B_{1/2},
\end{align*}
with $C$ and $C_j$ depending only on $j,\ n,\ s,\ \beta,\ \Omega,\ \lambda,\ \Lambda,$ and $\|K\|_{C^{2\beta+1}(\Sn)}$.
\end{corollary}

We start by proving some preliminary lemmas.

\begin{lemma}\label{lem:representation of L}
Let $L$ be an operator as in \eqref{L}-\eqref{L2} and $u\in W^{1,1}_{loc}(\Rn)$ be such that
\begin{align*}
\int_{\Rn}\frac{|\nabla u(y)|}{1+|y|^{n+2s-1}}\;dy<\infty.
\end{align*}
Then,
\begin{equation}
\begin{split}
Lu(x)& =-\frac1{2s}\;\pv\int_{\Rn}\nabla u(x+y)\cdot y\,K(y)\;dy\\
&= -\frac1{2s}\;\pv\int_{\Rn}\nabla u(y)\cdot (y-x)\,K(y-x)\;dy
\qquad x\in\Rn.
\end{split}
\end{equation}
\end{lemma}

\begin{proof}
Since $K$ is homogeneous, it follows from
\begin{align*}
\div\big(y K(y)\big)=nK(y)+y\cdot\nabla K(y)=nK(y)-(n+2s)K(y)=-2s\,K(y)
\end{align*}
and an integration by parts:
\begin{align*}
& \frac12\int_{\Rn}\big(2u(x)-u(x+y)-u(x-y)\big)K(y)\;dy \ = \\
& = -\frac1{4s}\int_{\Rn}\big(2u(x)-u(x+y)-u(x-y)\big)\div\big(y\,K(y)\big)\;dy \\
& = \frac1{4s}\int_{\Rn}\nabla_y\big(2u(x)-u(x+y)-u(x-y)\big)\cdot y\,K(y)\;dy \\
& = \frac1{4s}\int_{\Rn}\big(-\nabla u(x+y)+\nabla u(x-y)\big)\cdot y\,K(y)\;dy 
  = -\frac1{2s}\;\pv\int_{\Rn}\nabla u(x+y)\cdot y\,K(y)\;dy.
\end{align*}
\end{proof}

We next show how Corollary \ref{cor:Lds} follows from Theorem \ref{thm:Lds}.

\begin{proof}[Proof of Corollary \ref{cor:Lds}]
If, starting from Lemma \ref{lem:representation of L}, 
we take another step in the representation of $L(\eta d^s)$
by means of the product rule, we obtain
\begin{align*}
L(\eta d^s)(x) = & -\frac1{2s}\pv\int_{\Rn}\nabla(d^s)(y)\cdot (y-x)\,K(y-x)\;\eta(y)\;dy \\
& -\frac1{2s}\pv\int_{\Rn}d(y)^s\nabla\eta(y)\cdot (y-x)\,K(y-x)\;dy \\
= & \ \pv\int_{\Rn}d(y)^{s-1}\,K(y-x)\;\psi(y)\cdot (y-x)\;dy
\end{align*}
where we have denoted
\begin{align*}
\psi:=-\frac12\eta\nabla d-\frac1{2s}\, d\; \nabla\eta,
\qquad \text{in }\Rn.
\end{align*}
Notice that the regularity of $\psi$ is inherited by that of $d$ and $\eta$.

Since the function of $x$
\[\pv\int_{\Rn\setminus B_1}d(y)^{s-1}\,K(y-x)\;\psi(y)\cdot (y-x)\;dy\]
is $C^{2\beta+1}$ in $B_{1/2}$ (notice that the dependence on $x$ is only on the kernel $K$, which is $C^{2\beta+1}$ outside the origin and it is integrated in a region that does not contain the origin), then the result follows from Theorem~\ref{thm:Lds}.
\end{proof}

The rest of this section is devoted to the proof of Theorem \ref{thm:Lds}.
For this, we need several tools.

\subsection{Flattening of the boundary}
The first step is to flatten the boundary $\partial\Omega$ around $0\in\partial\Omega$.
Without loss of generality we can suppose the following facts:
\begin{itemize}
\item There exists a $C^\beta$-diffeomorphism $\phi:B_1\to B_1$ such that $\phi(0)=0$,
$\phi(B_1\cap\{z_n=0\})=B_1\cap\partial\Omega$, and $(z_n)_+=d(\phi(z))$,
\textit{i.e.} we do not need to rescale $\Omega$ for $\partial\Omega\cap B_1$
to be flattened via a single diffeomorphism; note that relation $z_n=d(\phi(z))$
in turn implies
\begin{align*}
\delta_{j,n}=\grad d(\phi(z))\partial_j\phi(z)
\quad\text{and therefore}\quad
\grad d(\phi(z))=\partial_n\phi(z);
\end{align*}
\item $\partial\Omega$ is flat outside $B_1$, so that $\phi$ can be extended to a global $C^\beta$-diffeomorphism 
$\phi:\Rn\to\Rn$ which coincides with the identity outside $B_1$.
\item $\phi\in C^\infty(B_1\cap\{z_n>0\})$ with   
\begin{align}\label{diffeo}
|D^j\phi| \leq C_j d^{\beta-j},
\qquad \text{in }B_1\cap\{z_n>0\}, \text{ for }j\in\N,\ j>\beta.
\end{align}
\end{itemize}
The construction of $\phi$ is provided in Lemma \ref{lem:diffeo}.

\begin{remark}
As seen in the proof of Corollary \ref{cor:Lds}, by splitting
\begin{align*}
L(\eta d^s)(x)  =\ &\pv\int_{B_1}d(y)^{s-1}\,K(y-x)\;\psi(y)\cdot (y-x)\;dy\,+\\
&\qquad\qquad+\int_{\Rn\setminus B_1}d(y)^{s-1}\,K(y-x)\;\psi(y)\cdot (y-x)\;dy,
\end{align*}
it is clear that we can limit our analysis to the first integral,
as the second one is returning a function as smooth as the kernel. 
For this reason,
from now on we only deal with
\begin{align}\label{I_1}
\pv\int_{B_1}d(y)^{s-1}\,K(y-x)\;\psi(y)\cdot (y-x)\;dy
\end{align}
by taking advantage of the above diffeomorphism.
\end{remark}

\subsection{Yet another representation for \texorpdfstring{$L$}{L}}

With the change of variables $\phi(z)=y$ and $\phi(\hat x)=x$ we get
\begin{multline*}
\pv\int_{B_1}d(y)^{s-1}\,K(y-x)\;\psi(y)\cdot (y-x)\;dy \ = \\
=\ \pv\int_{B_1}(z_n)_+^{s-1}K\big(\phi(z)-\phi(\hat x)\big)\;\psi(\phi(z))\cdot
\big(\phi(z)-\phi(\hat x)\big)\;\big|\!\det D\phi(z)\big|\;dz.
\end{multline*}
Let us define
\[J(w)  := K(w) w 			\label{J kernel}\]
and
\[\rho(z)  = \psi(\phi(z))\,\big|\!\det D\phi(z)\big|		\label{new jacobian}\]
in order to write
\begin{align*}
\pv\int_{B_1}d(y)^{s-1}\,K(y-x)\;\psi(y)\cdot (y-x)\;dy = 
\pv\int_{B_1}(z_n)_+^{s-1}J(\phi(z)-\phi(\hat x))\cdot\rho(z)\;dz
\end{align*}
and let us define
\begin{align}\label{operative}
I(\hat x):=\pv\int_{B_1}(z_n)_+^{s-1}J\big(\phi(z)-\phi(\hat x)\big)\cdot\rho(z)\;dz.
\end{align}

\begin{remark}\label{rmk:regularity of J}
For further reference, let us state here the regularity of the functions involved here.
The kernel $J$ is still homogeneous and it inherits the regularity of $K$ far from the origin.
Moreover, $J$ is odd (since $K$ is even) and 
\begin{align*}
J(w)=\frac{K(\lrangle{w})\lrangle{w}}{|w|^{n+2s-1}}=\frac{J(\lrangle{w})}{|w|^{n+2s-1}},
\qquad w\in\Rn\setminus\{0\}.
\end{align*}
On the other hand, $\rho\in C^{\beta-1}(\overline{B}_1)\cap C^\infty(\{z_n>0\}\cap B_1)$, with the corresponding interior bounds inherited from $\psi$ and $\phi$.
\end{remark}

\subsection{A supplementary variable}

In order to continue with the argument, we decouple
the dependence on~$\hat x$,  a trick that will be functional in the rest of the analysis.
Fix $r\in(0,1)$ and $p\in B_1$ such that $d(p)=2r$.
We set, 
\begin{align}\label{operative-splitted}
I(\hat x)=I_1(\hat x)+I_r(\hat x,\hat x)
\end{align}
where 
\begin{align*}
I_1(\hat x) &= \int_{B_1\setminus B_r(p)}(z_n)_+^{s-1}J\big(\phi(z)-\phi(\hat x)\big)\cdot\rho(z)\;dz, \\
I_r(\xi,\hat x) &= \pv\int_{B_r(p)}(z_n)_+^{s-1}J\big(\phi(\xi + z-\hat x)-\phi(\hat x)\big)\cdot\rho(z)\;dz.
\end{align*}
Notice that in $B_r(p)$ the function $\phi$ is $C^\infty$, while in $B_1\setminus B_r(p)$ it is only $C^{\beta-1}$.

The reader should be warned that, despite the splitting described above, each of the two integrals separately does not satisfy the bounds we want to prove, but they need to be combined again to prove the regularity of \eqref{operative}: one key step is the cancellation taking place in \eqref{radial claim}.

\subsection{Expansion of the kernel}

We are now going to Taylor-expand the function $J(\phi(\xi+z-\hat x)-\phi(\xi))$
around the point $\xi\in B_1$, using $z-\hat x$ as an increment: according to the order of the expansion 
we need, the size of $z-\hat x$ will be suitably chosen\footnote{Namely,
when we will expand to some order less than $\beta$ then $z-\hat x$ will be 
allowed to be arbitrarily large (because we have global $C^\beta$ regularity);
instead, when we will expand to order larger than $\beta$, we will restrict $z-\hat x$
to a small ball of radius $r$ in order to have $\xi_n+z_n-\hat x_n>0$.}.
In view of the regularity of $\phi$, for any $j\in\{1,\ldots,\intpart\beta\}$, we can write
\begin{align}\label{phi-expansion}
\phi(\xi+z-\hat x)-\phi(\xi)=\sum_{1\leq|\alpha|\leq j}
\partial^\alpha\phi(\xi)(z-\hat x)^\alpha+e_j(\xi,z-\hat x)
\end{align}
for some $e_j:B_1\times B_2\to \R$ which is uniformly $C^{\beta-j}$ in the first variable and 
uniformly $C^\beta$ in the second one, which moreover satisfies 
(as a consequence of \eqref{diffeo})
\begin{align*}
\big|e_j(\xi,z-\hat x)\big|\ \leq\ \left\lbrace\begin{aligned}
& |z-\hat x|^{j+1}					& & \text{ if }|z-\hat x|<1,\ j\leq\intpart\beta-1, \\
& |z-\hat x|^{\beta} 					& & \text{ if }|z-\hat x|<1,\ j=\intpart\beta, \\
& r^{\beta-j-1}|z-\hat x|^{j+1} 		& & \text{ if }|z-\hat x|<r,\ j\geq\intpart\beta+1.
\end{aligned}\right.
\end{align*}

Using \eqref{phi-expansion}, we deduce
\begin{align*}
& J\big(\phi(\xi+z-\hat x)-\phi(\xi)\big) = J\bigg(
\sum_{1\leq|\alpha|\leq j}|z-\overline{x}|^{|\alpha|}
\partial^\alpha\phi(\xi)\lrangle{z-\hat x}^\alpha+e_j(\xi,z-\hat x) 
\bigg) \\
& \ = |z-\hat x|^{-n-2s+1}\ J\bigg(
\sum_{1\leq|\alpha|\leq j}|z-\hat x|^{|\alpha|-1}
\partial^\alpha\phi(\xi)\lrangle{z-\hat x}^\alpha+|z-\hat x|^{-1}e_j(\xi,z-\hat x) 
\bigg)
\end{align*} 
where we have used the homogeneity of kernel $K$.
We further expand the last obtained quantity, this time by taking advantage of the regularity of $K$.
In particular, we expand around the point 
\begin{align*}
\sum_{|\alpha|=1} \partial^\alpha\phi(\xi)\,\lrangle{z-\hat x}^\alpha=D\phi(\xi)\lrangle{z-\hat x},
\end{align*}
deducing (using again the multi-index notation)
\begin{equation}\label{J-expansion-first}
\begin{aligned}
 &\hspace{-10mm} J\bigg(
\sum_{1\leq|\alpha|\leq j}|z-\hat x|^{|\alpha|-1}
\partial^\alpha\phi(\xi)\lrangle{z-\hat x}^\alpha+|z-\hat x|^{-1}e_j(\xi,z-\hat x) 
\bigg)= \\
& \qquad= J\big(D\phi(\xi)\lrangle{z-\hat x}\big)+E(\xi,z-\hat x)\ + \\
&  \hspace{-3mm}+\sum_{1\leq|\gamma|\leq j-1}\!\!\!\!\!\!\partial^\gamma J\big(D\phi(\xi)\lrangle{z-\hat x}\big)
\Bigg(\sum_{2\leq|\alpha|\leq j}\!\!\!|z-\hat x|^{|\alpha|-1}
\partial^\alpha\phi(\xi)\lrangle{z-\hat x}^\alpha+|z-\hat x|^{-1}e_j(\xi,z-\hat x) \Bigg)^\gamma \hspace{-3mm}
\end{aligned}
\end{equation}
and, after having grouped together the terms with the same homogeneity in $|z-\hat x|$,
\begin{align}\label{J-expansion}
J\big(\phi(\xi+z-\hat x)-\phi(\xi)\big) & = \sum_{i=0}^{j-1}
\frac{b_i\big(\xi,\lrangle{z-\hat x}\big)}{{|z-\hat x|}^{n+2s-i-1}}
+\frac{R_j(\xi,z-\hat x)}{{|z-\hat x|}^{n+2s-1}},
\end{align}
with
\begin{align*}
\big|R_j(\xi,z-\hat x)\big|\leq C \left\lbrace\begin{aligned}
& |z-\hat x|^{j}						& & \text{if }|z-\hat x|<1,\ j\leq\intpart\beta-1, \\
& |z-\hat x|^{\beta-1} 				& & \text{if }|z-\hat x|<1,\ j=\intpart\beta, \\
& r^{\beta-j-1}|z-\hat x|^{j+1} 		& & \text{if }|z-\hat x|<r,\ j\geq\intpart\beta+1.
\end{aligned}\right.
\end{align*}
Remark also that
\begin{align*}
b_i(\xi,-\theta)=(-1)^{i+1}b_i(\xi,\theta), \qquad \theta\in\Sn.
\end{align*}
As an example, one has
\begin{align*}
b_0(\xi,\theta)=J\big(D\phi(\xi)\,\theta\big)
\qquad\text{and}\qquad
b_1(\xi,\theta)=DJ\big(D\phi(\xi)\,\theta\big)\big[ D\phi(\xi)\,\theta\big].
\end{align*}

Inside $B_r(p)$, we have the following bounds for $b_i$.

\begin{lemma}\label{lem:bi-estim}
Let $\alpha\in\N^n$, $\theta\in\Sn$, and $\xi\in B_r(p)$. 
There exists $C>0$ (independent of $r$ and $p$) such that
\begin{align*}
\big|\partial^\alpha_\xi b_i(\xi,\theta)\big|\leq 
\left\lbrace\begin{aligned}
& C & & \text{if } |\alpha|+i+1<\beta, \\
& Cr^{\beta-i-1-|\alpha|} & & \text{if } \beta<|\alpha|+i+1<2\beta+2.
\end{aligned}\right.
\end{align*}
\end{lemma}
\begin{proof}
By \eqref{J-expansion-first} and \eqref{J-expansion},  
each $b_i(\cdot,\theta)$ contains derivatives and of the kernel $J$ of order $i$ and of the diffeomorphism $\phi$ of order $i+1$ at most.
From \eqref{diffeo} it follows then the claim of the lemma.
\end{proof}

\begin{lemma}\label{lem:error-estim}
Let $\alpha,\gamma\in\N^n$, $\gamma\leq\alpha$, $w\in B_r$, and $\xi\in B_r(p)$. 
There exists $C>0$ such that
\begin{align*}
\big| \partial^\gamma_w\partial^{\alpha-\gamma}_\xi R_{|\gamma|+1}(\xi,w) \big| \leq 
\left\lbrace\begin{aligned}
& C|w| & & \text{if } |\alpha|+2<\beta \\
& Cr^{\beta-|\alpha|-2}|w| & & \text{if } |\alpha|+2>\beta,\ |\alpha|+|\gamma|<2\beta.
\end{aligned}\right.
\end{align*}
\end{lemma}
\begin{proof}
Equation \eqref{J-expansion} can be rewritten, using the homogeneity of $J$, as 
\begin{align}\label{4564654546464}
J\Big(\frac{\phi(\xi+w)-\phi(\xi)}{|w|}\Big) & = \sum_{i=0}^{j-1}
b_i(\xi,\lrangle{w})|w|^i+R_j(\xi,w), \qquad w\in B_{r/2}, \xi\in B_r(p).
\end{align}
Fixing $\alpha,\gamma\in\N^n$, $w\neq 0$, and choosing $j=|\gamma|+1$ in the last formula, one can show that
\begin{align*}
\partial_w^\gamma\partial_\xi^{\alpha-\gamma} J\Big(\frac{\phi(\xi+w)-\phi(\xi)}{|w|}\Big) 
= \sum_{i=0}^{|\gamma|}
\partial_w^\gamma\partial_\xi^{\alpha-\gamma} \Big( b_i(\xi,\lrangle{w})|w|^i \Big)
+\partial_w^\gamma\partial_\xi^{\alpha-\gamma}R_{|\gamma|+1}(\xi,w),
\end{align*}
with\footnote{This can be performed by writing 
$w=t\theta$, $t>0$, $\theta\in\Sn$, decomposing the derivatives in $w$ into derivatives in $t$ and $\theta$,
and noticing that \eqref{4564654546464} is a Taylor expansion in the $t$ variable.}
\begin{align*}
\big| \partial_w^\gamma\partial_\xi^{\alpha-\gamma}R_{|\gamma|+1}(\xi,w) \big| \leq c(\phi) |w|.
\end{align*}
Since $R_{|\gamma|+1}$ contains derivatives of $J$ of order $|\gamma|+1$ and of $\phi$ of order $|\gamma|+2$ (\textit{cf.} \eqref{J-expansion-first}),
then
\begin{align*}
c(\phi)\leq C\|\phi\|_{C^{|\alpha|+2}(B_r(p))}\leq 
\left\lbrace\begin{aligned}
& C & & \text{if } |\alpha|+2<\beta, \\
& Cr^{\beta-|\alpha|-2} & & \text{if } |\alpha|+2>\beta,\ |\alpha|+|\gamma|<2\beta.
\end{aligned}\right.
\end{align*}
\end{proof}

The following is an important regularity result in which we use crucially the fact that $(z_n)_+^{s-1}$ solves $L\big[(z_n)_+^{s-1}\big]=0$ in $\{z_n>0\}$.

\begin{lemma}\label{lem:cancellation} 
Let $\psi\in C^{\beta-1}(\overline B_1)$ and, for $j\in\N$,
$a_j\in C^{j+\beta}(\Sn)$ be such that $\beta-s\not\in\N$ and
\begin{align}\label{parity-aj}
a_j(-\theta)= (-1)^{j+1} a_j(\theta),
\qquad\theta\in\Sn.
\end{align} 
Then the function defined by
\begin{align*}
I_j(x)\ :=\ \pv\int_{B_1} (z_n)_+^{s-1}\,\frac{a_j\big(\lrangle{z-x}\big)}{{|z-x|}^{n+2s-j-1}}\;\psi(z)\;dz
\end{align*}
is of class~$C^{j+\beta-1-s}$ in $B_{1/2}$ with
\begin{align*}
\|I_j\|_{C^{j+\beta-s-1}(B_{1/2})}\leq \|a_j\|_{C^{j+\beta}(\Sn)}\|\psi\|_{C^{\beta-1}(\overline B_1)}.
\end{align*}
\end{lemma}
\begin{proof}
Fix some point $x_0\in B_{1/2}$ and write
\begin{align*}
\psi(z) &= \sum_{|\alpha|\leq\lfloor\beta\rfloor-1}\Psi_\alpha(x_0)(z-x_0)^\alpha+P(x_0,z) 
         = \sum_{|\alpha|\leq\lfloor\beta\rfloor-1}\widetilde\Psi_\alpha(x_0,x)(z-x)^\alpha+P(x_0,z)
\end{align*}
where $\Psi_\alpha,\widetilde\Psi_\alpha(\cdot,x)\in C^{\beta-1-|\alpha|}(B_{1/2})$ for any $x\in B_1$,
$\widetilde\Psi_\alpha(x_0,\cdot)\in C^\infty(B_1)$ for every $x_0\in B_{1/2}$, and
$P\in C^{\beta-1}(B_{1/2}\times B_1)$ with 
\begin{align*}
\big|P(x_0,z)\big|\leq\|\psi\|_{C^{\beta-1}(B_1)}|z-x_0|^{\beta-1}.
\end{align*}
We plug such expansion for $\psi$ in the integral defining $I_j$: mind that the coefficients
$\widetilde\Psi_\alpha$ exit the integral.
For any $\alpha$ as above, we have 
\begin{align}\label{ij-taylor-terms}
\pv\int_{B_1} (z_n)_+^{s-1}\frac{a_j\big(\lrangle{z-x}\big)}{{|z-x|}^{n+2s-j-1}}\;(z-x)^\alpha\;dz
=\pv\int_{B_1} (z_n)_+^{s-1}\frac{a_j\big(\lrangle{z-x}\big)\,\lrangle{z-x}^\alpha}{{|z-x|}^{n+2s-j-|\alpha|-1}}\;dz,
\end{align}
where we underline that $a_j(\lrangle{z-x})\lrangle{z-x}^\alpha$ satisfies \eqref{parity-aj} by replacing $j$ with $j+|\alpha|$.
We differentiate $j+|\alpha|+1$ times (remark that, by assumption, $j+|\alpha|<j+\beta-1$)
by exploiting the homogeneity of the kernel as follows: 
take $\gamma\in\N^n$ with $|\gamma|=j+1$ and
\[\begin{split}
\partial_x^{\alpha+\gamma}\:\pv\int_{B_1} (z_n)_+^{s-1}\frac{a_j\big(\lrangle{z-x}\big)\lrangle{z-x}^\alpha}{{|z-x|}^{n+2s-j-|\alpha|-1}}\;dz 
&=\pv\int_{B_1} (z_n)_+^{s-1}\frac{\widetilde a_j\big(\lrangle{z-x}\big)}{{|z-x|}^{n+2s}}\;dz  \\
&=\ -\int_{\R^N\setminus B_1} (z_n)_+^{s-1}\frac{\widetilde a_j\big(\lrangle{z-x}\big)}{{|z-x|}^{n+2s}}\;dz
\end{split}\]
where we have used that $\widetilde a_j$ are homogeneous of degree 0, even on $\mathbb{S}^{n-1}$, and \cite{duke}*{Lemma 9.6}.
The expression obtained for the derivatives is smooth and \textit{a fortiori} the original function will be.

Now, we deal with the regularity of the remainder
\begin{align}\label{4653}
\int_{B_1} (z_n)_+^{s-1}\,\frac{a_j\big(\lrangle{z-x}\big)}{{|z-x|}^{n+2s-j-1}}\;P(z,x_0)\;dz.
\end{align}
The idea is to show that we can take $\lfloor\beta\rfloor-1$ derivatives in a fixed direction exactly at the point $x_0$
(which has been fixed before, but it is arbitrary) via appropriate limits of higher order difference quotients.
To this end, let us denote by
\begin{align}\label{def:finite difference}
\Delta_h^k f(x)=\sum_{i=0}^k(-1)^i\binom{k}{i}\;f\bigg(x+\Big(\frac{k}2-i\Big)h\bigg)
\end{align}
the centred finite difference of order $k$ and recall that
\begin{align*}
\lim_{|h|\downarrow 0} |h|^{-k}\Delta_h^k f(x) = \frac{\partial^k f}{\partial h^k}(x).
\end{align*}
Consider Lemma \ref{lem:incr-quot} with $N=\intpart{\beta}-1$, $\gamma=n+2s-j-1$, and
\begin{align*}
\kappa(x)=\frac{a_j(\lrangle{x})}{|x|^{n+2s-j-1}}.
\end{align*}
Since, for any $|h|<\frac12$, by Lemma \ref{lem:incr-quot} it follows
\begin{align*}
\frac1{|h|^{N}} |\Delta_h^N\kappa(z-x_0)\;P(z,x_0)|
& \leq C\frac{{(|z-x_0|+|h|)}^{(\gamma-1)N}}{\prod_{i=0}^N\left|z-x_0+\big(\frac{N}2-i\big)h\right|^{\gamma}}\;|P(z,x_0)| \\
& \leq C\|\psi\|_{C^{\beta-1}(B_1)}\frac{{(|z-x_0|+|h|)}^{(\gamma-1)N}}{\prod_{i=0}^N\left|z-x_0+\big(\frac{N}2-i\big)h\right|^{\gamma}}\;|z-x_0|^{\beta-1},
\end{align*}
then, for any $\eps>0$,
\begin{align*}
& \frac1{|h|^{N}} \int_{E} |\Delta_h^N\kappa(z-x_0)\,P(z,x_0)| \; dz \\
& \leq C\|\psi\|_{C^{\beta-1}(B_1)} \int_E 
\frac{{(|z-x_0|+|h|)}^{(\gamma-1)N}}{\prod_{i=0}^N\left|z-x_0+\big(\frac{N}2-i\big)h\right|^{\gamma}}\,|z-x_0|^{\beta-1} \, dz
\\ & \leq
C\|\psi\|_{C^{\beta-1}(B_1)} {|h|}^{-N-\gamma+\beta-1+n}\int_{\frac1{|h|}(E-x_0)}
\frac{{(|\zeta|+1)}^{(\gamma-1)N}}{\prod_{i=0}^N\left|\zeta+\big(\frac{N}2-i\big)\frac{h}{|h|}\right|^{\gamma}}\;|\zeta|^{\beta-1} \; d\zeta \\
& \leq C\|\psi\|_{C^{\beta-1}(B_1)} {|h|}^{-N-\gamma+\beta-1+n}\int_{\frac1{|h|}(E-x_0)}
{(|\zeta|+1)}^{-N-\gamma+\beta-1} \; d\zeta \\
& \leq C\|\psi\|_{C^{\beta-1}(B_1)} \int_{E} {(|z-x_0|+1)}^{-N-\gamma+\beta-1} \; dz < \eps
\end{align*}
provided that $|E|$ is small enough, regardless the value of $|h|$.
This means that $|h|^{-N}\Delta_h^N\kappa(z-x_0)P(z,x_0)$ is uniformly integrable in $B_1$ 
(as a family indexed on $|h|$). As it is also pointwisely converging, we conclude by the 
Vitali convergence theorem (\textit{cf.} for example \cite{tao}*{Theorem 1.5.13}) 
that every $N$-th order derivative of \eqref{4653}
is of type
\begin{align*}
\int_{B_1} (z_n)_+^{s-1}\,\frac{\widehat a_j\big(\lrangle{z-x_0}\big)}{{|z-x_0|}^{n+2s-j-1+N}}\;P(z,x_0)\;dz
\end{align*}
which, in turn, is of magnitude
\begin{align*}
\|a_j\|_{C^N(\Sn)}\|\psi\|_{C^{\beta-1}(B_1)} \int_{B_1} (z_n)_+^{s-1}{|z-x_0|}^{-n-2s+j+\beta-N} \;dz
\leq C\|a_j\|_{C^N(\Sn)}\|\psi\|_{C^{\beta-1}(B_1)}
\end{align*}
if $N\leq j+\intpart\beta-2$, or
\begin{multline*}
\|a_j\|_{C^N(\Sn)}\|\psi\|_{C^{\beta-1}(B_1)} \int_{B_1} (z_n)_+^{s-1}{|z-x_0|}^{-n-2s+j+\beta-N} \;dz\ \leq \\
\leq\ C\|a_j\|_{C^N(\Sn)}\|\psi\|_{C^{\beta-1}(B_1)}{(x_{0})}_{n}^{\beta-\intpart\beta-s}
\end{multline*}
if $N= j+\intpart\beta-1$ and $\beta-\intpart\beta<s$, or
\begin{multline*}
\|a_j\|_{C^N(\Sn)}\|\psi\|_{C^{\beta-1}(B_1)} \int_{B_1} (z_n)_+^{s-1}{|z-x_0|}^{-n-2s+j+\beta-N} \;dz\ \leq \\
\leq\ C\|a_j\|_{C^N(\Sn)}\|\psi\|_{C^{\beta-1}(B_1)}{(x_{0})}_{n}^{\beta-\intpart\beta-s-1}
\end{multline*}
if $N= j+\intpart\beta$ and $\beta-\intpart\beta>s$.
We therefore have that 
\begin{align*}
\|I_j\|_{C^{j+\beta-s-1}(\overline{B}_1)}\leq \|a_j\|_{C^N(\Sn)}\|\psi\|_{C^{\beta-1}(B_1)}.
\end{align*}
\end{proof}

The following is the interior regularity counterpart of Lemma \ref{lem:cancellation},
which was instead studying some global regularity.

\begin{lemma}\label{lem:cancellation2} 
Let $p\in\{x\in\Rn:x_n>0\}$ and $r>0$ be such that $\overline{B_r(p)}\subset B_1$.
For $k\in\N$, $\beta>1$, $k>\beta-s$, let $\psi\in C^{\beta-1}(\overline B_1)\cap C^k(\overline{B_r(p)})$. 
For $j\in\N$, let $a_j\in C^{j+k-1}(\Sn)$ satisfy \eqref{parity-aj}.
Then the function defined by
\begin{align*}
I_j(x)\ :=\ \pv\int_{B_1} (z_n)_+^{s-1}\,\frac{a_j\big(\lrangle{z-x}\big)}{{|z-x|}^{n+2s-j-1}}\;\psi(z)\;dz
\end{align*}
is of class~$C^{k+j-1}$ in $B_{r/2}(p)$ with
\begin{multline}\label{eq:cancellation2}
|D^{k+j-1}I_j(x)|\leq \\
\leq C \|a_j\|_{C^{j+k-1}(\Sn)}
\left(
r^{\intpart\beta+1-k-s}+\|\psi\|_{C^k(\overline{B_r(p)})} r^{1-s}+\|\psi\|_{C^{\beta-1}(\overline{B}_1)}r^{\beta-k-s}
\right),
\quad x\in B_{r/2}(p).
\end{multline}
\end{lemma}

\begin{proof}
The $\pv\ $ specification only matters when $j=0$, so we allow ourselves to drop it from now on.

Fix some point $x_0\in B_{r/2}(p)$ and write
\begin{align*}
\psi(z) &= \sum_{|\alpha|\leq k-1}\partial^\alpha\psi(x_0)(z-x_0)^\alpha+P_k(x_0,z) \\
        &= \sum_{|\alpha|\leq k-1}\Psi_\alpha(x_0,x)(z-x)^\alpha+P_k(x_0,z),
&  z\in B_r(p) \\
\psi(z) &= \sum_{|\alpha|\leq \intpart\beta-1}\partial^\alpha\psi(x_0)(z-x_0)^\alpha+P_\beta(x_0,z)  \\
        &= \sum_{|\alpha|\leq \intpart\beta-1}\Psi_\alpha(x_0,x)(z-x)^\alpha+P_\beta(x_0,z),
&  z\in B_1\setminus B_r(p)         
\end{align*}
where $\Psi_\alpha(\cdot,x)\in C^{k-|\alpha|}(\overline{B_{r/2}(p)})$ for any $x\in B_r(p)$,
$\Psi_\alpha(x_0,\cdot)\in C^\infty(B_r(p))$ for every $x_0\in B_{r/2}(p)$, and
$P\in C^k(\overline{B_{r/2}(p)}\times \overline{B_r(p)})$ with 
\begin{align*}
\big|P(x_0,z)\big|\leq\|\psi\|_{C^k(\overline{B_r(p)})}|z-x_0|^k.
\end{align*}
for any $x_0\in B_{r/2}(p)$ and $|\alpha|\leq\intpart\beta-1$.
We plug these expansions into the definition of $I_j$ so that
\begin{align}
& I_j(x) = \int_{B_r(p)} (z_n)_+^{s-1}\,\frac{a_j\big(\lrangle{z-x}\big)}{{|z-x|}^{n+2s-j-1}}\;\psi(z)\;dz
			+\int_{B_1\setminus B_r(p)} (z_n)_+^{s-1}\,\frac{a_j\big(\lrangle{z-x}\big)}{{|z-x|}^{n+2s-j-1}}\;\psi(z)\;dz 
			\nonumber \\
& = \sum_{|\alpha|\leq \intpart\beta-1}\Psi_\alpha(x_0,x)
\int_{B_1} (z_n)_+^{s-1}\,\frac{a_j\big(\lrangle{z-x}\big)\lrangle{z-x}^\alpha}{{|z-x|}^{n+2s-j-1-|\alpha|}}\;dz  
	\label{globalterm}\\
& + \sum_{\intpart\beta\leq|\alpha|\leq k-1}\Psi_\alpha(x_0,x) 
\int_{B_r(p)} (z_n)_+^{s-1}\,\frac{a_j\big(\lrangle{z-x}\big)\lrangle{z-x}^\alpha}{{|z-x|}^{n+2s-j-1-|\alpha|}}\;dz\ +
	\label{localterm}\\
&  \int_{B_r(p)} (z_n)_+^{s-1}\frac{a_j\big(\lrangle{z-x}\big)}{{|z-x|}^{n+2s-j-1}}\,P_k(x_0,z)\,dz
+ \int_{B_1\setminus B_r(p)} (z_n)_+^{s-1}\frac{a_j\big(\lrangle{z-x}\big)}{{|z-x|}^{n+2s-j-1}}\,P_\beta(x_0,z)\,dz.
	\label{remainders}
\end{align}
We study now the regularity of \eqref{globalterm}, \eqref{localterm}, and \eqref{remainders}. The one of \eqref{globalterm} is proved in an analogous way to that of \eqref{ij-taylor-terms}, so we skip this.

Let us look at \eqref{localterm}. Again we use some ideas from the study of \eqref{ij-taylor-terms}. 
For any $\gamma\in\N^n$ with $|\gamma|=j+|\alpha|$
\begin{multline*}
\partial^\gamma\int_{B_r(p)} (z_n)_+^{s-1}\,\frac{a_j\big(\lrangle{z-x}\big)\lrangle{z-x}^\alpha}{{|z-x|}^{n+2s-j-1-|\alpha|}}\;dz =
\int_{B_r(p)} (z_n)_+^{s-1}\,\frac{\widetilde a_j\big(\lrangle{z-x}\big)}{{|z-x|}^{n+2s-1}}\;dz\ = \\
=\ -\int_{\Rn\setminus B_1} (z_n)_+^{s-1}\,\frac{\widetilde a_j\big(\lrangle{z-x}\big)}{{|z-x|}^{n+2s-1}}\;dz 
-\int_{B_1\setminus B_r(p)} (z_n)_+^{s-1}\,\frac{\widetilde a_j\big(\lrangle{z-x}\big)}{{|z-x|}^{n+2s-1}}\;dz.
\end{multline*}
This means that, when we take $\eta\in\N^n$ such that $|\eta+\gamma|=j+k-1$ (and therefore $|\eta|=k-1-|\alpha|$), it holds
\begin{multline*}
\partial^{\eta+\gamma}\int_{B_r(p)} (z_n)_+^{s-1}\,\frac{a_j\big(\lrangle{z-x}\big)\lrangle{z-x}^\alpha}{{|z-x|}^{n+2s-j-1-|\alpha|}}\;dz = g(x)-\partial^\eta\int_{B_1\setminus B_r(p)} (z_n)_+^{s-1}\,\frac{\widetilde a_j\big(\lrangle{z-x}\big)}{{|z-x|}^{n+2s-1}}\;dz \\
=g(x)-\int_{B_1\setminus B_r(p)} (z_n)_+^{s-1}\,\frac{\widehat a_j\big(\lrangle{z-x}\big)}{{|z-x|}^{n+2s-2+k-|\alpha|}}\;dz
\end{multline*}
for some $g\in C^\infty(\overline{B_{r/2}(p)})$, 
and therefore
\begin{align*}
\left| D^{k+j-1} \int_{B_r(p)} (z_n)_+^{s-1}\,\frac{a_j\big(\lrangle{z-x}\big)\lrangle{z-x}^\alpha}{{|z-x|}^{n+2s-j-1-|\alpha|}}\;dz \right| \leq C\|a_j\|_{C^{k+j-1}(\Sn)}r^{|\alpha|+1-k-s}.
\end{align*}

As to \eqref{remainders}, we proceed as in the proof of Lemma \ref{lem:cancellation}.
In this case
\begin{multline*}
\left| D^{k+j-1}\big|_{x_0}\int_{B_r(p)} (z_n)_+^{s-1}\,\frac{a_j\big(\lrangle{z-x}\big)}{{|z-x|}^{n+2s-j-1}}\;P_k(x_0,z)\;dz \right|\ \leq \\
\leq\ C\|a_j\|_{C^{k+j-1}(\Sn)}\|\psi\|_{C^k(\overline{B_r(p)})}\int_{B_r(p)} \frac{(z_n)_+^{s-1}}{{|z-x|}^{n+2s-2}}\;dz \ \leq \\
\leq\ C\|a_j\|_{C^{k+j-1}(\Sn)}\|\psi\|_{C^k(\overline{B_r(p)})} r^{1-s}
\end{multline*}
where we have used Lemma \ref{scaling of integral}. Similarly,
\begin{multline*}
\left| D^{k+j-1}\big|_{x_0} \int_{B_1\setminus B_r(p)} (z_n)_+^{s-1}\,\frac{a_j\big(\lrangle{z-x}\big)}{{|z-x|}^{n+2s-j-1}}\;P_\beta(x_0,z)\;dz \right| \ \leq \\
\leq\ C\|a_j\|_{C^{k+j-1}(\Sn)}\|\psi\|_{C^{\beta-1}(\overline{B}_1)}\int_{B_1\setminus B_r(p)} \frac{(z_n)_+^{s-1}}{{|z-x|}^{n+2s-1+k-\beta}}\;dz \ \leq \\
\leq\ C\|a_j\|_{C^{k+j-1}(\Sn)}\|\psi\|_{C^{\beta-1}(\overline{B}_1)}r^{\beta-k-s}
\end{multline*}
again in view of Lemma \ref{scaling of integral}.
\end{proof}

Using the previous results, we can finally give the:

\begin{proof}[Proof of Theorem \ref{thm:Lds}]
We are interested in proving the $C^{\beta-s-1}$ regularity of $I$ defined as in~\eqref{operative} and remodulated as in \eqref{operative-splitted}. 
In order to do so, we are going to take the derivatives of $I_1$ and $I_r$ in the $\hat x$ variable \emph{evaluated} at the point $p$. 
For this reason, we can think of $\hat x\in B_r(p)$.

We fix $\alpha\in\N^n$ such that
\begin{align}\label{q>beta}
q:=|\alpha|>\beta.
\end{align}

For simplicity, throughout the rest of the proof we drop the hat script and we recast $\hat x$ to simply $x$.
We also drop the $\pv\ $ specification in (some of) the integrals.

Let us consider first $I_r$. 
By Lemma \ref{derivative-diagonal}, we have
\begin{align*}
\partial_x^\alpha\big|_p I_r(x,x) = \sum_{\gamma\leq\alpha}\binom{\alpha}{\gamma}
\partial_x^\gamma\big|_{p}\partial_\xi^{\alpha-\gamma}\big|_{p}I_r(\xi,x).
\end{align*}
Moreover, we take advantage of the expansion of the kernel $J$ in \eqref{J-expansion}.
In particular, we proceed as follows 
\begin{align*}
\partial_x^\alpha I_r(x,x) = & \ \sum_{\gamma\leq\alpha}\binom{\alpha}{\gamma}
\partial_x^\gamma\partial_\xi^{\alpha-\gamma}I_r(\xi,x) \\
= &\ \sum_{\gamma\leq\alpha}\binom{\alpha}{\gamma}\sum_{i=0}^{|\gamma|}
\partial_x^\gamma\partial_\xi^{\alpha-\gamma}\int_{B_r(p)}(z_n)_+^{s-1}
\frac{b_i\big(\xi,\lrangle{z-x}\big)}{{|z-x|}^{n+2s-i-1}}\cdot\rho(z)\;dz  \\
&\ +\sum_{\gamma\leq\alpha}\binom{\alpha}{\gamma}
\partial_x^\gamma\partial_\xi^{\alpha-\gamma}\int_{B_r(p)}(z_n)_+^{s-1}\frac{R_{|\gamma|+1}(\xi,z-x)}{{|z-x|}^{n+2s-1}}\cdot\rho(z)\;dz.
\end{align*}
Notice that, to use \eqref{J-expansion}, we implicitly use $|\gamma|+1$ derivatives on the kernel $J$. Then, we want to compute $|\alpha|$ more, so we need $|\alpha|+|\gamma|+1<2\beta+1$: since $\gamma\leq\alpha$, this gives $|\alpha|<\beta$.

First of all, let us deal with the integral borne by the error term: using Lemma \ref{lem:error-estim},
\begin{align*}
& \left|\partial_x^\gamma\partial_\xi^{\alpha-\gamma}\int_{B_r(p)}(z_n)_+^{s-1}\frac{R_{|\gamma|+1}(\xi,z-x)}{{|z-x|}^{n+2s-1}}\cdot\rho(z)\;dz\right|\leq  \\
& \leq \int_{B_r(p)}(z_n)_+^{s-1}\frac{\big|\partial_x^\gamma\partial_\xi^{\alpha-\gamma}R_{|\gamma|+1}(\xi,z-x)\big|}{{|z-x|}^{n+2s-1}}\,|\rho(z)|\;dz \\
& \leq C(\rho(0)+|p|)r^{\beta-|\alpha|-2}\int_{B_r(p)}(z_n)_+^{s-1}{|z-x|}^{-n-2s+2}\;dz
\leq C(\rho(0)+|p|)r^{\beta-s-1-|\alpha|}
\end{align*} 
where, in the last estimate, we have applied Lemma \ref{scaling of integral} 
and the fact that $|\rho(z)|\leq |\rho(0)|+C|p|$ in $B_r(p)$ (recall to this end that $d(p)=2r$ by assumption).
This takes care of the remainder term.

If $|\gamma|>i+\beta-1-s$, we write
\begin{multline*}
\partial_x^\gamma\int_{B_r(p)}(z_n)_+^{s-1}
\frac{\partial_\xi^{\alpha-\gamma}b_i\big(\xi,\lrangle{z-x}\big)}{{|z-x|}^{n+2s-i-1}}\cdot\rho(z)\;dz \ = \\
=\ -\partial_x^\gamma\int_{B_1\setminus B_r(p)}(z_n)_+^{s-1}
\frac{\partial_\xi^{\alpha-\gamma}b_i\big(\xi,\lrangle{z-x}\big)}{{|z-x|}^{n+2s-i-1}}\cdot\rho(z)\;dz 
+\Gamma(x) 
\end{multline*}
for 
\begin{align*}
\Gamma(x)=\partial_x^\gamma\int_{B_1}(z_n)_+^{s-1}
\frac{\partial_\xi^{\alpha-\gamma}b_i\big(\xi,\lrangle{z-x}\big)}{{|z-x|}^{n+2s-i-1}}\cdot\rho(z)\;dz 
\end{align*}
to which we can apply Lemma \ref{lem:cancellation2} 
to say
\begin{multline*}
\big|\Gamma(x)\big|\leq C \|\partial_\xi^{\alpha-\gamma}b_i\big(\xi,\cdot\big)\|_{C^{|\gamma|}(\Sn)}\times \\
\times\left(
r^{\intpart\beta-|\gamma|+i-s}+\|\rho\|_{C^{|\gamma|-i+1}(\overline{B_r(p)})} r^{1-s}+\|\rho\|_{C^{\beta-1}(\overline{B}_1)}r^{\beta-|\gamma|+i-1-s}
\right), 
\qquad x\in B_{r/2}(p).
\end{multline*}
Note that
\begin{align*}
\frac{\partial_\xi^{\alpha-\gamma}b_i\big(\xi,\lrangle{z-x}\big)}{{|z-x|}^{n+2s-i-1}}
\end{align*}
is still homogeneous in $z-x$ and, in view of this and of Lemma \ref{lem:bi-estim},
\begin{align}\label{32525}
\left|\partial_x^\gamma\left(\frac{\partial_\xi^{\alpha-\gamma}b_i\big(\xi,\lrangle{z-x}\big)}{{|z-x|}^{n+2s-i-1}}\right)\right|
\leq \left\lbrace\begin{aligned} 
& C{|z-x|}^{-n-2s+i+1-|\gamma|} & & \text{if }|\alpha|-|\gamma|+i+1<\beta, \\
& Cr^{\beta-i-1-|\alpha|+|\gamma|}{|z-x|}^{-n-2s+i+1-|\gamma|} & & \text{if }|\alpha|-|\gamma|+i+1>\beta.
\end{aligned}\right.
\end{align}
This yields, \textit{cf}. \eqref{eq:cancellation2},
\begin{multline*}
\big|\Gamma(x)\big|\leq C\left(
r^{\intpart\beta-|\gamma|+i-s}+r^{\beta-|\gamma|+i-1-s}+r^{\beta-|\gamma|+i-1-s}
\right) \times \\
\times\left\lbrace\begin{aligned} 
& 1 & & \text{if }|\alpha|-|\gamma|+i+1<\beta, \\
& r^{\beta-i-1-|\alpha|+|\gamma|}  & & \text{if }|\alpha|-|\gamma|+i+1>\beta.
\end{aligned}\right.
\end{multline*}

Now, when $|\gamma|<i+\beta-1-s$ write
\begin{multline*}
\int_{B_r(p)}(z_n)_+^{s-1}
\frac{b_i\big(\xi,\lrangle{z-x}\big)}{{|z-x|}^{n+2s-i-1}}\cdot\rho(z)\;dz\ = \\
=\ \int_{B_1}(z_n)_+^{s-1}
\frac{b_i\big(\xi,\lrangle{z-x}\big)}{{|z-x|}^{n+2s-i-1}}\cdot\rho(z)\;dz-
\int_{B_1\setminus B_r(p)}(z_n)_+^{s-1}
\frac{b_i\big(\xi,\lrangle{z-x}\big)}{{|z-x|}^{n+2s-i-1}}\cdot\rho(z)\;dz.
\end{multline*}
The analysis of the first addend on the right-hand side is covered by Lemma \ref{lem:cancellation},
in view of the relation $|\gamma|<i+\beta-1-s$. 
So we only deal with the integrals in the ``annular'' region $B_1\setminus B_r(p)$.
For $\beta-|\alpha|+|\gamma|-1<i\leq|\gamma|$ we have (\textit{cf}. \eqref{32525})
\begin{multline*}
\left|\partial_x^\gamma\partial_\xi^{\alpha-\gamma}\int_{B_1\setminus B_r(p)}(z_n)_+^{s-1}
\frac{b_i\big(\xi,\lrangle{z-x}\big)}{{|z-x|}^{n+2s-i-1}}\cdot\rho(z)\;dz\right|
\leq \\
\leq r^{\beta-i-1-|\alpha|+|\gamma|}\int_{B_1\setminus B_r(p)}(z_n)_+^{s-1}{|z-x|}^{-n-2s+i+1-|\gamma|}
|\rho(z)|\;dz;
\end{multline*}
using now that $|\rho(z)|\leq |\rho(0)|+C|z|$ for $z\in B_1$, and applying Lemma \ref{scaling of integral},
we deduce
\begin{align*}
\left|\sum_{i=\intpart\beta-|\alpha|+|\gamma|}^{|\gamma|}\partial_x^\gamma\big|_p\partial_\xi^{\alpha-\gamma}\big|_p\int_{B_1\setminus B_r(p)}(z_n)_+^{s-1}
\frac{b_i\big(\xi,\lrangle{z-x}\big)}{{|z-x|}^{n+2s-i-1}}\cdot\rho(z)\;dz\right|
\leq C(|\rho(0)|+|p|)\,r^{\beta-s-1-|\alpha|}.
\end{align*}
So we are left with
\begin{align*}
\partial_x^\alpha\big|_p I_r(x,x)=\sum_{\gamma\leq\alpha}\binom{\alpha}{\gamma}\sum_{i=0}^{\intpart\beta-|\alpha|+|\gamma|-1}
\partial_x^\gamma\partial_\xi^{\alpha-\gamma}\int_{B_1\setminus B_r(p)}
\!\!\!(z_n)_+^{s-1}
\frac{b_i\big(\xi,\lrangle{z-x}\big)}{{|z-x|}^{n+2s-i-1}}\cdot\rho(z)\;dz 
+\Theta_q(r,p)
\end{align*}
where $|\Theta_q(r,p)|\leq C(|\rho(0)|+|p|)\,r^{\beta-s-1-q}$, $q=|\alpha|$.

We now claim that
\begin{multline}\label{integral of radial claim}
\left| \big( \partial^\alpha I_1 \big) (p)-\sum_{\gamma\leq\alpha}\binom{\alpha}{\gamma}\sum_{i=0}^{\intpart\beta-|\alpha|+|\gamma|-1}
\partial_x^\gamma\big|_p\partial_\xi^{\alpha-\gamma}\big|_p\int_{B_1\setminus B_r(p)}(z_n)_+^{s-1}
\frac{b_i\big(\xi,\lrangle{z-x}\big)}{{|z-x|}^{n+2s-i-1}}\cdot\rho(z)\;dz \right| \leq \\
\leq\ C(|\rho(0)|+|p|)\,r^{\beta-s-1-q}.
\end{multline}
To this end, we are going to prove that, for $z\in B_1\setminus B_r(p)$,
\begin{align}\label{radial claim}
\left| \partial_x^\alpha\big|_p J\big(\phi(z)-\phi(x)\big)
-\sum_{\gamma\leq\alpha}\binom{\alpha}{\gamma}\sum_{i=0}^{\intpart\beta-|\alpha|+|\gamma|-1}
\partial_x^\gamma\big|_p\partial_\xi^{\alpha-\gamma}\big|_p
\frac{b_i\big(\xi,\lrangle{z-x}\big)}{{|z-x|}^{n+2s-i-1}} \right|
\leq C \frac{r^{\beta-|\alpha|}}{{|z-p|}^{n+2s}}
\end{align}
and we postpone the proof to further below.

If \eqref{radial claim} holds then, 
recalling the definition of $I_1$ in \eqref{operative-splitted} and using \eqref{radial claim}, we entail
\begin{align*}
& \left| \big( \partial^\alpha I_1 \big) (p)-\sum_{\gamma\leq\alpha}\binom{\alpha}{\gamma}\sum_{i=0}^{\intpart\beta-|\alpha|+|\gamma|-1}
\partial_x^\gamma\big|_p\partial_\xi^{\alpha-\gamma}\big|_p\int_{B_1\setminus B_r(p)}(z_n)_+^{s-1}
\frac{b_i\big(\xi,\lrangle{z-x}\big)}{{|z-x|}^{n+2s-i-1}}\cdot\rho(z)\;dz \right| \leq \\
& \leq\ \int_{B_1\setminus B_r(p)}(z_n)_+^{s-1} \; \Bigg| \partial_x^\alpha\big|_p J\big(\phi(z)-\phi(x)\big) \\
& \qquad\qquad -\sum_{\gamma\leq\alpha}\binom{\alpha}{\gamma}\sum_{i=0}^{\intpart\beta-|\alpha|+|\gamma|-1}
\partial_x^\gamma\big|_p\partial_\xi^{\alpha-\gamma}\big|_p
\frac{b_i\big(\xi,\lrangle{z-x}\big)}{{|z-x|}^{n+2s-i-1}} \Bigg| \; |\rho(z)| \;dz  \\
& \leq\ \, r^{\beta-|\alpha|} \int_{B_1\setminus B_r(p)}(z_n)_+^{s-1} {|z-p|}^{-n-2s} (|\rho(0)|+|z|) \;dz  
\ \leq\ C (|\rho(0)|+|p|) \, r^{\beta-1-s-|\alpha|}
\end{align*}
where we have used Lemma \ref{scaling of integral} in the last passage, proving \eqref{integral of radial claim}.

Remark that
\begin{align}
\partial_x^\gamma\big|_p\partial_\xi^{\alpha-\gamma}\big|_p
\frac{b_i\big(\xi,\lrangle{z-x}\big)}{{|z-x|}^{n+2s-i-1}}=
\frac{\widetilde b_i(p,\lrangle{z-p})}{{|z-p|}^{n+2s-i-1-|\gamma|}},
\qquad z\in B_1,\text{ for some $\widetilde b_i$ suitably chosen}, \label{homo of derivatives} \\
i\in\{0,\ldots,\intpart\beta-|\alpha|+|\gamma|-1\},\ \gamma\leq\alpha. \nonumber
\end{align}
and therefore
\begin{align}\label{homo of derivatives 2}
\left(\sum_{\gamma\leq\alpha}\binom{\alpha}{\gamma}\sum_{i=0}^{\intpart\beta-|\alpha|+|\gamma|-1}
\partial_x^\gamma\big|_p\partial_\xi^{\alpha-\gamma}\big|_p
\frac{b_i\big(\xi,\lrangle{z-x}\big)}{{|z-x|}^{n+2s-i-1}}\right)_{|\alpha|\leq q}
=\sum_{j=0}^{\intpart\beta-1}\frac{B_j(p,\lrangle{z-p})}{{|z-p|}^{n+2s-j-1+|\alpha|}}.
\end{align}

\emph{Proof of \eqref{radial claim}}.
Let us first remark that an inequality of the type of \eqref{radial claim} holds for $z\in B_r(p)$
as a result of \eqref{J-expansion} and Lemma \ref{lem:error-estim}; namely,
\begin{align}\label{radial claim in small ball}
\left| \partial_x^\alpha\big|_p J\big(\phi(z)-\phi(x)\big)
-\sum_{\gamma\leq\alpha}\binom{\alpha}{\gamma}\sum_{i=0}^{\intpart\beta-|\alpha|+|\gamma|-1}
\partial_x^\gamma\big|_p\partial_\xi^{\alpha-\gamma}\big|_p
\frac{b_i\big(\xi,\lrangle{z-x}\big)}{{|z-x|}^{n+2s-i-1}} \right|
\leq C \frac{r^{\beta-|\alpha|-1}}{{|z-p|}^{n+2s-1}}.
\end{align}
This is because $\phi$ is smooth in $B_r(p)$, although the estimates on its $\alpha$-derivatives
are getting worse upon approaching the boundary.

We use Lemma \ref{lem:super chain rule} to write
\begin{align*}
& \partial_x^\alpha|_p J\big(\phi(z)-\phi(x)\big) \ = \\
& =\ \sum_{j=0}^{q-1}\sum_{k_1+\ldots+k_{q-j}=j} (-1)^{q-j} c_{q,j,k_1,\ldots,k_{q-j}}
D^{q-j}J\big( \phi(z) - \phi(p) \big) \left[ D^{k_1}\phi(p),\ldots,D^{k_{q-j}}\phi(p) \right]. 
\end{align*}
If $j\geq \intpart\beta>\beta-1$ then 
\begin{align*}
& \Bigg| \sum_{\overset{k_1+\ldots+k_{q-j}=q}{k_1,\ldots,k_{q-j}>0}} (-1)^{q-j}
D^{q-j}J\big( \phi(z) - \phi(p) \big) \left[ D^{k_1}\phi(p),\ldots,D^{k_{q-j}}\phi(p) \right] \Bigg| \ \leq \\
& \leq\ C \, \Big| D^{q-j}J\big( \phi(z) - \phi(p) \big) \left[ D^{j+1}\phi(p), D\phi(p), \ldots, D\phi(p) \right] \Big| \\
& \leq\ C \, \Bigg| D^{q-j}J\Big( \frac{\phi(z) - \phi(p)}{|z-p|} \Big) \Bigg|  \frac{r^{\beta-j-1}}{\big|z-p\big|^{n+2s-1+q-j}}
\end{align*}
by the homogeneity of $J$ and \eqref{diffeo}. In particular, whenever $|z-p|>r$
\begin{multline*}
\Bigg| \sum_{j=\intpart\beta}^{q-1} \sum_{k_1+\ldots+k_{q-j}=j} (-1)^{q-j} c_{q,j,k_1,\ldots,k_{q-j}}
D^{q-j}J\big( \phi(z) - \phi(p) \big) \left[ D^{k_1}\phi(p),\ldots,D^{k_{q-j}}\phi(p) \right] \Bigg|\ \leq \\
\leq\ C\frac{r^{\beta-q}}{{|z-p|}^{n+2s}}.
\end{multline*}
Let us set $z=p+t\theta$, $\theta\in\Sn$, $t=|z-p|>0$.
In view of the last computations and of \eqref{homo of derivatives 2}, we can rewrite \eqref{radial claim} as
\begin{align*}
& \left| \sum_{j=0}^{\intpart\beta-1} t^{-n-2s+1-q+j} \Phi_{q,j}(t,\theta,p)
-\sum_{j=0}^{\intpart\beta-1}{t}^{-n-2s+j+1}B_j(p,\theta) \right|
\leq C \frac{r^{\beta-q}}{{t}^{n+2s}},
\qquad t>r, \\
& \Phi_{q,j}(t,\theta,p):=\binom{q}{j}\sum_{k_1+\ldots+k_{q-j}=j}\binom{j}{k_1,\ldots,k_{q-j}} (-1)^{q-j} \times\\
& \hspace{.29\linewidth} \times D^{q-j}J\Big( \frac{\phi(p+t\theta) - \phi(p)}t \Big) \left[ D^{k_1}\phi(p),\ldots,D^{k_{q-j}}\phi(p) \right] 
\end{align*}
which we write again as
\begin{align}\label{radial claim in its simplest form}
\left| \sum_{j=0}^{\intpart\beta-1} t^{j} \Phi_{q,j}(t,\theta,p)
-\sum_{j=0}^{\intpart\beta-1}{t}^{j}B_j(p,\theta) \right|
\leq C \, r^{\beta-q} \, t^{q-1},
\qquad t>r.
\end{align}
Notice that $\Phi_{q,j}(\cdot,\theta,p)\in C^{\beta-1}([0,1])$ thanks to Lemma \ref{lem:regularity of incremental quotients}
and the regularity of $J$.
Moreover, 
\begin{align*}
\|\Phi_{q,j}(\cdot,\theta,p)\|_{C^{\beta-1}([0,1])}\leq C\,\|\phi\|_{C^{\beta}(\overline{B}_1)}.
\end{align*}
Also, \eqref{radial claim in small ball}
translates to
\begin{align*}
\left| \sum_{j=0}^{\intpart\beta-1} t^{j} \Phi_{q,j}(t,\theta,p)
-\sum_{j=0}^{\intpart\beta-1}{t}^{j}B_j(p,\theta) \right|
\leq C \, r^{\beta-q} \, t^{q-1} \leq Cr^{\beta-1},
\qquad t\leq r.
\end{align*}
Therefore, these two last observations plus Lemma \ref{lem:prolongation Calpha} allow us to conclude that
\eqref{radial claim in its simplest form} holds and in turn \eqref{radial claim} holds as well.
This also completes the proof of \eqref{integral of radial claim}.
\end{proof}

\section{Nonlocal equations for functions with polynomial growth at infinity}
\label{sec:polynom}

We introduce now some tools that will be needed in the following section, where we develop a new higher order version of the blow-up and compactness argument from \cite{duke}.

First, we need the following.

\begin{definition}\label{def:Lgrowth}
Let $k\in\N$, $\Omega\subset\Rn$ a bounded domain,
$u\in L^1_{\rm loc}(\R^n)$, and $f\in L^\infty(\Omega)$. 
Assume that $u$ satisfies
\begin{align*}
\int_{\Rn}\frac{|u(y)|}{1+|y|^{n+2s+k}}\;dy.
\end{align*}
We say that
\begin{align}\label{k-equation}
Lu \stackrel{k}{=} f
\qquad\text{in }\Omega
\end{align}
if there exist a family of polynomials ${(p_R)}_{R>0}\subseteq\mathbf{P}_{k-1}$ 
and a family of functions ${(f_R:\Omega\to\R)}_{R>0}$
such that
\begin{align*}
& L(u\chi_{B_R}) = f_R+p_R
\qquad\text{in }\Omega, \text{ for any }R>\diam(\Omega) \\
\text{and }\ & \lim_{R\uparrow\infty}\|f_R-f\|_{L^\infty(\Omega)} = 0.
\end{align*}
In the case of an unbounded $\Omega$, we say that \eqref{k-equation} holds
if it does in any bounded subdomain.
The equations are to be understood in the distributional sense.
\end{definition}

\begin{remark} 
This definition is very similar to \cite{growth}*{Definition 1.1}
but with one important difference: in \cite{growth} the authors require the convergence $f_R\to f$
to be merely pointwise a.e., whereas we strengthen this by asking it to be uniform.
This simplifies some proofs and it allows us to prove Proposition \ref{prop:growth-bound-reg},
which is an essential tool in our blow-up arguments in Section \ref{sec:uover}.
\end{remark}

\subsection{An associated extension problem}\label{par:extension}
The above definition enjoys, in the particular case when $L$ is the fractional Laplacian, an extension property of Caffarelli-Silvestre type. 

\begin{lemma}\label{lem:cs-extension}
Let $\Omega\subset\R^n$ be any domain, and $\phi\in C(\Rn)\cap C^2(\Omega)$ be such that
\begin{align}\label{weighted-Linfty}
|\phi(x)| \leq C_0\big(1+|x|\big)^{k+s+\alpha},
\qquad \text{for some }k\in\N,\ \alpha<s \text{ and for any }x\in\Rn,
\end{align}
Then, there exists an extension $\widetilde\phi:\Rn\times[0,\infty)\to \R$, with polynomial growth in $\R^n\times [0,\infty)$, such that
\begin{align*}
\left\lbrace\begin{aligned}
\div\big(y^{1-2s}\nabla\widetilde\phi\big) & = 0 & & \text{in }\Rn\times(0,+\infty), \\
\widetilde\phi & = \phi & & \text{on }\Rn\times\{0\}, \\
-a_{n,s}\,y^{1-2s}\partial_y\widetilde\phi & 
\stackrel{k}{=} \Ds\phi & & \text{on }\Omega\times\{0\}.
\end{aligned}\right.
\end{align*}
\end{lemma}

\begin{proof}
Cut off $\phi$ on a ball of radius $R>0$ and define 
\begin{align*}
\phi_R:=\phi\chi_{B_R}
\qquad \text{in }\R^n.
\end{align*}
Recall that (see \cite{caffarelli-silvestre})
\begin{align*}
\widetilde\phi_R(x,y):=\big(P(\cdot,y)*\phi_R\big)(x),
\quad \text{with }\ P(x,y)=c_{n,s}\frac{y^{2s}}{\big(|x|^2+y^2\big)^{n/2+s}},
\qquad x\in\Rn,\ y>0,
\end{align*}
solves
\begin{align*}
\left\lbrace\begin{aligned}
\div\big(y^{1-2s}\nabla\widetilde\phi_R\big) & = 0 & & \text{in }\Rn\times(0,+\infty) \\
\widetilde\phi_R & = \phi_R & & \text{on }\Rn\times\{0\}.
\end{aligned}\right.
\end{align*}
Remark that, for any $j\in\N$,
\begin{multline}
\left| D^j\big(y^{-2s}\widetilde\phi_R(x,y)\big) \right| =
\left| c_{n,s}\,D^j\int_{B_R}\frac{\phi(z)}{\big(|x-z|^2+y^2\big)^{n/2+s}}\;dz \right| \ \leq \\
\leq\ C \int_{B_R}\frac{|\phi(z)|}{\big(|x-z|^2+y^2\big)^{n/2+s+j/2}}\;dz
\end{multline}
and, if $j=k$ (where $k$ is the one in \eqref{weighted-Linfty}), then
\begin{align*}
& \left| D^k\big(y^{-2s}\widetilde\phi_R(x,y)\big) \right| \leq
C\int_{\Rn}\frac{|\phi(z)|}{\big(|x-z|^2+y^2\big)^{(k+s+\alpha)/2}}\frac{dz}{\big(|x-z|^2+y^2\big)^{(n+s-\alpha)/2}} \\
& \leq\ CC_0\int_{\Rn}\frac{\big(1+|z|\big)^{k+s+\alpha}}{\big(|x-z|^2+y^2\big)^{(k+s+\alpha)/2}}\frac{dz}{\big(|x-z|^2+y^2\big)^{(n+s-\alpha)/2}} \\
& \leq\ \frac{CC_0}{y^{2s+k}}      
\int_{\Rn}\frac{\big(1+|x+y\zeta|\big)^{k+s+\alpha}}{\big(1+|\zeta|^2\big)^{(k+s+\alpha)/2}}\frac{d\zeta}{\big(1+|\zeta|^2\big)^{(n+s-\alpha)/2}} 
 \leq CC_0\left(y^{-2s-k}\big(1+|x|\big)^{k+s+\alpha}+y^{\alpha-s}\right)
\end{align*}
where the constant $C$ is independent of $R$.
Let us also denote by $Q_R$ the Taylor polynomial of degree $k-1$ of $y^{-2s}\widetilde\phi_R(x,y)$ centred at $(x,y)=(0,1)$. 
The difference $y^{-2s}\widetilde\phi_R(x,y)-Q_R(x,y)$ of course satisfies the same
estimate as above and moreover
\begin{align}
& D^j\big|_{(0,1)}\big(y^{-2s}\widetilde\phi_R(x,y)-Q_R(x,y)\big)=0
\qquad\text{whenever }0\leq j\leq k-1  \label{8989988999}  \\
& \big|y^{-2s}\widetilde\phi_R(x,y)-Q_R(x,y)\big|\leq CC_0\big(y^{-2s}(1+|x|)^{2k+\alpha+s}+|y|^{k+\alpha-s}\big)
\quad x\in\Rn,\ y>0, \nonumber
\end{align}
where again the value of $C$ is independent of $R$.

We now claim to have
\begin{align}\label{466466444}
\left\lbrace\begin{aligned}
\div\big[y^{1-2s}\nabla\big(\widetilde\phi_R(x,y)-y^{2s}Q_R(x,y)\big)\big] & = 0 & & \text{in }\Rn\times(0,\infty), \\
\widetilde\phi_R(x,y)-y^{2s}Q_R(x,y) & = \phi_R(x) & & \text{on }\Rn\times\{0\}.
\end{aligned}\right.
\end{align}
In order to justify \eqref{466466444} we only need to verify
\begin{align*}
\div\big[y^{1-2s}\nabla\big(y^{2s}Q_R(x,y)\big)\big]=0
\qquad\text{in }\Rn\times(0,\infty), \text{ for any }R>0.
\end{align*}
Let us first notice that this equality at the point $(0,1)$ because $Q_R$ is a Taylor polynomial based at that point
(\textit{cf}. \eqref{8989988999}).
In a small neighbourhood of $(0,1)$ the same must be true, because the remainder term in the Taylor expansion
is always lower order with respect to $Q_R$, so no cancellation is in order.
Then, the equality extends to the full $\Rn\times(0,\infty)$ by unique continuation of harmonic polynomials.

We send $R\uparrow\infty$ and, using the uniform bounds above and the elliptic estimates entailed by the equation,
deduce the existence of some $\widetilde\phi:\Rn\times[0,\infty)$
satisfying
\begin{align*}
\left\lbrace\begin{aligned}
\div\big(y^{1-2s}\nabla\widetilde\phi(x,y)\big) & = 0 & & \text{in }\Rn\times(0,\infty) \\
\widetilde\phi(x,y) & = \phi(x) & & \text{on }\Rn\times\{0\} \\
|\widetilde\phi(x,y)| & \leq CC_0\big((1+|x|)^{2k+\alpha+s}+|y|^{k+\alpha+s}\big) & & \text{in }\Rn\times(0,\infty).
\end{aligned}\right.
\end{align*}
Moreover,
\begin{align*}
\Ds\phi_R
&= -a_{n,s}\big(y^{1-2s}\partial_y\widetilde \phi_R \big)\big|_{y=0} \\
&= -a_{n,s}\big(y^{1-2s}\partial_y(\widetilde \phi_R -y^{2s} Q_R)\big)\big|_{y=0}
-a_{n,s}\big(y^{1-2s}\partial_y (y^{2s} Q_R ) \big)\big|_{y=0}
\end{align*}
where $y^{1-2s}\partial_y (y^{2s} Q_R)$ is a polynomial of degree at most $k-1$ and
\begin{align*}
\big(y^{1-2s}\partial_y(\widetilde \phi_R -y^{2s} Q_R)\big)\big|_{y=0} 
\longrightarrow \big(y^{1-2s}\partial_y\widetilde \phi\big)\big|_{y=0}
\qquad \text{in } L^\infty_{loc}(\Omega),\text{ as }R\uparrow\infty.
\end{align*}
\end{proof}

\begin{remark}
Clearly, $\widetilde\phi$ might be suitably modified by adding harmonic polynomials 
with trivial trace on $\Rn\times\{0\}$, so that the notion of harmonic extension
is not unambiguously determined.
\end{remark}

\subsection{Limiting problems}

In the following result we will denote by $\mathfrak{L}=\mathfrak{L}(\lambda,\Lambda,s,k)$ the set of all operators $L$ of the form \eqref{L}-\eqref{L2} such that $K\in C^k(\Sn)$.

\begin{lemma}\label{limit with growth}
Let ${(u_m)}_{m\in\N}\subseteq C(\Rn)$ be such that
\begin{align*}
L_m u_m = f_m + P_m, &
\qquad\text{in } B_1,\ f_m\in L^\infty(B_1),\ P_m\in\mathbf{P}_{k-1},\ L_m\in\mathfrak{L}, \\
\int_{\Rn}\frac{|u_m(y)|}{1+|y|^{n+2s+k}}\;dy<\infty. &
\end{align*}
Suppose that there exist $u\in C(\Rn),\ f\in L^\infty(B_1),\text{ and }L\in\mathfrak{L}$ 
such that, as $m\uparrow\infty$,
\begin{align*}
& u_m\to u \text{ in }L^\infty_{loc}(\Rn),
\quad f_m\to f \text{ in }L^\infty(B_1), 
\quad K_m\to K \text{ in }C^k(\Sn), \\
& \text{and }\int_{\Rn}\frac{|u_m(y)-u(y)|}{1+|y|^{n+2s+k}}\;dy\to 0.
\end{align*}
Then
\begin{align*}
Lu \overset{k}{=} f 
\quad \text{in }B_1
\end{align*}
in the sense of Definition \ref{def:Lgrowth}.
\end{lemma}
\begin{proof}
This result is the counterpart of \cite{growth}*{Theorem 1.6}.

By definition, we have $Lu\overset{k}{=}f$ in $B_1$ if there exists a family of polynomials
${(p_R)}_{R>0}\subseteq\mathbf{P}_{k-1}$
and a family of functions ${(g_R)}_{R>0}\subseteq L^\infty(B_1)$ such that
\begin{align*}
& L(u\chi_{B_R}) = g_R+p_R
\qquad\text{in }B_1, \text{ for any }R>1 \\
\text{and }\ & \lim_{r\uparrow\infty}\|g_R-f\|_{L^\infty(\Omega)} = 0.
\end{align*}
For $R>2$, by the convergence $u_m\to u$ in $L^\infty(B_R)$ and $K_m\to K$ in $C^k(\Sn)$ as $m\uparrow\infty$,
we have 
\begin{align*}
L_m(u_m\chi_{B_R})\to L(u\chi_{B_R})
\qquad\text{in }B_1,\ \text{ as }m\uparrow\infty
\end{align*}
in the distributional sense.
Therefore,
\begin{align*}
L(u\chi_{B_R})=\lim_{m\uparrow\infty}\big(f_m+P_m-L_m(u_m\chi_{\Rn\setminus B_R})\big)
=f+\lim_{m\uparrow\infty}\big(P_m-L_m(u_m\chi_{\Rn\setminus B_R})\big)
\end{align*}
and let us notice that, since $P_m\in\mathbf{P}_{k-1}$ for any $m\in\N$,
then for any $\gamma\in\N^n$, $|\gamma|=k$, for $x\in B_1$ we have
\begin{align*}
& \partial^\gamma\big(P_m-L_m(u_m\chi_{\Rn\setminus B_R})\big)(x) =
-\partial^\gamma \int_{\Rn\setminus B_R}u_m(y)\,K_m(x-y)\;dy= \\
& = -\int_{\Rn\setminus B_R}u_m(y)\,\partial^\gamma_x K_m(x-y)\;dy 
\ \longrightarrow\ -\int_{\Rn\setminus B_R}u(y)\,\partial^\gamma_x K(x-y)\;dy,
\end{align*}
uniformly as $m\uparrow\infty$.
Integrating the above relation $k$ times, we deduce that there exists $p_R\in\mathbf{P}_{k-1}$ such that, for $x\in B_1$,
\begin{align*}
\lim_{m\uparrow\infty}\big(P_m-L_m(u_m\chi_{\Rn\setminus B_R})\big)(x)=
p_R(x)-\Theta(x)
\end{align*}
with
\begin{align*}
\partial^\gamma\Theta(x)=\int_{\Rn\setminus B_R}u(y)\,\partial^\gamma_x K(x-y)\;dy \quad\text{for }x\in B_1,
\quad\text{and}\quad \partial^\alpha\Theta(0)=0 \text{ for any }\alpha\leq\gamma
\end{align*}
For this reason, for $x\in B_1$,
\begin{align*}
& \lim_{R\uparrow\infty}\Big|\lim_{m\uparrow\infty}\big(P_m-L_m(u_m\chi_{\Rn\setminus B_R})\big)(x)-p_R(x)\Big|\leq \\
& \leq C
\lim_{R\uparrow\infty}\Big(|x|^k\sup_{z\in B_1}\int_{\Rn\setminus B_R}\big| u(y) D^k K(x-y)\big|\;dy\Big) \\
& \leq C\|K\|_{C^k(\Sn)}\lim_{R\uparrow\infty}\sup_{z\in B_1}\int_{\Rn\setminus B_R}\frac{|u(y)|}{{|x-y|}^{n+2s+k}}\;dy \\
& \leq C\|K\|_{C^k(\Sn)}\lim_{R\uparrow\infty}\int_{\Rn\setminus B_R}\frac{|u(y)|}{{(|y|-1)}^{n+2s+k}}\;dy \ =\ 0 
\end{align*}
by dominated convergence.
\end{proof}

\subsection{Regularity estimates}

We next establish some regularity estimates for functions with polynomial growth.
They will essentially follow from the following.

\begin{lemma}\label{lem: u tilde}
Let $L$ be an operator as in \eqref{L}-\eqref{L2}.
Assume that $U\subset B_1\subset\Rn$ is a
$C^\beta$ domain, $\beta>1$.
Consider $f\in L^\infty(U)$,  
and assume to have a solution of
\begin{align*}
\left\lbrace\begin{aligned}
Lu & \stackrel{k}{=} f & & \text{in } U \\
\frac{|u(x)|}{1+|x|^{k+s+\alpha}} & \in L^\infty(\Rn) & & \alpha<s. 
\end{aligned}\right.
\end{align*}
Then
\begin{align*}
\widetilde u := u\chi_{B_2}
\end{align*}
satisfies 
\begin{align}\label{utilde}
L \widetilde u = \widetilde f \qquad\text{in}\quad U
\end{align}
with 
\begin{align*}
\big\| \widetilde f \big\|_{L^\infty(U)} \leq C\left(\|f\|_{L^\infty(U)}+\Big\|\frac{u}{1+|x|^{k+s+\alpha}}\Big\|_{L^\infty(\Rn)}\right)
\end{align*}
for some $C>0$ depending only on $n,s,U,\lambda$, and $\Lambda$.
\end{lemma}

\begin{proof}
By definition we have that there exist ${(f_R)}_{R>1}\subset L^\infty(U)$
and ${(p_R)}_{R>1}\subseteq\mathbf{P}_{k-1}$ such that
\begin{align}\label{trunc}
& L(u\chi_{B_R})=f_R+p_R\qquad\text{in }U \text{ for any }R>1, \\
& \lim_{R\uparrow\infty}\|f_R-f\|_{L^\infty(U)}=0.
\end{align}
Let us define $\widetilde{u}:=u\chi_{B_2}$. Then
\begin{align*}
L\widetilde{u}=-L(u\chi_{B_R\setminus B_2})+f_R+p_R
\qquad \text{in }U.
\end{align*}
Let us remark that, for every multi-index $\gamma$, $|\gamma|=k$,
\begin{align*}
\partial^\gamma\big[ L(u\chi_{B_R\setminus B_2}) \big](x)
=-\partial^\gamma\int_{B_R\setminus B_2}u(z)\,K(z-x)\;dz
=-\int_{B_R\setminus B_2}u(z)\,\partial^\gamma K(z-x)\;dz,
\quad x\in U,
\end{align*}
and therefore
\begin{align}\label{3543143131213}
\left| D^k\big[ L(u\chi_{B_R\setminus B_2}) \big](x) \right|
\leq C \int_{\Rn\setminus B_2}\frac{|u(z)|}{{|z-x|}^{n+2s+k}}
\leq C\left\|\frac{u}{1+|x|^{k+s+\alpha}}\right\|_{L^\infty(\Rn)},
\qquad x\in U.
\end{align}
From this and \eqref{trunc} we deduce, for $x\in U$ and $|h|\leq 2\dist(x,\partial U)/k$,
\footnote{Recall the finite difference operator as defined in \eqref{def:finite difference}.}
\begin{align}\label{bound on Lutilda}
\begin{split}
\left\|\Delta^k_h(L\widetilde{u})(x)\right\|_{L^\infty(U)} 
& \leq \left\|\Delta^k_h[L(u\chi_{B_R\setminus B_2})](x)\right\|_{L^\infty(U)}
+ \|\Delta^k_h f_R\|_{L^\infty(U)} + \|\Delta^k_h p_R\|_{L^\infty(U)} \\
& \leq C\left\|\frac{u}{1+|x|^{k+s+\alpha}}\right\|_{L^\infty(\Rn)}|h|^k
+ 2\|f\|_{L^\infty(U)}\sum_{j=0}^k\binom{k}{j}
\end{split}
\end{align}
at least for $R>0$ large enough (in such a way that 
$\|f_R\|_{L^\infty(U)}\leq2\|f\|_{L^\infty(U)}$).
Mind that, here, we have also used that $\Delta_h^k p_R \equiv 0$ as $p_R$ has degree at most $k-1$ by assumption.

From \eqref{bound on Lutilda} we deduce that 
there exist $\widetilde g\in L^\infty(U)$, 
and a polynomial $\widetilde p\in\mathbf{P}_{k-1}$
such that 
\begin{align*}
\left\lbrace\begin{aligned}
L\widetilde u &= \widetilde g + \widetilde p  && \text{in }U \\
\widetilde u  &= 0  && \text{in } \Rn\setminus B_2.
\end{aligned}\right.
\end{align*}
We split $\widetilde u=\widetilde u_1+\widetilde u_2$ by setting
\begin{align*}
\left\lbrace\begin{aligned}
L\widetilde u_1 &= \widetilde g  && \text{in }U \\
\widetilde u_1  &= \widetilde u  && \text{in } \Rn\setminus U 
\end{aligned}\right.
\qquad\text{and}\qquad
\left\lbrace\begin{aligned}
L\widetilde u_2 &= \widetilde p  && \text{in }U \\
\widetilde u_2  &= 0  && \text{in } \Rn\setminus U.
\end{aligned}\right.
\end{align*}
Remark that we have
\begin{align*}
\|\widetilde g\|_{L^\infty(U)}\leq 
C\left(\|f\|_{L^\infty(U)}+\left\|\frac{u}{1+|x|^{k+s+\alpha}}\right\|_{L^\infty(\Rn)}\right)
\end{align*}
by construction. This entails also
\begin{align*}
\|\widetilde u_1\|_{L^\infty(U)}\leq 
C\left(\|f\|_{L^\infty(U)}+\left\|\frac{u}{1+|x|^{k+s+\alpha}}\right\|_{L^\infty(\Rn)}\right)
\end{align*}
by standard elliptic estimates. Therefore
\begin{align*}
\|\widetilde u_2\|_{L^\infty(U)}\leq 
\|\widetilde u\|_{L^\infty(U)}+\|\widetilde u_1\|_{L^\infty(U)}\leq 
C\left(\|f\|_{L^\infty(U)}+\left\|\frac{u}{1+|x|^{k+s+\alpha}}\right\|_{L^\infty(\Rn)}\right).
\end{align*}
Thanks to Lemma \ref{lem:tildep} below, this implies that 
\begin{align*}
\|\widetilde p\|_{L^\infty(U)}\leq 
C\left(\|f\|_{L^\infty(U)}+\left\|\frac{u}{1+|x|^{k+s+\alpha}}\right\|_{L^\infty(\Rn)}\right),
\end{align*}
and the result follows. 
\end{proof}

\begin{lemma}\label{lem:tildep}
Let $D\subset\Rn$ be a bounded $C^\beta$ domain, $\beta>1$, and $Q\in\mathbf{P}_\ell$,  $\ell\in\N$.
Let $u$ be the only bounded solution of 
\begin{align*}
\left\lbrace\begin{aligned}
Lu &= Q & &  \text{in }D \\
u  &= 0 & &  \text{in }\Rn\setminus D.
\end{aligned}\right.
\end{align*}
Then there exists $C=C(n,\ell,D,\lambda,\Lambda)>0$ such that 
\begin{align*}
\|u\|_{L^\infty(D)}\geq C\|Q\|_{L^\infty(D)}.
\end{align*}
\end{lemma}
\begin{proof}
Suppose that there are sequences ${(L_k)}_{k\in\N}\subseteq\mathfrak{L},\ {(v_k)}_{k\in\N}\subseteq L^\infty(D),\ 
{(Q_k)}_{k\in\N}\subseteq\mathbf{P}_\ell$, satisfying
\begin{align*}
\|Q_k\|_{L^\infty(D)}=1
\qquad\text{and}\qquad
\left\lbrace\begin{aligned}
L_kv_k &= Q_k & & \text{in } D \\
v_k &= 0 & & \text{in } \Rn\setminus D \\
\lim_{k\uparrow\infty} \|v_k\|_{L^\infty(D)} &= 0.  
\end{aligned}\right.
\end{align*}
We can now extract subsequences ${(L_{k_m})}_{m\in\N},\ {(v_{k_m})}_{m\in\N},\ 
{(Q_{k_m})}_{m\in\N}$ in such a way that, as $m\uparrow\infty$,
\begin{align*}
L_{k_m} \to L  & \qquad \text{ weakly} \\
v_{k_m} \to v  & \qquad \text{ in }L^\infty(D) \\
Q_{k_m} \to Q  & \qquad \text{ in }L^\infty(D) 
\end{align*}
In particular, for the convergence of $L_{k_m}$ we can use \cite{jde}*{Lemma 3.1}, whereas
for that of $v_{k_m}$ we need \cite{annali}*{Theorem 1.2} plus the Ascoli-Arzel\`a Theorem; 
the convergence of $Q_{k_m}$ simply follows from its boundedness in a finite-dimensional space.
We now apply \cite{jde}*{Lemma 3.1} and we have
\begin{align*}
\|Q\|_{L^\infty(D)}=1
\qquad\text{and}\qquad
\left\lbrace\begin{aligned}
Lv &= Q & & \text{in }\Omega \\
v &= 0 & & \text{in }\Rn\setminus\Omega, 
\end{aligned}\right.
\end{align*}
but at the same time $\|v\|_{L^\infty(D)}=0$, a contradiction.
\end{proof}

As a consequence of Lemma \ref{lem: u tilde}, we deduce the following.

\begin{proposition}[Boundary regularity]\label{prop:growth-bound-reg}
Let $L$ be an operator as in \eqref{L}-\eqref{L2}.
Assume that $\Omega$ is a domain of class $C^{1,\gamma}$, $\gamma>0$. 
Consider $f\in L^\infty(\Omega)$ and assume to have a solution of
\begin{align*}
\left\lbrace\begin{aligned}
Lu & \stackrel{k}{=} f & & \text{in }\Omega\cap B_1 \\
u & = 0 & & \text{in }B_1\setminus \Omega \\
\frac{|u(x)|}{1+|x|^{k+s+\alpha}} & \in L^\infty(\Rn) & & \alpha<s. 
\end{aligned}\right.
\end{align*}
Then there exists $C>0$ depending only on $n,s,\alpha,k$, and $\Omega$, such that
\begin{align*}
\|u\|_{C^s(\overline B_{1/2})}\leq C\left(\|f\|_{L^\infty(\Omega\cap B_1)}+\left\|\frac{u}{1+|x|^{k+s+\alpha}}\right\|_{L^\infty(\Rn)}\right).
\end{align*}
Furthermore, the same result holds if $\|f\|_{L^\infty(\Omega\cap B_1)}$ is replaced by $\|d^{s-\varepsilon}f\|_{L^\infty(\Omega\cap B_1)}$, with $\varepsilon>0$.
\end{proposition}

\begin{proof}
Remark that $u=\widetilde u:=u\chi_{B_2}$ in $B_{1/2}$, so that it is sufficient 
to estimate $\widetilde u$.
Applying the $C^s$ regularity estimates to problem \eqref{utilde}, see \cite{annali}*{Theorem 1.2},
we deduce 
\begin{align*}
\|\widetilde u\|_{C^s(B_{1/2})}\leq C\big(\|\widetilde f\|_{L^\infty(\Omega\cap B_1)}+\|u\|_{L^\infty(B_2)}\big).
\end{align*}
Now, by Lemma \ref{lem: u tilde}, we know that
\begin{align*}
\|\widetilde f\|_{L^\infty(\Omega\cap B_1)}+\|u\|_{L^\infty(B_2)}
\leq
C\Big(\|f\|_{L^\infty(\Omega)}+\left\|\frac{u}{1+|x|^{k+s+\alpha}}\right\|_{L^\infty(\Rn)}\Big),
\end{align*}
and thus the result follows.

The case in which $|f|\leq Cd^{\varepsilon-s}$ is analogous, using \cite{annali}*{Proposition~3.1} instead of \cite{annali}*{Theorem~1.2}.
\end{proof}

Finally, we also prove interior estimates.

\begin{proposition}[Interior regularity]\label{prop:growth-int-reg}
Let $L$ be an operator as in \eqref{L}-\eqref{L2}.
Let $f\in C^{\eta-s}(\overline B_1)$, $\eta>s$, $\eta+s\not\in\N$, and $k\in\N$.
Let $u\in L^\infty(\Rn)$ be a solution of
\begin{align*}
Lu = f \qquad \text{in }B_1.
\end{align*}
Then, for some $C=C(n,s,\eta)>0$
\begin{align*}
\|u\|_{C^{\eta+s}(\overline B_{1/2})} \leq 
C\Big(\Big\|\frac{u}{1+|x|^{\eta+s}}\Big\|_{L^\infty(\Rn)}+\big[f\big]_{C^{\eta-s}(\overline B_1)}\Big).
\end{align*}
\end{proposition}

\begin{proof}
Remark that $u=\widetilde u:=u\chi_{B_2}$ in $B_{1/2}$, so that it is sufficient 
to estimate $\widetilde u$. As we have done in \eqref{3543143131213}, we can show that
\begin{align*}
\big[L\widetilde u\big]_{C^{\eta-s}(\overline B_1)}
\leq \big[L(u\chi_{\Rn\setminus B_2})\big]_{C^{\eta-s}(\overline B_1)} + \big[f\big]_{C^{\eta-s}(\overline B_1)} 
\leq C\Big\|\frac{u}{1+|x|^{\eta+s}}\Big\|_{L^\infty(\Rn)} + \big[f\big]_{C^{\eta-s}(\overline B_1)}.
\end{align*}
It suffices now to apply the interior Schauder estimates to $\widetilde u$, see \cite{jde}*{Theorem 1.1}.
\end{proof}

\subsection{The Liouville theorem in a half-space}

In our higher order blow-up and compactness argument we also need the following classification result.

\begin{theorem}\label{liouville-nd}
Let $e\in\Sn$ be fixed. 
Let $u$ satisfy
\begin{align*}
\left\lbrace\begin{aligned}
Lu & \stackrel{k}{=} 0 & & \text{in }\{x\cdot e>0\} \\
u & = 0 & & \text{in }\{x\cdot e\leq 0\} \\
|u(x)| & \leq C_0\big(1+|x|\big)^{k+s+\alpha} & & \text{in }\Rn,\ \alpha<s. 
\end{aligned}\right.
\end{align*}

Then, $u$ is of the form
\begin{align}\label{1D}
u(x)=p(x)(x\cdot e)_+^s
\end{align}
for some polynomial $p\in\mathbf{P}_k$.
\end{theorem}

First, we need the following one-dimensional version of the result.

\begin{proposition}\label{lem:Liouville-1D}
Let $u:\R\to\R$ satisfy
\begin{align*}
\left\lbrace\begin{aligned}
\Ds u & \stackrel{k}{=} 0 & & \text{in }\{x>0\} \\
u & = 0 & & \text{in }\{x\leq 0\} \\
|u(x)| & \leq C_0\big(1+|x|\big)^{k+s+\alpha} & & \text{in }\R,\ \alpha<s. 
\end{aligned}\right.
\end{align*}
Then there exists a polynomial $p:\R\to\R$ of degree at most $k$ such that
\begin{align*}
u(x)=p(x)x_+^s.
\end{align*}
\end{proposition}

\begin{proof}
Let $U:\R\times[0,\infty)\to\R$ be a harmonic extension of $u$ in the sense of Lemma \ref{lem:cs-extension}
satisfying
\begin{align*}
\left\lbrace\begin{aligned}
\div\big(y^{1-2s}\nabla U(x,y)\big) & = 0 & & \text{in }\R\times(0,\infty) \\
U(x,y) & = u(x) & & \text{on }\R\times\{0\} \\
|U(x,y)| & \leq CC_0\big(1+|x|^{2k+\alpha+s}+|y|^{k+\alpha+s}\big)  & & \text{in }\R\times(0,\infty).
\end{aligned}\right.
\end{align*}
We now exploit \cite{duke}*{Lemma 6.1 and (the proof of) Lemma 6.2} to write $U$ as
\begin{align*}
U(x,y)=U(r\cos\theta,r\sin\theta)=\sum_{j=0}^\infty a_j\Theta_j(\theta)\, r^{j+s},
\qquad a_j\in\R,\ x\in\R,\ y\in[0,\infty),
\end{align*}
where ${(\Theta_j)}_{j\in\N}$ is a complete orthogonal system in 
$L^2\big((0,\pi),(\sin\theta)^{1-2s}d\theta\big)$ and therefore
\begin{align*}
\int_{\partial B_R\cap\{y>0\}}U(x,y)^2\,  y^{1-2s}\;d\sigma
=\sum_{j=0}^\infty a_j^2\, R^{2+2j}.
\end{align*}
The known polynomial bound on $U$ yields
\begin{align*}
\int_{\partial B_R\cap\{y>0\}}U(x,y)^2\,  y^{1-2s}\;d\sigma
\leq CR^{4k+2\alpha+2}
\end{align*}
from which we deduce that $a_j=0$ for any $j\geq 2k+1$.
This entails
\begin{align*}
U(x,y)=\sum_{j=0}^{2k} a_j\Theta_j(\theta)\, r^{j+s},
\qquad x\in\R, y\in[0,\infty)
\end{align*}
and, in particular,
\begin{align*}
u(x)=\left\lbrace\begin{aligned}
& \sum_{j=0}^{2k} a_j\Theta_j(0)\, x^{j+s}=x^s\,\sum_{j=0}^{2k}a_j\Theta_j(0)x^j
& & \text{for } x>0, \\
& \sum_{j=0}^{2k} a_j\Theta_j(\pi)\, |x|^{j+s}=0
& & \text{for } x<0.
\end{aligned}\right.
\end{align*}
Similarly as above, the polynomial bound on $u$ gives that $a_j=0$ also for $j\in\{k+1,\ldots,2k\}$,
and this induces the claimed representation on $u$,
concluding the proof.
\end{proof}

We can now give the proof of the Liouville-type theorem.

\begin{proof}[Proof of Theorem \ref{liouville-nd}]
Define
\begin{align*}
v_R(x)=R^{-k-s-\alpha}u(Rx),
\qquad x\in\Rn,\ R>1.
\end{align*}
Then
\begin{align*}
\left\|\frac{v_R(x)}{\big(1+|x|\big)^{k+s+\alpha}}\right\|_{L^\infty(\Rn)}
\leq C_0 
\end{align*}
and 
\begin{align*}
\left\lbrace\begin{aligned}
Lv_R & \stackrel{k}{=} 0 & & \text{in }\{x\cdot e>0\} \\
v_R & = 0 & & \text{in }\{x\cdot e\leq 0\}.
\end{aligned}\right.
\end{align*}
Applying Proposition \ref{prop:growth-bound-reg} to $v_R$ yields that
$\|v_R\|_{C^s(B_{1/2})}\leq c_0$, with $c_0>0$ independent of $R$, which in turn implies
that $[u]_{C^s(B_{R/2})}=R^{k+\alpha}[v_R]_{C^s(B_{1/2})}\leq c_0R^{k+\alpha}$.
From now on we suppose $e=e_n$.

Pick now any $\tau\in\Sn$ such that $\tau_n=0$ and $h\in(0,R/2)$.
Consider 
\begin{align*}
w_1(x)=\frac{u(x+h\tau)-u(x)}{h^s},
\qquad x\in\Rn.
\end{align*}
The above analysis gives
\begin{align*}
\|w_1\|_{L^\infty(B_R)}\leq c_0R^{k+\alpha},
\qquad\text{for any }R>1
\end{align*} 
and, since $\tau$ is orthogonal to $e$,
\begin{align*}
\left\lbrace\begin{aligned}
Lw_1 & \stackrel{k}{=} 0 & & \text{in }\{x_n>0\}, \\
w_1 & = 0 & & \text{in }\{x_n\leq 0\}.
\end{aligned}\right.
\end{align*}
Repeating the same argument as in the first part of the proof,
we deduce that $[w_1]_{C^s(B_R)}\leq c_1R^{k+\alpha-s}$.

Iterating the above scheme a finite number of times, 
we will eventually end up with some $w_j$ satisfying
\begin{align*}
\left\lbrace\begin{aligned}
Lw_j & \stackrel{k}{=} 0 & & \text{in }\{x_n>0\}, \\
w_j & = 0 & & \text{in }\{x_n\leq 0\}.
\end{aligned}\right.
\end{align*}
and $[w_j]_{C^s(B_R)}\leq c_jR^{k+\alpha-js}$,
with $k+\alpha< js$.
Letting $R\to\infty$, this entails that $w_j\equiv 0$ in $\R^n$, regardless the choice of $\tau$, as long as $\tau_n=0$. 
This means that
\begin{align*}
w_{j-1}(x)=\widetilde W_1(x_n)
\end{align*}
for some $\widetilde W_1:\R\to\R$. 
In turn, this gives that 
\begin{align*}
w_{j-2}(x+h\tau)-w_{j-2}(x)=\widetilde W_1(x_n)\,h^s,
\qquad x\in\R^n,\ h>0,\ \tau_n=0,
\end{align*}
which implies\footnote{In general, if $f\in C(\R)$ satisfies $f(x+h)-f(x)=c_h$ for every $x\in\R$ and $h>0$,
then $f$ is an affine function.}
\begin{align*}
w_{j-2}(x)=\widetilde W_2(x_n)\cdot x' + \widetilde W_3(x_n)
\end{align*}
for some $\widetilde W_2:\R\to\R^{n-1}$, $\widetilde W_3:\R\to\R$.
Iterating the process, what we deduce on $u$ is that
\begin{align*}
u(x)=\sum_{\alpha\in\N^{n-1},\, |\alpha|\leq j-1} (x')^\alpha\, W_\alpha(x_n).
\end{align*}

Now, notice that for every $\alpha\in\N^{n-1}$ with $|\alpha|\leq j-1$ we have
\begin{align*}
\frac{1}{\alpha!}\,\partial_{x'}^{\alpha}u(x) = W_\alpha(x_n),
\end{align*}
and a similar identity can be written in terms of incremental quotients of $u$.
Then, since $Lu\stackrel{k}{=}0$ in $\{x_n>0\}$, it is not difficult to see that $L W_\alpha\stackrel{k}{=}0$ in $\{x_n>0\}$.
Since $W_\alpha$ is a one-dimensional function, \cite{duke}*{Lemma 2.1} yields that
$\Ds W_\alpha\stackrel{k}{=}0$ in $(0,\infty)$.

Finally, by Proposition \ref{lem:Liouville-1D} each of the $W_\alpha$ must be of the form $W_\alpha(x_n)=p_\alpha(x_n)(x_n)_+^s$ for some polynomial $p_\alpha:\R\to\R$, and therefore $u$ must be of the form $u(x)=p(x)(x_n)_+^s$, for some polynomial $p$.
By the growth condition on $u$, $p$ must be of degree at most $k$, and the theorem is proved.
\end{proof}

\section{Higher order boundary Schauder and boundary Harnack estimates}
\label{sec:uover}

The goal of this section is to prove Theorems \ref{thm:uoverds} and \ref{thm:u1overu2}.
For this, we develop a higher order blow-up and compactness argument that allows us for the first time to show sharp boundary regularity results for nonlocal equations in $C^\beta$ domains.

The key step towards the proof of Theorems \ref{thm:uoverds} is the following.

\begin{proposition}\label{uoverds}
Let $\beta>s$ be such that $\beta\not\in\N$ and $\beta\pm s\not\in\N$.
Let $\Omega\subseteq\Rn$ be a bounded domain of class
$C^{\beta+1}$, $z\in\partial\Omega$, and $u\in L^\infty(\Rn)$ any solution of 
\begin{align*}
\left\lbrace\begin{aligned}
Lu &= f & & \hbox{in }\Omega\cap B_1(z) \\
u &= 0 & & \hbox{in }B_1(z)\setminus\Omega
\end{aligned}\right.
\end{align*}
with $[f]_{C^{\beta-s}(\overline\Omega)}\leq 1$ and $\|u\|_{L^\infty(\R^n)}\leq 1$.
Suppose that $\partial\Omega\cap B_1(z)$ is the graph of a function with $C^{\beta+1}$ norm less than $1$.

Then, for any $z\in\partial\Omega\cap B_{1/2}$, there exists a $Q(\cdot,z)\in\mathbf{P}_{\intpart\beta}$ such that
\begin{align*}
\big|u(x)-Q(x,z)d^s(x)\big| \leq C{|x-z|}^{\beta+s},
\qquad\text{for any }x\in B_1(z),
\end{align*}
where $C>0$ depends only on $n,s,\beta,$ and $\|K\|_{C^{2\beta+3}(\Sn)}$.

Moreover, if $x_0\in\Omega\cap B_1(z)$, $d(x_0)=2r=|z-x_0|>0$, then
\begin{align}\label{calpha}
\big[u-Q(\cdot,z)d^s\big]_{C^{\beta+s}(\overline B_r(x_0))}\leq C.
\end{align}
\end{proposition}

\begin{remark}
The case $\beta\in(0,s)$ in Theorem \ref{thm:uoverds} is covered by
\cite{annali}*{Theorem 1.2}.
\end{remark}

\begin{proof}[Proof of Proposition \ref{uoverds}]
We assume without loss of generality that $0\in\partial\Omega$ and $z=0$.

We argue by contradiction: suppose that, for any $j\in\N$, there exists 
$\Omega_j\subseteq\Rn,u_j\in L^\infty(\Rn),f_j\in C^{\beta-s}(\overline\Omega_j),r_j>0$, and $L_j\in\mathfrak L$
such that
\begin{align*}
\left\lbrace\begin{aligned}
L_ju_j &= f_j & & \hbox{in }\Omega_j\cap B_1,  \\
u_j &= 0 & & \hbox{in }B_1\setminus\Omega_j,
\end{aligned}\right.
\end{align*}
with $\|f_j\|_{C^{\beta-s}(\overline\Omega)}+\|u_j\|_{L^\infty(\R^n)}+\|K_j\|_{C^{\beta+2}(\Sn)}\leq C_0$ and $0\in\partial\Omega_j\in C^{\beta+1}$;
moreover,
\begin{align*}
\sup_{j\in\N}\sup_{r>0}r^{-\beta-s}\|u_j-Qd_j^s\|_{L^\infty(B_r)}=\infty,
\qquad\text{for any }Q\in\textbf{P}_{\intpart{\beta}}.
\end{align*}
Let us consider $Q_{j,r}\in\textbf{P}_{\intpart{\beta}}$ as the polynomial obtained
upon taking the $L^2(B_r)$-projection
of $u_j$ over $d_j^s\textbf{P}_{\intpart{\beta}}$: in particular,
\begin{align*}
\| u_j - Q_{j,r} d_j^s \|_{L^2(B_r)} &\leq \| u_j - Qd_j^s \|_{L^2(B_r)} 
& \quad\text{for any }Q\in\textbf{P}_{\intpart{\beta}}, \\
\int_{B_r} \big( u_j - Q_{j,r}d_j^s \big) \, Qd_j^s &= 0
& \quad\text{for any }Q\in\textbf{P}_{\intpart{\beta}}.
\end{align*}
Define the monotone quantity
\begin{align*}
\theta(r):=\sup_{j\in\N}\sup_{\rho\geq r}\rho^{-\beta-s}\|u_j-Q_{j,\rho}d_j^s\|_{L^\infty(B_\rho)}.
\end{align*}
We have that $\theta(r)\uparrow\infty$ as $r\downarrow0$ ---the proof of which we defer to Lemma \ref{lem:theta explodes}---
and therefore there are sequences 
$(r_m)_{m\in\N}$ and $(j_m)_{m\in\N}$ such that
\begin{align}\label{4654646463465}
\frac{\|u_{j_m}-Q_{j_m,r_m}d_{j_m}^s\|_{L^\infty(B_{r_m})}}{r_m^{\beta+s}\,\theta(r_m)} \geq \frac12.
\end{align} 
Define now
\begin{align*}
v_m(x):=\frac{u_{j_m}(r_mx)-Q_{j_m,r_m}(r_mx)d_{j_m}^s(r_mx)}{r_m^{\beta+s}\,\theta(r_m)} 
\qquad x\in\Rn,\ m\in\N.
\end{align*} 
and notice that $\|v_m\|_{L^\infty(B_1)}\geq 1/2$ and
\begin{align}\label{1654864}
\int_{B_1} v_m(x)\,Q(r_mx)\,d_{j_m}^s(r_mx)\;dx=0,
\qquad m\in\N,\ Q\in\mathbf{P}_{\intpart\beta}.
\end{align}

Write now
\begin{align*}
Q_{j,r}(x)=\sum_{|\alpha|\leq\intpart{\beta}}q_{j,r}^{(\alpha)}x^\alpha,
\qquad\alpha\in\N^n,\ q_{j,r}^{(\alpha)}\in\R.
\end{align*}
Using a rescaled version of Lemma \ref{lem:poly} and that $d_j^s\geq cr^s$ in $B_r\cap\{d_j>r/2\}$, we estimate for any $\alpha$ such that $|\alpha|\leq\intpart{\beta}$
\begin{align*}
& r^{|\alpha|+s}\big|q_{j,r}^{(\alpha)}-q_{j,2r}^{(\alpha)}\big|\ \leq \\
& \leq\ c_\beta \big\|Q_{j,r}d_j^s-Q_{j,2r}d_j^s\big\|_{L^\infty(B_r\cap\{d_j>r/2\})} 
\leq c_\beta \big\|u_j-Q_{j,r}d_j^s\big\|_{L^\infty(B_r)}
+ c_\beta \big\|u_j-Q_{j,2r}d_j^s\big\|_{L^\infty(B_{2r})}  \\
& \leq c_\beta\,\theta(r)r^{\beta+s}+c_\beta\,\theta(2r)(2r)^{\beta+s}
\leq 2c_\beta\,\theta(r)(2r)^{\beta+s}
\end{align*}
so that it holds
\begin{align*}
\big|q_{j,r}^{(\alpha)}-q_{j,2r}^{(\alpha)}\big|
\leq c_\beta' \, \theta(r) \, r^{\beta-|\alpha|},
\qquad\text{for any }|\alpha|\leq\intpart{\beta},\ r>0,\ j\in\N.
\end{align*}
Iterating the inequality above we get, for any $k\in\N$,
\begin{multline*}
\big|q_{j,r}^{(\alpha)}-q_{j,2^kr}^{(\alpha)}\big| \leq \sum_{i=0}^{k-1}\big|q_{j,2^{i}r}^{(\alpha)}-q_{j,2^{i+1}r}^{(\alpha)}\big|
  \leq c_\beta'\sum_{i=0}^{k-1} \theta(2^ir)(2^ir)^{\beta-|\alpha|} \ \leq \\
\leq\ c_\beta' \, \theta(r) \, r^{\beta-|\alpha|}\sum_{i=0}^{k-1} \frac{\theta(2^i r)}{\theta(r)}2^{i(\beta-|\alpha|)}
  \leq c_\beta'' \, \theta(r) \, (2^kr)^{\beta-|\alpha|}.
\end{multline*}
It follows from this that, for any $R>1$,
\begin{align*}
\big|q_{j,r}^{(\alpha)}-q_{j,Rr}^{(\alpha)}\big|\leq c_\beta\theta(r)(Rr)^{\beta-|\alpha|}
\end{align*}
and thus 
\begin{align*}
\big\|Q_{j,Rr}d_{j}^s-Q_{j,r}d_{j}^s\big\|_{L^\infty(B_{Rr})}\leq c_\beta\theta(r)(Rr)^{\beta+s}.
\end{align*}
Hence,
\begin{align*}
& \|v_m\|_{L^\infty(B_R)} 
=\frac1{r_m^{\beta+s}\,\theta(r_m)}\big\|u_{j_m}-Q_{j_m,r_m}d_{j_m}^s\big\|_{L^\infty(B_{Rr_m})} \\
& \leq\ \frac1{r_m^{\beta+s}\,\theta(r_m)}\left(\big\|u_{j_m}-Q_{j_m,Rr_m}d_{j_m}^s\big\|_{L^\infty(B_{Rr_m})} 
		+\big\|Q_{j_m,Rr_m}d_{j_m}^s-Q_{j_m,r_m}d_{j_m}^s\big\|_{L^\infty(B_{Rr_m})}\right) \\
& \leq\ \frac1{r_m^{\beta+s}\,\theta(r_m)}\big(\theta(Rr_m)(Rr_m)^{\beta+s}+c_\beta\theta(r_m)(Rr_m)^{\beta+s}\big)
\leq (1+c_\beta) R^{\beta+s}
\end{align*}
where we recall that, by definition, $\theta$ is monotone decreasing.

Moreover, for each $r>0$ we have
\[\frac{|q_{j,r}^{(\alpha)}-q_{j,2^kr}^{(\alpha)}|}{\theta(r)} \leq \sum_{i=0}^k\frac{\theta(2^{k-i}r)}{\theta(r)}(2^{k-i}r)^{\beta-|\alpha|},\]
and choosing $k\in\mathbb N$ such that $2^kr\in [1,2)$, we deduce
\[\frac{|q_{j,r}^{(\alpha)}-q_{j,2^kr}^{(\alpha)}|}{\theta(r)} \leq C\sum_{i=0}^k\frac{\theta(2^{-i})}{\theta(r)}(2^{-i})^{\beta-|\alpha|}\longrightarrow 0\qquad\textrm{as}\ r\downarrow0.\]
In particular, 
\[\frac{|q_{j,r}^{(\alpha)}|}{\theta(r)}\longrightarrow 0\qquad\textrm{as}\ r\downarrow0.\]

Let us now consider the identity 
\begin{align*}
L_{j_m}v_m(x) 
& =\frac{r_m^{s-\beta}}{\theta(r_m)}\big[L_{j_m}u_{j_m}(r_mx)-L_{j_m}(Q_{j_m,r_m}d^s_{j_m})(r_mx)\big] \\
& =\frac{r_m^{s-\beta}}{\theta(r_m)}\big[f_{j_m}(r_mx)-L_{j_m}(Q_{j_m,r_m}d^s_{j_m})(r_mx)\big]
\qquad\text{for }x\in r_{j_m}^{-1}\Omega_{j_m}=\{y:r_{j_m}y\in\Omega_{j_m}\}.
\end{align*}
As $r_m\downarrow 0$, up to extracting a further subsequence, $r_m^{-1}\Omega_{j_m}$ is converging to 
a half-space $\Pi=\{x\in\Rn:x\cdot e>0, \text{ for some }e\in\Sn\}$. 
Moreover, as we have both $f_{j_m}, L_{j_m}(Q_{j_m,r_m}d_{j_m}^s)\in C^{\beta-s}(\overline\Omega_{j_m})$
---by Theorem \ref{thm:Lds}---,
there exists a polynomial $P_m\in\mathbf{P}_{\intpart{\beta-s}}$
\begin{align*}
r_m^{s-\beta}\big|f_{j_m}(r_mx)- L_{j_m}(Q_{j_m,r_m}d_{j_m}^s)(r_mx)-P_m(r_mx)\big|
\leq  C_0 |x|^{\beta-s}\big([f_{j_m}]_{C^{\beta-s}}+\|Q_{j_m,r_m}\|_{B_1}\big),
\end{align*}
and therefore $|L_{j_m}v_m-P_m|\downarrow 0$ as $m\uparrow\infty$ in $L^\infty_{loc}(\Pi)$.

By Proposition \ref{prop:growth-bound-reg} 
and the Ascoli-Arzel\`a Theorem we deduce that\footnote{Up to extracting a subsequence;
with an abuse of notation we keep $v_m$.} $v_m$ 
is converging in $L^\infty_{loc}(\Rn)$ to some $v\in C(\Rn)$ 
(recall that $v_m=0$ in $\Rn\setminus\Omega_{j_m}$ for every $m\in\N$).
Also, as $\|K_{j_m}\|_{C^{2\beta+3}(\Sn)}$ is uniformly bounded by assumption, 
then (up to passing to a subsequence) $K_{j_m}$ is converging to $K_\star$ in $C^{\intpart{\beta-s}+1}(\Sn)$,
since $\beta>1$. 
In conclusion, as an application of Lemma~\ref{limit with growth} we have that 
\begin{align*}
\left\lbrace\begin{aligned}
L_\star v & \stackrel{k}{=} 0  					& & \text{in }\Pi, \quad  k\geq\intpart{\beta-s}+1, \\
v &= 0											& & \text{in }\Rn\setminus\Pi, \\
\|v\|_{L^\infty(B_R)} &\leq (1+c_\beta)R^{\beta+s},
\end{aligned}\right.
\end{align*}
and moreover it follows from  \eqref{4654646463465} that
\begin{equation}\label{v-contr}
\|v\|_{L^\infty(B_1)}\geq \frac12.
\end{equation}
We are now in the assumptions of Proposition \ref{liouville-nd} (note in particular that $\beta+s<\intpart{\beta-s}+1+2s\leq k+2s$) and therefore 
\begin{align*}
v(x)=p(x\cdot e)(x\cdot e)_+^s,
\qquad x\in\Rn,\ p\in\mathbf{P}_k.
\end{align*}
Actually, the control on $\|v\|_{L^\infty(B_R)}$ yields $p\in\mathbf{P}_{\intpart\beta}$ and 
\begin{align*}
p(t)=\sum_{i=0}^d \pi_i t^i,
\qquad t\in\R,\quad \deg(p)\leq\intpart\beta.
\end{align*}
Let $i_0\in\{0,\ldots,\deg(p)\}$ be the minimum value for which $\pi_i\neq 0$.
Notice at this point that, by \eqref{1654864}, we have in particular that
\begin{align*}
r_m^{-s-i_0}\int_{B_1} v_m(x)\,Q(r_mx)\,d_{j_m}^s(r_mx)\;dx=0,
\qquad m\in\N,\ Q\in\mathbf{P}_{\intpart\beta}.
\end{align*}
Choose in particular
\begin{align*}
Q_m(x)=\sum_{i=0}^d \pi_ir_m^{-i} (x\cdot e)^i.
\end{align*}
Passing to the limit as $m\uparrow\infty$ ($r_m\downarrow 0$), we get
\begin{multline*}
0 = \lim_{m\uparrow\infty}r_m^{-s}\int_{B_1} v_m(x)\,Q_m(r_mx)\,d_{j_m}^s(r_mx)\;dx=
\lim_{m\uparrow\infty}r_m^{-s}\int_{B_1} v_m(x)\,p(x\cdot e)\,d_{j_m}^s(r_mx)\;dx \ = \\
=\ \int_{B_1} p(x\cdot e)^{2}\,(x\cdot e)^{2s}_+\;dx.
\end{multline*}
This yields that $p\equiv 0$ and in turn $v\equiv 0$, too.
This is in contradiction with \eqref{v-contr}, and hence the first part of the Proposition is proved.

We finally show \eqref{calpha}. 
Let
\begin{align*}
v_r(x):=r^{-\beta-s}u(x_0+rx)-r^{-\beta-s}Q(x_0+rx,z)d^s(x_0+rx)
\qquad x\in \Rn.
\end{align*}
The first part of the proof is telling us that
\begin{align*}
\|v_r\|_{L^\infty(B_1)}\leq C.
\end{align*}
By Proposition \ref{prop:growth-int-reg} we have that
\begin{align*}
\big[u-Qd^s\big]_{C^{\beta+s}(\overline B_{r/2}(x_0))}
& =\big[v_r\big]_{C^{\beta+s}(\overline B_{1/2})} \leq C\left(\Big\|\frac{v_r}{1+|x|^{\beta+s}}\Big\|_{L^\infty(\Rn)}+\big[Lv_r\big]_{C^{\beta-s}(\overline B_1)}\right) \\
& \leq C\left(\|u\|_{L^\infty(\Rn)}+\big[Lu\big]_{C^{\beta-s}(\overline B_r(x_0))}+\big[L(Q(\cdot,z)d^s)\big]_{C^{\beta-s}(\overline B_r(x_0))}\right) \\
& \leq C\left(\|u\|_{L^\infty(\Rn)}+\big[f\big]_{C^{\beta-s}(\overline\Omega)}+\big[L(Q(\cdot,z)d^s)\big]_{C^{\beta-s}(\overline B_{1/2})}\right)
\end{align*}
which is finite by Corollary \ref{cor:Lds} ---recall that $\partial\Omega\in C^{\beta+1}$. 
\end{proof}

\begin{lemma}\label{lem:theta explodes}
Let $\beta>0$, $\beta\not\in\N$, $\Omega\subset\Rn$ such that $0\in\partial\Omega$, and $u\in C(\overline{B}_1)$.
If, for any $r\in(0,1)$, $Q_r\in\mathbf{P}_{\intpart\beta}$ satisfies
\begin{align*}
\|u-Q_rd^s\|_{L^2(B_r)}\leq\|u-Qd^s\|_{L^2(B_r)},
\qquad\text{for any }Q\in\mathbf{P}_{\intpart\beta},
\end{align*}
and
\begin{align*}
\|u-Q_rd^s\|_{L^\infty(B_r)}\leq c_0r^{\beta+s}
\end{align*}
then there exists $Q_0\in\mathbf{P}_{\intpart\beta}$ such that
\begin{align*}
\|u-Q_0d^s\|_{L^\infty(B_r)}\leq Cc_0 \, r^{\beta+s},
\qquad r\in(0,1),
\end{align*}
where $C>0$ only depends on $n,s,$ and $\beta$.
\end{lemma}
\begin{proof}
It holds
\begin{align*}
\|Q_rd^s-Q_{2r}d^s\|_{L^\infty(B_r)}\leq \|u-Q_rd^s\|_{L^\infty(B_r)}+\|u-Q_{2r}d^s\|_{L^\infty(B_{2r})}\leq 
(1+2^{\beta+s})c_0r^{\beta+s}.
\end{align*}
In particular,
\begin{align*}
|Q_r(x)d^s(x)-Q_{2r}(x)d^s(x)|\leq
(1+2^{\beta+s})c_0r^{\beta+s},
\qquad x\in\partial B_r,
\end{align*}
which yields, by a rescaled version of Lemma \ref{lem:poly},
\begin{align}\label{qr}
\big|q^{(\alpha)}_r-q^{(\alpha)}_{2r}\big|\leq
Cc_0\,r^{\beta-|\alpha|},
\qquad \alpha\in\N^n,\ |\alpha|\leq\intpart\beta.
\end{align}
Also, by a similar reasoning,
\begin{align*}
\|Q_1d^s\|_{L^\infty(B_1)} \leq c_0+\|u\|_{L^\infty(B_1)}
\end{align*}
implies
\begin{align}\label{q1}
\big|q_1^{(\alpha)}\big|\leq C(c_0+\|u\|_{L^\infty(B_1)}),
\qquad \alpha\in\N^n,\ |\alpha|\leq\intpart\beta.
\end{align}
Since $\beta\not\in\N$, then $\beta-|\alpha|\geq\beta-\intpart\beta>0$,
and this, together with \eqref{qr} and \eqref{q1}, yields the existence of limits
\begin{align*}
q_0^{(\alpha)}=\lim_{r\downarrow 0}q_r^{(\alpha)},
\qquad \alpha\in\N^n,\ |\alpha|\leq\intpart\beta.
\end{align*}
Moreover, using a telescopic series and \eqref{qr}, 
\begin{align*}
\big|q^{(\alpha)}_0-q^{(\alpha)}_{r}\big|\leq\sum_{j=0}^\infty\big|q^{(\alpha)}_{2^{-j}r}-q^{(\alpha)}_{2^{-j-1}r}\big|
\leq Cc_0\sum_{j=0}^\infty (2^{-j}r)^{\beta-|\alpha|}\leq Cc_0\,r^{\beta-|\alpha|}, &
\qquad \alpha\in\N^n,\ |\alpha|\leq\intpart\beta, \\
\|Q_0d^s-Q_rd^s\|_{L^\infty(B_r)}\leq 
C\sum_{|\alpha|\leq\intpart\beta}\big|q^{(\alpha)}_0-q^{(\alpha)}_{r}\big|r^{|\alpha|+s}\leq 
Cc_0 \, r^{\beta+s}.
\end{align*}
and therefore 
\begin{align*}
\|u-Q_0d^s\|_{L^\infty(B_r)}\leq\|u-Q_rd^s\|_{L^\infty(B_r)}+\|Q_rd^s-Q_0d^s\|_{L^\infty(B_r)}\leq
Cc_0r^{\beta+s}.
\end{align*}
\end{proof}

We are now in position to prove Theorem \ref{thm:uoverds}.

\begin{proof}[Proof of Theorem \ref{thm:uoverds}]
Let $\gamma\in\N^n$, $|\gamma|=\intpart\beta$. Let us compute
\begin{align*}
\partial^\gamma\big(d^{-s}u\big)=\sum_{\alpha\leq\gamma}\big(\partial^\alpha u\big)\big(\partial^{\gamma-\alpha}d^{-s}\big)
\end{align*}
Let $r>0$ be fixed and $x_0\in\Omega,z\in\partial\Omega$ be such that $d(x_0)=2r=|x_0-z|$.
Consider $x_1,x_2\in B_r(x_0)\subset\Omega$. Then, for $Q=Q(\cdot,z)\in\mathbf{P}_{\intpart\beta}$
as constructed in Proposition \ref{uoverds},
\begin{align*}
& \partial^\gamma\big(d^{-s}u\big)(x_1)-\partial^\gamma\big(d^{-s}u\big)(x_2)=
\sum_{\alpha\leq\gamma}\big[
\partial^\alpha u(x_1)\,\partial^{\gamma-\alpha}d^{-s}(x_1)
-\partial^\alpha u(x_2)\,\partial^{\gamma-\alpha}d^{-s}(x_2) \big] \\
& =\ \sum_{\alpha\leq\gamma}\big[
\partial^\alpha u(x_1)
-\partial^\alpha u(x_2) \big]\,\partial^{\gamma-\alpha}d^{-s}(x_1)
+ \sum_{\alpha\leq\gamma}\partial^\alpha u(x_2)\,\big[
\partial^{\gamma-\alpha}d^{-s}(x_1)
-\partial^{\gamma-\alpha}d^{-s}(x_2) \big] \\
& =\ \sum_{\alpha\leq\gamma}\big[
\partial^\alpha \big(u-Qd^s\big)(x_1)
-\partial^\alpha \big(u-Qd^s\big)(x_2) \big]\,\partial^{\gamma-\alpha}d^{-s}(x_1) \\
& \qquad +\sum_{\alpha\leq\gamma}\big[
\partial^\alpha \big(Qd^s\big)(x_1)
-\partial^\alpha \big(Qd^s)(x_2) \big]\,\partial^{\gamma-\alpha}d^{-s}(x_1) \\
& \qquad +\sum_{\alpha\leq\gamma}\partial^\alpha \big(u-Qd^s\big)(x_2)\,\big[
\partial^{\gamma-\alpha}d^{-s}(x_1)
-\partial^{\gamma-\alpha}d^{-s}(x_2) \big] \\
& \qquad +\sum_{\alpha\leq\gamma}\partial^\alpha \big(Qd^s\big)(x_2)\,\big[
\partial^{\gamma-\alpha}d^{-s}(x_1)
-\partial^{\gamma-\alpha}d^{-s}(x_2) \big] \\
& =\ \sum_{\alpha\leq\gamma}\big[
\partial^\alpha \big(u-Qd^s\big)(x_1)
-\partial^\alpha \big(u-Qd^s\big)(x_2) \big]\,\partial^{\gamma-\alpha}d^{-s}(x_1) \\
& \qquad +\sum_{\alpha\leq\gamma}\partial^\alpha \big(u-Qd^s\big)(x_2)\,\big[
\partial^{\gamma-\alpha}d^{-s}(x_1)
-\partial^{\gamma-\alpha}d^{-s}(x_2) \big] \\
& \qquad +\sum_{\alpha\leq\gamma}\big[
\partial^\alpha \big(Qd^s\big)(x_1)\,\partial^{\gamma-\alpha}d^{-s}(x_1)
-\partial^\alpha \big(Qd^s\big)(x_2)\,\partial^{\gamma-\alpha}d^{-s}(x_2) \big] \\
& =\ \sum_{\alpha\leq\gamma}\big[
\partial^\alpha \big(u-Qd^s\big)(x_1)
-\partial^\alpha \big(u-Qd^s\big)(x_2) \big]\,\partial^{\gamma-\alpha}d^{-s}(x_1) \\
& \qquad +\sum_{\alpha\leq\gamma}\partial^\alpha \big(u-Qd^s\big)(x_2)\,\big[
\partial^{\gamma-\alpha}d^{-s}(x_1)
-\partial^{\gamma-\alpha}d^{-s}(x_2) \big] + \partial^\gamma Q(x_1)-\partial^\gamma Q(x_2) 
\end{align*}
where we notice that, as $|\gamma|=\intpart\beta\geq\deg Q$,
\begin{align*}
\partial^\gamma Q(x_1)-\partial^\gamma Q(x_2)=q^{(\gamma)}-q^{(\gamma)}=0.
\end{align*}
Now we have that, by \eqref{calpha}, that
\begin{align*}
\big|\partial^\alpha \big(u-Qd^s\big)(x_1)
-\partial^\alpha \big(u-Qd^s\big)(x_2) \big|& \leq [u-Qd^s]_{C^{\beta+|\alpha|-\intpart\beta}(B_r(x_0))}|x_1-x_2|^{\beta-\intpart\beta} \\ &\leq Cr^{\beta+s-(|\alpha|+\beta-\intpart\beta)}|x_1-x_2|^{\beta-\intpart\beta};
\end{align*}
also, by Lemma \ref{lem:regularized distance},
\begin{align*}
\big|\partial^{\gamma-\alpha}d^{-s}(x_1)\big|\leq Cr^{-s-|\gamma|+|\alpha|},
\end{align*}
so that 
\begin{align*}
\Big|\sum_{\alpha\leq\gamma}\big[
\partial^\alpha \big(u-Qd^s\big)(x_1)
-\partial^\alpha \big(u-Qd^s\big)(x_2) \big]\,\partial^{\gamma-\alpha}d^{-s}(x_1)\Big|
\leq C|x_1-x_2|^{\beta-\intpart\beta}.
\end{align*}
In a similar way we can also estimate the term 
\begin{align*}
\sum_{\alpha\leq\gamma}\partial^\alpha \big(u-Qd^s\big)(x_2)\,\big[
\partial^{\gamma-\alpha}d^{-s}(x_1)
-\partial^{\gamma-\alpha}d^{-s}(x_2) \big]
\end{align*}
and we conclude that
\begin{align*}
\big[d^{-s}u\big]_{C^\beta(\overline B_r(x_0))}\leq C
\end{align*}
with $C$ independent of $r$ and $x_0$.
\end{proof}

The key step towards the proof of Theorem \ref{thm:u1overu2} is the following.

\begin{proposition}\label{u1overu2}
Let $\beta>1$ be such that $\beta\not\in\N$ and $\beta\pm s\not\in\N$.
Let $\Omega\subseteq\Rn$ be a bounded domain of class
$C^{\beta}$, $z\in \partial\Omega$, and $u_1,u_2\in L^\infty(\Rn)$ solutions of 
\begin{align*}
\left\lbrace\begin{aligned}
Lu_i &= f_i & & \hbox{in }\Omega\cap B_1(z) \\
u_i &= 0 & & \hbox{in }B_1(z)\setminus\Omega
\end{aligned}\right.
\end{align*}
with $f_i\in C^{\beta-s}(\overline\Omega),\ i=1,2$.
Suppose that $\partial\Omega\cap B_1(z)$ is the graph of a function with $C^{\beta}$ norm less than~$1$.
Assume that, for some $c_1>0$,
\begin{align}\label{u2-growth}
u_2(x)\geq c_1d^s(x),
\qquad \text{for any }x\in B_1(z).
\end{align}

Then, for any $z\in\partial\Omega\cap B_{1/2}$, there exists a $Q(\cdot,z)\in\mathbf{P}_{\intpart\beta}$ such that
\begin{align}\label{u1minusQu2}
\big|u(x)-Q(x,z)u_2(x)\big| \leq C{|x-z|}^{\beta+s},
\qquad\text{for any }x\in B_1(z),
\end{align}
where $C>0$ depends only on $n,s,\beta,c_1,$ and the $C^{2\beta+1}(\Sn)$ norm of $K$.

Moreover, if $x_0\in\Omega\cap B_1(z)$, $d(x_0)=2r=|z-x_0|>0$,
\begin{align}\label{calpha12}
\big[u_1-Q(\cdot,z)u_2\big]_{C^{\beta+s}(\overline B_r(x_0))}\leq C.
\end{align}
\end{proposition}

\begin{proof}
The argument starts along the same lines of the proof of Proposition \ref{uoverds}. 
Let us set, without loss of generality, $z=0$.

If we write 
\begin{align*}
Q(x)=\sum_{|\alpha|\leq\intpart\beta}\!\! q^{(\alpha)}x^\alpha
=q^{(0)}+\!\!\sum_{1\leq|\alpha|\leq\intpart\beta}\!\! q^{(\alpha)}x^\alpha
=q^{(0)}+Q^{(1)}(x),
\quad Q,Q^{(1)}\in\mathbf{P}_{\intpart\beta},\ |Q_1(x)|\leq C|x|,
\end{align*}
and, in view of\footnote{Proposition \ref{uoverds} justifies \eqref{5555555} when $\beta>1+s$; 
if $\beta<1+s$, then \eqref{5555555}  is covered by \cite{annali}*{Theorem 1.2}.}
Proposition \ref{uoverds},
\begin{align}\label{5555555}
u_2(x)=Q_2(x)\,d^s(x)+v_2(x),
\qquad x\in B_1,\ Q_2\in\mathbf{P}_{\intpart\beta},\ |v_2(x)|\leq C|x|^{\beta-1+s},
\end{align}
then \eqref{u1minusQu2} is equivalent to
\begin{align*}
\big|u_1(x)-q^{(0)}u_2(x)-Q^{(1)}(x)\,Q_2(x)\,d^s(x)-Q^{(1)}(x)\,v_2(x)\big|\leq C|x-z|^{\beta+s}
\end{align*}
so that the claim of the theorem is equivalent to saying that there exists $\widetilde Q\in\mathbf{P}_{\intpart\beta}$ such that
\begin{align*}
\big| u_1(x) - \tilde q^{(0)} \, u_2(x) - \widetilde Q^{(1)}(x) d^s(x) \big| &\leq C|x|^{\beta+s},
\qquad \text{for any }x\in B_1.
\end{align*}
We argue by contradiction: suppose that, for any $i=1,2$ and $j\in\N$, there exists 
$\Omega_j\subseteq\Rn,u_{i,j}\in L^\infty(\Omega_j),f_{i,j}\in C^{\beta-s}(\overline\Omega_j),r_j>0$, and $L_j\in\mathfrak L$
such that
\begin{align*}
\left\lbrace\begin{aligned}
L_ju_{i,j} &= f_{i,j} & & \hbox{in }\Omega_j, \quad \|f_{i,j}\|_{C^{\beta+1-s}(\overline\Omega)}\leq C_0, \\
u_{i,j} &= 0 & & \hbox{in }\Rn\setminus\Omega_j,
\end{aligned}\right.
\end{align*}
and $0\in\partial\Omega_j\in C^{\beta}$; moreover,
\begin{align}\label{u2j-growth}
u_{2,j}(x)\geq c_1d^s(x), \qquad x\in B_1(z),
\end{align}
with $c>0$ independent of $j$, and
\begin{align*}
\sup_{j\in\N}\sup_{r>0}r^{-\beta-s}\Big\|u_{1,j}-q^{(0)}\,u_{2,j}-Q^{(1)}d_j^s\Big\|_{L^\infty(B_r)}=\infty,
\qquad\text{for any }Q\in\textbf{P}_{\intpart{\beta}}.
\end{align*}
Let us consider $Q_{j,r}\in\textbf{P}_{\intpart{\beta}+1}$ as the polynomial obtained
via minimization
\begin{align*}
& \Big\| u_{1,j} - q^{(0)}_{j,r}u_{2,j} - Q^{(1)}_{j,r} d_j^s \Big\|_{L^2(B_r)} \leq \Big\| u_{1,j} - q^{(0)}u_{2,j} - Q^{(1)}d_j^s \Big\|_{L^2(B_r)}
& & \quad\text{for any }Q\in\textbf{P}_{\intpart{\beta}}, \\
& \int_{B_r} \big( u_{1,j} - q^{(0)}_{j,r}u_{2,j} - Q^{(1)}_{j,r} d_j^s \big) \, \big(q^{(0)}u_{2,j} + Q^{(1)}d_j^s\big) = 0
& & \quad\text{for any }Q\in\textbf{P}_{\intpart{\beta}}.
\end{align*}
Define the monotone quantity
\begin{align*}
\theta(r):=\sup_{j\in\N}\sup_{\rho\geq r}\rho^{-\beta-s}\Big\|u_{1,j}-q^{(0)}_{j,r}u_{2,j}-Q^{(1)}_{j,\rho}d_j^s\Big\|_{L^\infty(B_\rho)}.
\end{align*}
Note that $\theta(r)\uparrow\infty$ as $r\downarrow0$, as we prove in Lemma \ref{lem:theta explodes2} below, 
therefore there are sequences 
${(r_m)}_{m\in\N}$ and ${(j_m)}_{m\in\N}$ such that
\begin{align}\label{654+54}
\frac{\big\|u_{1,j_m}-q^{(0)}_{j_m,r_m}u_{2,j_m}-Q^{(1)}_{j_m,r_m}d_{j_m}^s\big\|_{L^\infty(B_{r_m})}}
{r_m^{\beta+s}\,\theta(r_m)} \geq \frac12.
\end{align} 
Define now
\begin{align*}
v_m(x):=\frac{u_{1,j_m}(r_mx)-q^{(0)}_{j_m,r_m}u_{2,j_m}(r_mx)-Q^{(1)}_{j_m,r_m}(r_mx)d_{j_m}^s(r_mx)}{r_m^{\beta+s}\,\theta(r_m)} 
\qquad x\in\Rn,\ m\in\N.
\end{align*} 
and notice that $\|v_m\|_{L^\infty(B_1)}\geq 1/2$ and
\begin{align}\label{1654864-2}
\int_{B_1} v_m(x)\,\big(q^{(0)}u_{2,j_m}(r_mx)+Q^{(1)}(r_mx)\,d_{j_m}^s(r_mx)\big)\;dx=0,
\qquad m\in\N,\ Q\in\mathbf{P}_{\intpart\beta}.
\end{align}
Using Lemma \ref{lem:poly}, we estimate for any $\alpha\in\N^n$ with $|\alpha|\leq\intpart\beta$
\begin{align*}
& r^{|\alpha|+s}\big|q_{j,r}^{(\alpha)}-q_{j,2r}^{(\alpha)}\big|\ \leq \\
& \leq c_\beta \big\|Q_{j,r}^{(1)}d_j^s-Q_{j,2r}^{(1)}d_j^s\big\|_{L^\infty(B_r\cap\{d>r/2\})} \\
& \leq c_\beta \big\|u_{1,j}-q^{(0)}_{j,r}u_{2,j}-Q_{j,r}^{(1)}d_j^s\big\|_{L^\infty(B_r)}
+ c_\beta \big\|u_{1,j}-q^{(0)}_{j,2r}u_{2,j}-Q_{j,2r}^{(1)}d_j^s\big\|_{L^\infty(B_{2r})} \ + \\ 
& \quad + c_\beta \big\|q^{(0)}_{j,r}u_{2,j}-q^{(0)}_{j,2r}u_{2,j}\big\|_{L^\infty(B_r)} \\
& \leq c_\beta\,\theta(r)r^{\beta+s}+c_\beta\,\theta(2r)(2r)^{\beta+1+s}
+ c_\beta \big\|q^{(0)}_{j,r}u_{2,j}-q^{(0)}_{j,2r}u_{2,j}\big\|_{L^\infty(B_r)} \\
& \leq c_\beta\,\theta(r) \, r^{\beta+s}+c_\beta \Big|q^{(0)}_{j,r}-q^{(0)}_{j,2r}\Big|\|u_{2,j}\|_{L^\infty(B_r)} 
\end{align*}
which means
\begin{align}\label{546321}
r^{|\alpha|}\big|q_{j,r}^{(\alpha)}-q_{j,2r}^{(\alpha)}\big|
\leq 2c_\beta\,\theta(r)\,r^{\beta}+c_\beta \big|q^{(0)}_{j,r}-q^{(0)}_{j,2r}\big|,
\qquad 1\leq|\alpha|\leq\intpart\beta.
\end{align}
Also, by the definition of $\theta$, we have
\begin{align*}
\Big\|u_{1,j}-q^{(0)}_{j,r}u_{2,j}-Q^{(1)}_{j,r}d_j^s\Big\|_{L^\infty(B_r)}\leq\theta(r)r^{\beta+s}
\end{align*}
from which we deduce by the triangle inequality that
\begin{align*}
\Big\|\big(q^{(0)}_{j,r}-q^{(0)}_{j,2r}\big)\frac{u_{2,j}}{d_j^s}-Q^{(1)}_{j,r}+Q^{(1)}_{j,2r}\Big\|_{L^\infty(B_r\cap\{d>r/2\})}\leq\theta(r)r^{\beta}.
\end{align*}
Recalling assumption \eqref{u2j-growth} and using Lemma \ref{lem:poly2} we deduce
\begin{align*}
\big|q^{(0)}_{j,r}-q^{(0)}_{j,2r}\big|\leq\theta(r)r^{\beta}
\end{align*}
which yields, thanks to \eqref{546321},
\begin{align*}
\big|q^{(\alpha)}_{j,r}-q^{(\alpha)}_{j,2r}\big|\leq c_\beta\theta(r)r^{\beta-|\alpha|},
\qquad |\alpha|\leq\intpart\beta.
\end{align*}


From the last inequality, in a similar way to what we have done in Proposition \ref{uoverds}, it is now possible to prove that
\begin{align*}
\|v_m\|_{L^\infty(B_R)}\leq c_\beta \, R^{\beta+s},
\end{align*}
while we remark that
\begin{align*}
L_{j_m}v_m(x)=\frac{f_{1,j_m}(r_mx)-q^{(0)}_{j_m,r_m}f_{2,j_m}(r_mx)-L_{j_m}\big(Q^{(1)}_{j_m,r_m}d_{j_m}^s\big)(r_mx)}{r_m^{\beta-s}\,\theta(r_m)} 
\qquad x\in r_m^{-1}\Omega_{j_m},\ m\in\N.
\end{align*}
By the regularity of $f_{1,j_m}$ and $f_{2,j_m}$, and by Lemma \ref{lem:integrated inequalities} 
applied to $L_{j_m}(Q^{(1)}_{j_m,r_m}d_{j_m}^s)$ 
---which satisfies the assumption in view of\footnote{Again, Corollary \ref{cor:Lds}
applies when $\beta>1+s$, otherwise we refer to \cite{annali}*{Proposition 2.3}.} Corollary \ref{cor:Lds}---,
there exists $P_m\in\mathbf{P}_{\intpart{\beta-s}}$ for which
\begin{multline*}
r_m^{s-\beta}\big|f_{1,j_m}(r_mx)-q_{j_m,r_m}^{(0)}f_{2,j_m}(r_mx)-L_{j_m}(Q_{j_m,r_m}^{(1)}d_{j_m}^s)(r_mx)-P_m(r_mx)\big|
\ \leq \\
\leq\ C_0 r_m^{s-\beta} d_{j_m}(r_mx)^{\beta-s-1} \big( d_{j_m}(r_mx) + |r_mx| \big).
\end{multline*}
Denoting $\Omega_m:=r_m^{-1}\Omega_{j_m}$, this gives
\[\qquad \big|L_{j_m} v_m\big| \leq C_0 |x|^{\beta-s}\quad \textrm{in}\ \Omega_m,\qquad \textrm{if}\quad \beta>1+s,\]
while
\[\qquad \big|L_{j_m} v_m\big| \leq C_0|x| d_m^{\beta-s-1}\quad \textrm{in}\ \Omega_m,\qquad \textrm{if}\quad \beta<1+s,\]
where we denoted $d_m(x):={\rm dist}(x, \Omega_m^c)$.
In any case, we get $\big|L_{j_m} v_m\big| \leq C_0 d_m^{\varepsilon-s}$ in $\Omega_m$, $\varepsilon>0$, with $C_0$ independent of $m$, and therefore by Proposition \ref{prop:growth-bound-reg} we get a uniform bound on $\|v_m\|_{C^s(K)}$, for any compact set $K\subset\R^n$.

We now proceed as in the proof of Proposition \ref{uoverds}, and
using an Ascoli-Arzelà argument and the Liouville result in Theorem \ref{liouville-nd}, we conclude that the sequence 
${(v_m)}_{m\in\N}$ is converging in $L^\infty_{loc}(\Rn)$ to
\begin{align*}
v(x)=p(x\cdot e)\,(x\cdot e)_+^s,
\qquad p\in\mathbf{P}_{\intpart\beta},\ e\in\Sn,\ x\in\R^n.
\end{align*}
We underline how this is made possible by the fine estimate 
in Corollary \ref{cor:Lds} which improves of one order the decay at $0$ when $\eta(0)=0$
(and we are applying the corollary with $\eta=Q^{(1)}_{j_m,r_m}$ in our case).
Call now
\begin{align*}
\ell_2=\lim_{m\uparrow\infty}\frac{u_{2,j_m}(r_mx)}{(r_mx\cdot e)_+^s}\ :
\end{align*}
the limit exists by Proposition \ref{uoverds} and is different from zero by \eqref{u2-growth}.
Choose 
\begin{align*}
Q^{(1)}_m(x)=p^{(1)}(x/r_m),
\quad q^{(0)}=\frac{p^{(0)}}{\ell_2},
\end{align*}
and deduce
\[\begin{split}
0 &= \lim_{m\uparrow\infty} r_m^{-s}\int_{B_1} v_m(x)\,\big(q^{(0)}u_{2,j_m}(r_mx)+Q^{(1)}_m(r_mx)\,d_{j_m}^s(r_mx)\big)\;dx\  \\
&=\ \int_{B_1} v(x)\,\big(q^{(0)}\ell_2\,(x\cdot e)_+^s+p^{(1)}(x)\,(x\cdot e)_+^s\big)\;dx
=\int_{B_1} p(x\cdot e)^2\,(x\cdot e)_+^{2s}\;dx,
\end{split}\]
which means that $p\equiv 0$ and then also $v\equiv 0$. 
But this is in contradiction with $\|v\|_{L^\infty(B_1)}\geq 1/2$ ---which follows from \eqref{654+54}---, and thus \eqref{u1minusQu2} is proved.

We now move on to the proof of \eqref{calpha12}. Let
\begin{align*}
v_r(x):=r^{-\beta-s}u_1(x_0+rx)-r^{-\beta-s}Q(x_0+rx,z)u_2(x_0+rx)
\qquad x\in \Rn.
\end{align*}
The first part of the proof is telling us that
\begin{align*}
\|v_r\|_{L^\infty(B_1)}\leq C.
\end{align*}
By Proposition \ref{prop:growth-int-reg} we have that
\begin{align*}
& \hspace{-8mm}\big[u\big]_{C^{\beta+s}(\overline B_{r/2}(x_0))}
=\big[v_r\big]_{C^{\beta+s}(\overline B_{1/2})} \leq C\left(\Big\|\frac{v_r}{1+|x|^{\beta+s}}\Big\|_{L^\infty(\Rn)}+\big[Lv_r\big]_{C^{\beta-s}(\overline B_1)}\right) \\
& \leq C\left(\|u\|_{L^\infty(\Rn)}+\big[Lu_1\big]_{C^{\beta-s}(\overline B_r(x_0))}+\big[L(Q(\cdot,z)u_2)\big]_{C^{\beta-s}(\overline B_r(x_0))}\right) \\
& \leq C\left(\|u\|_{L^\infty(\Rn)}+\big[f_1\big]_{C^{\beta-s}(\overline\Omega)}+\big[f_2\big]_{C^{\beta-s}(\overline\Omega)}+\big[L(Q^{(1)}(\cdot,z)Q_2(\cdot,z)d^s)\big]_{C^{\beta-s}(\overline B_{1/2})}\right)
\end{align*}
which is finite by Corollary \ref{cor:Lds} ---to this end, 
recall that $Q(\cdot,z)=q^{(0)}+Q^{(1)}(\cdot,z)$ with $Q^{(1)}(z,z)=0$. 
\end{proof}

\begin{lemma}\label{lem:theta explodes2}
Let $\beta>0$, $\beta\not\in\N$, $\Omega\subset\Rn$ such that $0\in\partial\Omega$, and $u_1,u_2\in C(\overline{B}_1)$.
If, for any $r\in(0,1)$, $Q_r\in\mathbf{P}_{\intpart\beta}$ satisfies
\begin{align*}
\|u-q^{(0)}_ru_2-Q^{(1)}_rd^s\|_{L^2(B_r)}\leq\|u-q^{(0)}u_2-Q^{(1)}d^s\|_{L^2(B_r)},
\qquad\text{for any }Q\in\mathbf{P}_{\intpart\beta},
\end{align*}
and
\begin{align*}
\frac1{c_0}d^s\leq u_2\leq c_0d^s\quad\text{in }B_1,\qquad\|u-q^{(0)}_ru_2-Q^{(1)}_rd^s\|_{L^\infty(B_r)}\leq c_0r^{\beta+s},
\end{align*}
then there exists $Q_0\in\mathbf{P}_{\intpart\beta}$ such that
\begin{align*}
\|u-q^{(0)}_0u_2-Q^{(1)}_0d^s\|_{L^\infty(B_r)}\leq Cc_0 \, r^{\beta+s},
\qquad r\in(0,1),
\end{align*}
where $C>0$ only depends on $n,s,$ and $\beta$.
\end{lemma}
\begin{proof}
It holds
\begin{align*}
& \big\|Q^{(1)}_rd^s-Q^{(1)}_{2r}d^s\big\|_{L^\infty(B_r)} \ \leq \\
& \leq\ \big\|u-q^{(0)}_r u_2-Q^{(1)}_rd^s\big\|_{L^\infty(B_r)}+\big\|u-q^{(0)}_{2r}u_2-Q^{(1)}_{2r}d^s\big\|_{L^\infty(B_{2r})} 
+\big\|(q^{(0)}_r-q^{(0)}_{2r})u_2\big\|_{L^\infty(B_r)}  \\
& \leq\ (1+2^{\beta+s})c_0r^{\beta+s}+\big|q^{(0)}_r-q^{(0)}_{2r}\big|\,\|u_2\|_{L^\infty(B_r)}.
\end{align*}
In particular,
\begin{align*}
\big|Q^{(1)}_r(x)d^s(x)-Q^{(1)}_{2r}(x)d^s(x)\big|\leq 
(1+2^{\beta+s})c_0r^{\beta+s}+\big|q^{(0)}_r-q^{(0)}_{2r}\big|\,\|u_2\|_{L^\infty(B_r)},
\qquad x\in\partial B_r,
\end{align*}
which yields, by a rescaled version of Lemma \ref{lem:poly},
\begin{align}\label{qr2}
\big|q^{(\alpha)}_r-q^{(\alpha)}_{2r}\big|\leq
Cc_0\,r^{\beta-|\alpha|}+c_0\big|q^{(0)}_r-q^{(0)}_{2r}\big|\,r^{-|\alpha|},
\qquad \alpha\in\N^n,\ 1\leq|\alpha|\leq\intpart\beta.
\end{align}
Also,
\begin{align*}
\Big\|\big(q^{(0)}_r-q^{(0)}_{2r}\big)\frac{u_2}{d^s}+Q^{(1)}_r-Q^{(1)}_{2r}\Big\|_{L^\infty(B_r\cap\{d>r/2\})}
\leq Cc_0\,r^\beta
\end{align*}
Using Lemma \ref{lem:poly2} and the assumptions on $u_2$ we deduce 
\begin{align*}
\big|q^{(0)}_r-q^{(0)}_{2r}\big|\leq Cc_0 \, r^\beta
\end{align*}
which, along with \eqref{qr2}, gives 
\begin{align*}
\big|q^{(\alpha)}_r-q^{(\alpha)}_{2r}\big|\leq
Cc_0\,r^{\beta-|\alpha|},
\qquad \alpha\in\N^n,\ 1\leq|\alpha|\leq\intpart\beta.
\end{align*}
Also, by a similar reasoning,
\begin{align*}
\big\|q^{(0)}_1u_2+Q^{(1)}_1d^s\big\|_{L^\infty(B_1)} \leq c_0+\|u_1\|_{L^\infty(B_1)}
\end{align*}
implies
\begin{align}\label{q12}
\big|q_1^{(\alpha)}\big|\leq C(c_0+\|u\|_{L^\infty(B_1)}),
\qquad \alpha\in\N^n,\ |\alpha|\leq\intpart\beta.
\end{align}
Since $\beta\not\in\N$, then $\beta-|\alpha|\geq\beta-\intpart\beta>0$,
and this, together with \eqref{qr} and \eqref{q12}, yields the existence of limits
\begin{align*}
q_0^{(\alpha)}=\lim_{r\downarrow 0}q_r^{(\alpha)},
\qquad \alpha\in\N^n,\ |\alpha|\leq\intpart\beta.
\end{align*}
Moreover, using a telescopic series and \eqref{qr}, 
\begin{align*}
& \big|q^{(\alpha)}_0-q^{(\alpha)}_{r}\big|\leq\sum_{j=0}^\infty\big|q^{(\alpha)}_{2^{-j}r}-q^{(\alpha)}_{2^{-j-1}r}\big|
\leq Cc_0\sum_{j=0}^\infty (2^{-j}r)^{\beta-|\alpha|}\leq Cc_0\,r^{\beta-|\alpha|}, 
\qquad \alpha\in\N^n,\ |\alpha|\leq\intpart\beta, \\
& \big\|q^{(0)}_0u_2+Q^{(1)}_0d^s-q^{(0)}_ru_2-Q^{(1)}_rd^s\big\|_{L^\infty(B_r)}\leq 
C\sum_{|\alpha|\leq\intpart\beta}\big|q^{(\alpha)}_0-q^{(\alpha)}_{r}\big|r^{|\alpha|+s}\leq 
Cc_0 \, r^{\beta+s}.
\end{align*}
and therefore 
\begin{align*}
& \big\|u_1-q^{(0)}_0u_2-Q^{(1)}_0d^s\big\|_{L^\infty(B_r)} \ \leq \\
& \leq\ \big\|u_1-q^{(0)}_ru_2-Q^{(1)}_rd^s\big\|_{L^\infty(B_r)}+\big\|q^{(0)}_0u_2+Q^{(1)}_0d^s-q^{(0)}_ru_2-Q^{(1)}_rd^s\big\|_{L^\infty(B_r)} \leq Cc_0 \, r^{\beta+s}.
\end{align*}
\end{proof}

We have now all the ingredients to prove Theorem \ref{thm:u1overu2}.

\begin{proof}[Proof of Theorem \ref{thm:u1overu2}]
Let $r>0$ be fixed and $x_0\in\Omega,z\in\partial\Omega$ be such that $d(x_0)=2r=|x_0-z|$.
Consider $x_1,x_2\in B_r(x_0)\subset\Omega$. Then, for $Q=Q(\cdot,z)\in\mathbf{P}_{\intpart\beta}$
as constructed in Proposition \ref{u1overu2},
We closely follow the proof of Theorem \ref{thm:uoverds}.  
Let us first notice that, as therein, it is possible to write
\begin{align*}
\partial^\gamma\big(u_2^{-1}u_1\big)(x_1)-\partial^\gamma\big(u_2^{-1}u_1\big)(x_2)= & 
\sum_{\alpha\leq\gamma}\big[
\partial^\alpha \big(u_1-Qu_2\big)(x_1)
-\partial^\alpha \big(u_1-Qu_2\big)(x_2) \big]\,\partial^{\gamma-\alpha}u_2^{-1}(x_1) \\
& +\sum_{\alpha\leq\gamma}\partial^\alpha \big(u_1-Qu_2\big)(x_2)\,\big[
\partial^{\gamma-\alpha}u_2^{-1}(x_1)
-\partial^{\gamma-\alpha}u_2^{-1}(x_2) \big].
\end{align*}
We first estimate 
\begin{align*}
\big|\partial^{\gamma-\alpha}u_2^{-1}(x_1)\big|\leq Cr^{-s-\intpart\beta+|\alpha|},
\end{align*}
which follows after explicit differentiation and the regularity properties of $u_2$.
By \eqref{calpha12}, we have
\begin{align*}
\big|\partial^\alpha \big(u_1-Qu_2\big)(x_1)
-\partial^\alpha \big(u_1-Qu_2\big)(x_2) \big| \leq
Cr^{\beta+s-(|\alpha|+\beta-\intpart\beta)}|x_1-x_2|^{\beta-\intpart\beta}
\end{align*}
which the implies 
\begin{align*}
\Big|\sum_{\alpha\leq\gamma}\big[
\partial^\alpha \big(u_1-Qu_2\big)(x_1)
-\partial^\alpha \big(u_1-Qu_2\big)(x_2) \big]\,\partial^{\gamma-\alpha}u_2^{-1}(x_1)\Big|
\leq C|x_1-x_2|^{\beta-\intpart\beta}.
\end{align*}
The estimate for 
\begin{align*}
\sum_{\alpha\leq\gamma}\partial^\alpha \big(u_1-Qu_2\big)(x_2)\,\big[
\partial^{\gamma-\alpha}u_2^{-1}(x_1)
-\partial^{\gamma-\alpha}u_2^{-1}(x_2) \big]
\end{align*}
is analogous. So we conclude that
\begin{align*}
\big[u_1/u_2\big]_{C^\beta(\overline B_r(x_0))}\leq C
\end{align*}
with $C$ independent of $r$ and $x_0$.
\end{proof}

\section{Smoothness of free boundaries in obstacle problems}
\label{sec:final}

Using the results from the previous sections, we can now show our main results on the higher regularity of free boundaries for obstacle problems of type \eqref{obstacle}.

\begin{proof}[Proof of Theorem \ref{thm:low-reg}]
Notice first that, by \cite{inventiones}, we have $v\in C^1(\R^n)$.
Let $x_0\in\partial\{v>\varphi\}$ be any regular point. 
By \cite{inventiones}*{Theorem 1.1}, there exists $r>0$ such that $\partial\{v>\varphi\}\cap B_r(x_0) \in  C^{\beta}$ for any~$\beta<1+s$.

Let us define
\begin{align*}
w=v-\varphi,
\end{align*}
which solves
\begin{align}\label{obstacle-v}
\left\lbrace\begin{aligned}
Lw &= f & & \text{in } \{ w > 0 \} \\
w &\geq 0 & & \text{in }\Rn,
\end{aligned}\right.
\end{align}
where $f=-L\varphi\in C^{\theta-s}(\Rn)$. 
Note that $w\in C^1(\R^n)$ so that, 
for any $i\in\{1,\ldots,n\}$, we can differentiate \eqref{obstacle-v} to get
\begin{align}\label{obstacle-v'}
\left\lbrace\begin{aligned}
L\big(\partial_i w\big) &= f_i & & \text{in } \{ w > 0 \}\cap B_r(x_0) \\
\partial_i w &= 0 & & \text{in } B_r(x_0)\setminus \{ w > 0 \}
\end{aligned}\right.
\end{align}
with $f_i:=\partial_i f \in C^{\theta-1-s}(\Rn)$. 
Suppose now, without loss of generality, 
that $e_n$ is normal to $\partial\{v>\varphi\}$ at $x_0$. 
Since at $x_0$ we have \eqref{regular-point} and the free boundary is $C^\beta$ in $B_r(x_0)$, with $\beta>1$, it follows from \cite{annali} that 
\begin{align*}
\partial_n w \geq c_1 d^s
\qquad \text{in } \{ w > 0 \}\cap B_r(x_0)
\end{align*}
for some $c_1>0$. 

We are therefore in the assumptions of Theorem \ref{thm:u1overu2} and, as long as $\beta\leq \theta-1$ (recall that $f_i\in C^{\theta-1-s}$), we deduce that
\begin{align*}
\frac{\partial_i w}{\partial_n w} \in C^\beta\big(\overline{\{w>0\}}\cap B_r(x_0)\big),
\end{align*}
for any~$i\in\{1,\ldots,n-1\}$.

Now, notice that the normal vector $\nu(x)$ to the level set $\{w=t\}$ for $t>0$ and $w(x)=t$ is given by 
\begin{align}\label{nu1}
\nu^i(x)= \frac{\partial_i w}{|\nabla w|}(x) = \frac{\partial_i w/\partial_n w}{\sqrt{\sum_{j=1}^{n-1}(\partial_j w/\partial_n w)^2 +1}},\qquad i=1,...,n.
\end{align}
Therefore, denoting $\Omega=\{w>0\}$ we deduce that in $B_r(x_0)$ we have
\begin{align}\label{nu2}
\partial\Omega\in C^\beta\quad \Longrightarrow \quad \frac{\partial_i w}{\partial_n w} \in C^\beta \quad \Longrightarrow 
\quad \nu\in C^\beta \quad \Longrightarrow \quad \partial\Omega\in C^{\beta+1},
\end{align}
as long as $\beta\leq \theta-1$.

Bootstrapping this argument and recalling that $\Omega=\{v>\varphi\}$, in a finite number of steps we find that $\partial\{v>\varphi\}\cap B_r(x_0)\in C^{\theta}$, as wanted.
\end{proof}

\begin{remark}
The statement of \cite{inventiones}*{Theorem 1.1} requires 
$\varphi\in C^{2,1}(\Rn)$.
Nevertheless, a quick inspection of the proofs therein reveals that this is inessential and that the assumption
on the regularity of the obstacle can be weakened to $\varphi\in C^\gamma$, with $\gamma>\max\{2,1+2s\}$.
In particular, if $\varphi\in C^{\theta+s}$ with $\theta>2$, then \cite{inventiones}*{Theorem 1.1} holds.
\end{remark}

To conclude, we give the proof of the $C^\infty$ regularity.

\begin{proof}[Proof of Theorem \ref{thm:Cinfty}]
It follows immediately from Theorem \ref{thm:low-reg}.
\end{proof}

\appendix

\section{Technical lemmas and tools}

\begin{notation} We define the binomial of two multi-indices $\alpha=(\alpha_1,\dots,\alpha_n)$ and
$\gamma=(\gamma_1,\ldots,\gamma_n)$ (with $\gamma\leq\alpha$ as multi-indices) as
\begin{align*}
\binom{\alpha}{\gamma}=\prod_{i=1}^n\binom{\alpha_i}{\gamma_i}.
\end{align*}
\end{notation}

\begin{lemma}\label{lem:regularized distance} 
Let $U\subseteq\Rn$ be an open bounded domain with $\partial U \in C^\beta$, $\beta>1$, $\beta\not\in\N$.
Then there exists $d\in C^\infty(U)\cap C^\beta(\overline{U})$ 
such that for every $j\in\N,\ j>\beta$, there exists $C=C(n,j,U)>0$ such that
\begin{align}\label{high derivatives of d}
\frac1C\,\dist(\cdot,\partial U)\leq d \leq C\,\dist(\cdot,\partial U),
\qquad
\big| D^j d \big| \leq C_j d^{\beta-j}
\qquad\text{in }U.
\end{align}
\end{lemma}
\begin{proof}
We define $d$ as the only solution of 
\begin{align*}
\left\lbrace\begin{aligned}
-\Delta d &= 1 & & \text{in } U, \\
d &= 0 & & \text{on } \partial U.
\end{aligned}\right.
\end{align*}
The existence and uniqueness of $d$ is classical so let us directly go for \eqref{high derivatives of d}.
Let us consider $\gamma\in\N^n$, $|\gamma|=\intpart\beta$. Then 
\begin{align*}
\left\lbrace\begin{aligned}
-\Delta (\partial^\gamma d) &= 0 & & \text{in }U \\
\partial^\gamma d &\in C^{\beta-\intpart\beta}(\overline{U}).
\end{aligned}\right.
\end{align*}
Let $x_0\in U$ be arbitrary and $r=\dist(x_0,\partial U)/3$.
Using the interior estimates for harmonic functions (\textit{cf}.~\cite{gt}*{Theorem 2.10}) we get,
for any $j\in\N$
\begin{align*}
\sup_{x\in B_r(x_0)}\big|D^j\partial^\gamma d(x)\big|\leq 
\bigg(\frac{nj}{r}\bigg)^j \sup_{x\in B_{2r}(x_0)}\big|\partial^\gamma d(x)-\partial^\gamma d(x_0)\big|\leq
(nj)^j\,[\partial^\gamma d]_{C^{\beta-\intpart\beta}(\overline U)}\,r^{\beta-\intpart\beta-j}.
\end{align*}
\end{proof}

\begin{lemma}\label{lem:diffeo}
Let $U\subseteq\Rn$ be an open bounded domain with $\partial U\in C^\beta$, $\beta>1$, $\beta\not\in\N$, $0\in\partial U$.
Suppose that $-\nu(z)\cdot e_n>1/2$ and every $z\in\partial U\cap B_1$, where $\nu(z)$ denotes the outward unit normal vector to $\partial U$ at $z$. There exists a diffeomorphism $\phi:B_1\cap\{x_n \geq 0\}\to B_1\cap\overline{U}$ such that $\phi(B_1\cap\{x_n=0\})=\partial U$, $\phi(B_1\cap\{x_n>0\})=B_1\cap U$, $\phi\in C^\beta(B_1)\cap C^\infty(B_1\cap\{x_n>0\})$,
and satisfying~\eqref{diffeo}.
\end{lemma}
\begin{proof}
Recall the construction of $d$ given by Lemma \ref{lem:regularized distance}. Define 
\begin{align*}
\Phi: B_1 \cap \overline{U} & \longrightarrow B_1\cap \{x_n \geq 0\} \\
x & \longmapsto \big(x_1,\ldots,x_{n-1}, d(x) \big).
\end{align*}
Then $|\det D\Phi(x)|=|\nabla d(x)\cdot e_n|\neq 0$ for any $x\in B_1\cap U$.
Moreover, the derivatives of $\Phi$ inherit the boundary estimates from $d$ (see \eqref{high derivatives of d}).
Then consider $\phi=\Phi^{-1}$.
\end{proof}

\begin{lemma}\label{lem:incr-quot}
Let $\kappa:\R^n\setminus\{0\}\to\R$ be homogeneous of degree $-\gamma,\gamma\in\R$.
Then there exists $C>0$ such that for any $N\in\N$ and $x,h\in\R^n\setminus\{0\}$
\begin{align*}
\left|\frac1{|h|^N}\Delta_h^N\kappa(x)\right|\leq C\;
\frac{{(|x|+|h|)}^{N(\gamma-1)}}{\prod_{i=0}^N\left|x+\big(\frac{N}2-i\big)h\right|^\gamma}.
\end{align*}
\end{lemma}
\begin{proof}
Write
\begin{align*}
& \frac{\prod_{j=0}^N\left|x+\big(\frac{N}2-j\big)h\right|^\gamma}{|h|^N}\left|\Delta_h^N\kappa(x)\right|
=\frac{\prod_{j=0}^N\left|x+\big(\frac{N}2-j\big)h\right|^\gamma}{|h|^N}\left|\sum_{i=0}^N(-1)^i\binom{N}{i}\;\kappa\bigg(x+\Big(\frac{N}2-i\Big)h\bigg)\right| \\
& =\ \frac{\prod_{j=0}^N\left|x+\big(\frac{N}2-j\big)h\right|^\gamma}{|h|^N}\left|\sum_{i=0}^N(-1)^i\binom{N}{i}\;\kappa\bigg(\lrangle{x+\Big(\frac{N}2-i\Big)h}\bigg)\Big|x+\Big(\frac{N}2-i\Big)h\Big|^{-\gamma}\right| \\
& \leq \ \frac{C}{|h|^N}\sum_{i=0}^N\Big|x+\Big(\frac{N}2-i\Big)h\Big|^{-\gamma} \prod_{j=0}^N\left|x+\Big(\frac{N}2-j\Big)h\right|^\gamma \\
& = \ C\frac{|x|^{\gamma N}}{|h|^N}\sum_{i=0}^N\Big|\lrangle{x}+\Big(\frac{N}2-i\Big)\frac{h}{|x|}\Big|^{-\gamma} \prod_{j=0}^N\left|\lrangle{x}+\Big(\frac{N}2-j\Big)\frac{h}{|x|}\right|^\gamma.
\end{align*}
Assume first that $\gamma\geq 0$. Suppose now that $|x|\leq 2|h|$, then
\begin{align*}
& \frac{\prod_{j=0}^N\left|x+\big(\frac{N}2-j\big)h\right|^\gamma}{|h|^N}\left|\Delta_h^N\kappa(x)\right|
\leq\ C\frac{|x|^{\gamma N}}{|h|^N}\sum_{i=0}^N\Big(1+\Big(\frac{N}2+i\Big)\frac{|h|}{|x|}\Big)^{-\gamma} \prod_{j=0}^N\left(1+\Big(\frac{N}2+j\Big)\frac{|h|}{|x|}\right)^\gamma \\
& \leq\ C\frac{|x|^{\gamma N}}{|h|^N}\bigg(1+\frac{|h|}{|x|}\bigg)^{\gamma N}
\leq C|h|^{(\gamma-1)N}.
\end{align*}
If, instead, $2|h|\leq|x|$ then
\begin{align*}
& \frac{\prod_{j=0}^N\left|x+\big(\frac{N}2-j\big)h\right|^\gamma}{|h|^N}\left|\Delta_h^N\kappa(x)\right|
\leq C\frac{|x|^{\gamma N}}{|h|^N}\sum_{i=0}^N\Big|\lrangle{x}+\Big(\frac{N}2-i\Big)\frac{h}{|x|}\Big|^{-\gamma} \prod_{j=0}^N\left(1+\Big(\frac{N}2+j\Big)\frac12\right)^\gamma \\
& \leq C|x|^{(\gamma-1)N}\frac{|x|^{N}}{|h|^N}\sum_{i=0}^N\Big|\lrangle{x}+\Big(\frac{N}2-i\Big)\frac{h}{|x|}\Big|^{-\gamma} 
\leq C|x|^{(\gamma-1)N} 
\end{align*}
where the last inequality is justified by the regularity of the function $x\mapsto|x|^{-\gamma}$ at points of~$\partial B_1$.

The proof for $\gamma<0$ follows by adapting the technical details of the above computations and we omit them here.
\end{proof}

\begin{lemma}\label{derivative-diagonal}
Let $U\subseteq\Rn$ be open and $f:U\times U\to\R$ be a function of class $C^q$, $q\in\N$. 
Then, for any multi-index $\alpha\in\N^n$, $|\alpha|\leq q$, and $x_0\in U$,
\begin{align*}
\partial_x^\alpha\big|_{x_0} f(x,x)
=\sum_{\gamma\leq\alpha}\binom{\alpha}{\gamma}
\partial_x^\gamma\big|_{x_0}\partial_y^{\alpha-\gamma}\big|_{x_0}f(x,y).
\end{align*}
\end{lemma}
\begin{proof}
If $|\alpha|=1$, then this follows by the chain rule applied to the composition $x\mapsto(x,x)\mapsto f(x,x)$.
If the claim holds for some multi-index $\alpha$, then, for any multi-index $e\in\N^n$ with $|e|=1$,
\begin{align*}
& \partial_x^{e+\alpha}\big|_{x_0} f(x,x)=
\partial_{x_0}^e\sum_{\gamma\leq\alpha}\binom{\alpha}{\gamma}
\partial_x^\gamma\big|_{x_0}\partial_y^{\alpha-\gamma}\big|_{x_0}f(x,y) \\
& =\ \sum_{\gamma\leq\alpha}\binom{\alpha}{\gamma}\left(
\partial_x^{e+\gamma}\big|_{x_0}\partial_y^{\alpha-\gamma}\big|_{x_0}f(x,y)
+\partial_x^{\gamma}\big|_{x_0}\partial_y^{e+\alpha-\gamma}\big|_{x_0}f(x,y) \right) \\
& =\ \sum_{\gamma\leq\alpha+e}\left(\binom{\alpha}{\gamma-e}+\binom{\alpha}{\gamma}\right)
\partial_x^{\gamma}\big|_{x_0}\partial_y^{\alpha-\gamma}\big|_{x_0}f(x,y)
=\sum_{\gamma\leq\alpha+e}\binom{\alpha+e}{\gamma}
\partial_x^{\gamma}\big|_{x_0}\partial_y^{\alpha-\gamma}\big|_{x_0}f(x,y)
\end{align*}
where we have used identity $\binom{\alpha}{\gamma-e}+\binom{\alpha}{\gamma}=\binom{\alpha+e}{\gamma}$.
\end{proof}

\begin{lemma}\label{lem:super chain rule}
Let $U,V\subseteq\Rn$ be open and $g:U\to V$, $f:V\to\Rn$ be two functions of class $C^q$, $q\in\N$.
Then, for any $j\leq q$, there exists constants $\{c_{q,j,k_1,\ldots,k_{q-j}}:k_1+\ldots+k_{q-j}=j\}$ for which it holds
\begin{align*}
D^q(f\circ g)=\sum_{j=0}^{q-1}\sum_{k_1+\ldots+k_{q-j}=j}
c_{q,j,k_1,\ldots,k_{q-j}}
(D^{q-j}f)(g)\left[ D^{1+k_1}g,\ldots,D^{1+k_{q-j}}g \right]
\end{align*}
\end{lemma}
\begin{proof}
For $q=1$ the claim simply follows by the chain rule. For $q=2$ and $q=3$ we respectively have
\begin{align*}
D^2(f\circ g) &= D\big(Df(g)\,Dg\big)=(D^2 f)(g)\,[Dg,Dg]+(Df)(g)\, D^2g, \\
D^3(f\circ g) &= D\big(Df(g)\,Dg\big)=(D^3 f)(g)\,[Dg,Dg,Dg]+3(D^2 f)(g)\,[D^2g,Dg]+(Df)(g)\,D^3g.
\end{align*}
The general formula follows by iteratively applying the chain and the product rules.
\end{proof}

\begin{lemma}\label{lem:prolongation Calpha}
Let $\alpha>0$, $f\in C^{\alpha}([0,1])$, $p$ a polynomial of degree at most $\intpart{\alpha}$, and $r\in(0,1)$.
Suppose that there exists $c>0$ such that
\begin{align}\label{14313}
\big|f(t)-p(t)\big|\leq c_0r^{\alpha}, \qquad \text{for }t\in[0,r].
\end{align}
Then
\begin{align*}
\big|f(t)-p(t)\big|\leq C\left(c_0+\|f\|_{C^\alpha([0,1])}\right) \, t^q \, r^{\alpha-q}, \qquad \text{for }t\in[r,1],\ q\geq\alpha.
\end{align*}
\end{lemma}
\begin{proof}
Let us write
\begin{align*}
f(t)=\sum_{i=0}^{\intpart\alpha}a_it^i+e(t),
\qquad |e(t)|\leq \|f\|_{C^\alpha([0,1])}t^\alpha,
\qquad t\in[0,1].
\end{align*}
Then by \eqref{14313} it follows, by reverse triangle inequality,
\begin{align*}
\left|p(t)-\sum_{i=0}^{\intpart\alpha}a_it^i\right|\leq \left(c_0+\|f\|_{C^\alpha([0,1])}\right)r^\alpha,
\qquad \text{for } t\in[0,r],
\end{align*}
which implies 
\begin{align*}
|p^{(i)}-a_i|\leq C\left(c_0+\|f\|_{C^\alpha([0,1])}\right)r^{\alpha-i}
\qquad i=0,\ldots,\intpart\alpha,
\end{align*}
see Lemma \ref{lem:poly}.
Therefore
\begin{align*}
\big|f(t)-p(t)\big| & \leq \left|p(t)-\sum_{i=0}^{\intpart\alpha}a_i t^i\right|+|e(t)|
\leq C\left(c_0+\|f\|_{C^\alpha([0,1])}\right)\sum_{i=0}^{\intpart\alpha}r^{\alpha-i}t^i 
+\|f\|_{C^\alpha([0,1])}t^{\alpha} \\
& \leq C\left(c_0+\|f\|_{C^\alpha([0,1])}\right)\, t^q \, r^{\alpha-q}
\qquad \text{for } t\in[r,1].
\end{align*}
\end{proof}

\begin{lemma}\label{lem:regularity of incremental quotients}
Let $f\in C^\alpha([0,1])$, $\alpha>1$. Then, for $g$ defined as
\begin{align*}
g(t):=\frac{f(t)-f(0)}{t}, \qquad t\in[0,1],
\end{align*}
it holds
\begin{align*}
\|g\|_{C^{\alpha-1}([0,1])}\leq\|f\|_{C^\alpha([0,1])}.
\end{align*}
\end{lemma}
\begin{proof}
Write
\begin{align*}
g(t)=\int_0^1 f'(\eta t)\;d\eta, \qquad t\in[0,1],
\end{align*}
and for any $k\in\{1,\ldots,\intpart\alpha-1\}$
\begin{align*}
g^{(k)}(t)=\int_0^1 f^{(k+1)}(\eta t)\;d\eta, \qquad t\in[0,1].
\end{align*}
Then $\|g^{(k)}\|_{L^\infty([0,1])}\leq\|f^{(k+1)}\|_{L^\infty([0,1])}$ for any $k\in\{1,\ldots,\intpart\alpha-1\}$ and 
\begin{align*}
\left|\frac{g^{(\intpart\alpha-1)}(t_1)-g^{(\intpart\alpha-1)}(t_2)}{t_1-t_2}\right|
\leq \int_0^1 \left| \frac{f^{(\intpart\alpha)}(\eta t_1)-f^{(\intpart\alpha)}(\eta t_2)}{t_1-t_2} \right| d\eta,
\qquad t_1,t_2\in[0,1].
\end{align*}
\end{proof}

\begin{lemma}\label{scaling of integral}
Let $a,b,r>0$ and $x_0\in B_1$ be such that $B_r(x_0)\subseteq B_1$ and $(x_0)_n>2r$. Then
there exists $c>0$ such that, for every $x\in B_{r/2}(x_0)$, it holds
\begin{align*}
\int_{B_1\setminus B_r(x_0)}(y_n)_+^{s-1}\big|y-x\big|^{-n-a}dy
& \leq cr^{s-1-a}  \qquad \text{and}  \\
\int_{B_r(x_0)}(y_n)_+^{s-1}\big|y-x\big|^{-n+b}dy
& \leq cr^{s-1+b}.
\end{align*}
\end{lemma}
\begin{proof}
We proceed by applying the change of variable $y=r\eta$ in order to deduce
\begin{multline*}
\int_{B_1\setminus B_r(x_0)}(y_n)_+^{s-1}\big|y-x\big|^{-n-a}dy = 
r^{s-1-a}\int_{B_{1/r}\setminus B_1(x_0/r)}(\eta_n)_+^{s-1}\left|\eta-\frac{x}r\right|^{-n-a}d\eta\leq \\
Cr^{s-1-a}\int_{\Rn\setminus B_{1/2}}(\eta_n)_+^{s-1}\left|\eta\right|^{-n-a}d\eta
\end{multline*}
and 
\begin{align*}
\int_{B_r(x_0)}(y_n)_+^{s-1}\big|y-x\big|^{-n+b}dy\leq Cr^{s-1}\int_{B_r(x_0)}\big|y-x\big|^{-n+b}dy\leq Cr^{s-1+b}.
\end{align*}
\end{proof}

\begin{lemma}\label{lem:poly} 
For some $\ell\in\N$, let $Q\in\mathbf{P}_\ell$ with
\begin{align*}
Q(x)=\sum_{|\alpha|\leq\ell}q_{\alpha}\,x^\alpha.
\end{align*} 
Let $U\subseteq\Rn$ be a bounded domain. Then there exists a constant $c=c(\ell,U)>0$ such that
\begin{align}\label{90e832}
\frac1{c} \|Q\|_{L^\infty(U)}\leq\sum_{|\alpha|\leq\ell}|q_{\alpha}|\leq c \|Q\|_{L^\infty(U)}.
\end{align}
\end{lemma}
\begin{proof}  
As both the expressions in \eqref{90e832} are norms on $\mathbf{P}_\ell$,
then the claim follows by the equivalence of all norms in finite dimensional
vector spaces.
\end{proof}

\begin{lemma}\label{lem:poly2}
Let $U\subset\Rn$ be a bounded domain such that $0\in\partial U$.
Let $a\in\R$, $f\in L^\infty(U)$ with $f\geq c_2$ in $U$ for some $c_2>0$, 
$Q\in\mathbf{P}_k$ such that $Q(0)=0$, $\theta,r>0$ such that
\begin{align*}
\|af+Q\|_{L^\infty(\{\dist(\cdot,\partial U)>r\}}\leq\theta.
\end{align*} 
Then 
\begin{align*}
|a|+\|Q\|_{L^\infty(U)}\leq C(1+\|f\|_{L^\infty(U)})\theta,
\end{align*}
with $C=C(n,k,U,c_2)$.
\end{lemma}
\begin{proof}
By the triangle inequality
\begin{align*}
\|Q\|_{L^\infty(\{\dist(\cdot,\partial U)>r\}}\leq\theta+|a|\|f\|_{L^\infty(U)}
\end{align*}
and therefore, by Lemma \ref{lem:poly}, we can also say
\begin{align*}
|Q(x)|\leq C|x|\Big(\theta+|a|\|f\|_{L^\infty(U)}\Big),
\qquad x\in U,\ \dist(x,\partial U)>r.
\end{align*}
In particular, if we pick and fix $x_0\in\{x\in U:\dist(x,\partial U)>r\}\cap B_{2r}$, $r<1$ small,
\begin{align*}
|a|f(x_0)-Cr(\theta+|a|\|f\|_{L^\infty(U)})  \leq |af(x_0)+Q(x_0)|\leq \theta
\end{align*}
and therefore the claimed estimate holds for $|a|$ and, in turn, also for $\|Q\|_{L^\infty(U)}$.
\end{proof}

\begin{lemma}\label{lem:integrated inequalities}
Let $U\subset\Rn$ be an open bounded domain with $0\in \partial U\in C^\beta$, $\beta>1$, $\sigma>0$,
and $d$ defined as in Lemma \ref{lem:regularized distance}. 
Let $f\in C^j(U\cap B_1)\cap C(\overline{U})$, $j\leq\intpart\sigma+1$, such that
\begin{align*}
|D^j f(x)|\leq C|x|\,d(x)^{\sigma-j}
\qquad \text{for any } x\in B_1.
\end{align*}
Then there is $Q\in\mathbf{P}_{j-1}$ such that
\begin{align*}
|f(x)-Q(x)|\leq C|x|\,d(x)^{\sigma}
\qquad \text{for any } x\in B_1.
\end{align*}
\end{lemma}
\begin{proof}
We provide a proof for $j=1$, the general statement follows by iterating this case.
Write
\begin{align*}
f(x)=f(0)+\int_0^1\nabla f(tx)\cdot x\;dt
\end{align*}
and using the assumptions
\begin{multline*}
\big|f(x)-f(0)\big|\leq C|x|^2\int_0^1 t\,d(tx)^{\sigma-1}\;dt
=C\int_0^{|x|} t\,d(t\lrangle{x})^{\sigma-1}\;dt\ \leq \\
\leq C|x|\int_0^{|x|} d(t\lrangle{x})^{\sigma-1}\;dt
\leq C|x| d(x)^{\sigma},
\end{multline*}
where we recall that $\lrangle{x}=x/|x|$ for $x\neq 0$.
\end{proof}

\end{document}